\documentclass[11pt,times,twoside]{article}

\usepackage{geometry}
\usepackage{color}

\usepackage[numbers,sort&compress]{natbib}
\usepackage{graphicx,latexsym,euscript,makeidx,color,bm}
\usepackage{amsmath,amsfonts,amssymb,amsthm,thmtools,mathtools,mathrsfs,enumerate}
\usepackage[colorlinks,linkcolor=blue,anchorcolor=green,citecolor=red]{hyperref}
\def\5n{\negthinspace \negthinspace \negthinspace \negthinspace \negthinspace }
\def\4n{\negthinspace \negthinspace \negthinspace \negthinspace }
\def\3n{\negthinspace \negthinspace \negthinspace }
\def\2n{\negthinspace \negthinspace }
\def\1n{\negthinspace }
\def\dbE{\mathbb{E}}     
\def\dbF{\mathbb{F}} \def\sF{\mathscr{F}}    
         
\def\dbH{\mathbb{H}}

\def\dbN{\mathbb{N}}     
     
\def\dbP{\mathbb{P}}     
\def\dbQ{\mathbb{Q}}     
\def\dbR{\mathbb{R}}

\def\Om{\Omega}            \def\sgn{\mathop{\rm sgn}}

\def\ms{\medskip}

\def\no{\noindent}        \def\q{\quad}                      
    \def\qq{\qquad}                    
    \def\hb{\hbox}                     
                   
         \def\rf{\eqref}                    
  \def\deq{\triangleq}               
            \def\({\Big (}
\def\les{\leqslant}                  \def\){\Big )}
\def\leq{\leqslant}       \def\geq{\geqslant}
\def\ges{\geqslant}       \def\esssup{\mathop{\rm esssup}}   \def\[{\Big[}
           \def\]{\Big]}
                   \def\cd{\cdot}

        \def\ts{\times}                      
\def\a{\alpha}        \def\G{\Gamma}         
            \def\d{\delta}        
               
\def\e{\varepsilon}             
           \def\i{\infty}   \def\k{\kappa}

\theoremstyle{plain}

\theoremstyle{rmk}

\makeatletter

\@addtoreset{equation}{section}
\makeatother
\newtheorem{theorem}{Theorem}[section]
\newtheorem{definition}[theorem]{Definition}
\newtheorem{proposition}[theorem]{Proposition}

\newtheorem{lemma}[theorem]{Lemma}
\newtheorem{remark}[theorem]{Remark}

\newtheorem{assumption}{Assumption}
\makeatletter

\@addtoreset{equation}{section}
\makeatother

\allowdisplaybreaks[4]
\begin{document}
	\title{
		\Large \bf Multi-dimensional anticipated backward stochastic differential equations with  quadratic growth\thanks{
			YH is partially supported by Lebesgue Center of Mathematics ``Investissements d'avenir'' program-ANR-11-LABX-0020-01, by CAESARS-ANR-15-CE05-0024 and by MFG-ANR16-CE40-0015-01.
			Jiaqiang Wen is supported by National Natural Science Foundation of China (Grant No. 12101291), Guangdong Basic and Applied Basic Research Foundation (Grant No. 2025B151502009), and Shenzhen Fundamental Research General Program grant  (Grant No. JCYJ20230807093309021).
		}
	}
	
	\author{
		Ying Hu\thanks{Univ. Rennes, CNRS, IRMAR - UMR 6625, F-35000 Rennes, France
			(Email: {\tt ying.hu@univ-rennes1.fr}).}~,~~~
		Feng Li\thanks{
		%	Corresponding author. 
			Department of Mathematics,
			Southern University of Science and Technology, Shenzhen, Guangdong, 518055, China
			(Email: {\tt 12331011@mail.sustech.edu.cn}).}~,~~~
		Jiaqiang Wen\thanks{ Department of Mathematics and SUSTech International center for Mathematics,
			Southern University of Science and Technology, Shenzhen, Guangdong, 518055, China
			(Email: {\tt wenjq@sustech.edu.cn}).}
	}
	
	%	\author{Ying Hu, Feng Li, Jiaqiang Wen}
	\date{}
	\maketitle
	\no\bf Abstract. \rm
This paper is devoted to the general solvability of anticipated backward stochastic differential equations with   quadratic growth  by relaxing the assumptions made by Hu, Li, and Wen \cite[Journal of Differential Equations, 270 (2021), 1298--1311]{hu2021anticipated} from  the one-dimensional case with bounded terminal values to the multi-dimensional situation with bounded/unbounded terminal values. Three new results regarding the existence and uniqueness of local and global solutions are established. More precisely, 
for the local solution with bounded terminal values, the generator $f(t, Y_t, Z_t, Y_{t+\delta_t},Z_{t+\zeta_t})$ is of general growth with respect to  $Y_t$ and $Y_{t+\d_{t}}$.
For the global solution with bounded terminal values, the generator $f(t, Y_t, Z_t, Y_{t+\delta_t},Z_{t+\zeta_t})$ is of skew sub-quadratic but also ``strictly and diagonally" quadratic growth in $Z_t$.
For the global solution with unbounded terminal values,  the generator $f(t, Y_t, Z_t, Y_{t+\delta_t})$ is of diagonal quadratic growth in $Z_t$ in the first case; and in the second case,  the generator $f(t, Z_t)+\dbE[g(t, Y_t,Z_t, Y_{t+\delta_t},Z_{t+\zeta_t}) ] $ is of diagonal quadratic growth  in $Z_t$ and linear growth in $Z_{t+\zeta_t}$.

	\ms
	
	\no\bf Key words:  
	\rm  
	Anticipated backward stochastic differential equation;
	 quadratic generator;
	multi-dimension;
	unbounded terminal value
	
	\ms
	
	\no\bf AMS subject classifications. \rm 60H10, 60H30.
\section{Introduction}
 Let $\{W_t;t\ges0\}$ be a $d$-dimensional standard Brownian motion defined on some complete probability space $(\Om,\sF,\dbP)$ and  $\dbF \triangleq \{\sF_{t};t \ges 0\}$ be the natural filtration of ${W}$ augmented by all the $\dbP$-null sets in $\sF$.
Fix a terminal time $T\in (0,+\i)$, a constant $K\in [0, +\i)$, and two positive integers $n$ and $d$.
 Consider the following  anticipated backward stochastic differential equations (ABSDEs, for short) over a finite time horizon $[0, T+K]$:
\begin{equation}\label{eq::1.1}\left\{\begin{aligned}
		&Y_t = \xi_{T} + \int_{t}^{T}f(s, Y_s, Z_s, Y_{s+\delta_s}, Z_{s+\zeta_s})ds - \int_{t}^{T} Z_s dW_s,   \quad t\in [0,T]; \\
		&Y_t = \xi_t, \quad   Z_t = \eta_t,   \quad    t\in[T, T+K],
	\end{aligned}\right.\end{equation}
where $\d$ and $\zeta$ are two deterministic $\dbR^{+}$-valued continuous functions defined on $[0, T]$, and $\xi$ and $\eta$ are some given processes defined on $[T,T+K]$.
The process $\xi$ is called the \emph{terminal value} and the function $f$ is called  the {\it generator}  of ABSDE (\ref{eq::1.1}). The unknown processes $(Y, Z)$, taking values in $\dbR^{n} \times \dbR^{n\times d}$, are called a pair of $\dbF$-progressively measurable solutions if they make Eq. (\ref{eq::1.1}) hold almost surely. 

In 2009,  to deal with the delayed stochastic optimal control problem,  Peng and Yang \cite{Peng2009anticipated}  introduced ABSDE \rf{eq::1.1}  and established the existence and uniqueness  under the condition that the generator is of uniformly Lipschitz and linear growth.
From then on,  this theory has attracted a lot of attentions. 
By regarding ABSDE as a new duality type of stochastic differential delay equations,  Chen and Wu \cite{Chen2010delaySDDE} derived the maximum principle for stochastic optimal control problem with delay. Along this line of research, Meng, Shi, Wang, and Zhang \cite{MengShiWangZhang2025} obtained a general maximum principle for a stochastic optimal control problem with delay.
Hamaguchi \cite{Hamaguchi2023} studied a maximum principle for  optimal control problems of   with delay.
Nie, Wang, and Wu \cite{NieWangWu2024} studied a type of linear-quadratic delayed mean-field social optimization, and
de Feo and \'{S}wiech 
\cite{deFeoSwiech2025} 
{studied} a verification theorem and  constructed the optimal feedback controls for  the optimal control problems with the  stochastic delay differential equations.
At the same time, the theory of ABSDE has also been well-developed.  For example,
{Wu, Wang, and Ren \cite{Wu2012nonlipschitz}, along with Lu and Ren \cite{lu2013anticipated}, proved the existence and uniqueness of  ABSDEs with non-Lipschitz generators and linear growth, as well as ABSDEs driven by Markov chains, respectively.}
Yang and Elliott \cite{YangElliott2013} obtained  some properties for generalized ABSDEs.
Hu and Chen \cite{HuChen2016} obtained the $L^p$ solutions of ABSDEs under monotonicity and general increasing conditions.
Cheridito and Nam \cite{Cheridito-Nam-18} studied a class of anticipated BSDEs when the generator is a path-functional of the solution under Lipschitz condition.
Wen and Shi \cite{wen2017delayed} and Douissi, Wen, and Shi \cite{Douissi2019fractional} proved the solvability of ABSDEs driven by a fractional Brownian motion.
Moreover, by using ABSDEs, Menoukeu-Pamen \cite{Menoukeu-Pamen2015} studied a non-linear time-advanced backward stochastic partial differential equations with jumps.
Chen and Yang \cite{ChenYang2024} obtained the well-posedness for anticipated backward stochastic Schrödinger equations.
Some other recent developments of  ABSDEs can be found in 
Chen, Wu and Yu \cite{chen2012delayed},  
Guatteri and Masiero \cite{GM2021MCRF},
Klimsiak and Rzymowski \cite{KR2023SPA},
Han and Li \cite{han2024stochastic}, etc.

%As an important development of BSDEs and duality with stochastic differential delay equations
%(SDDEs), ABSDEs with uniformly Lipschitz have important applications , especially in the stochastic optimal control problem with delay (see, ). 

However,  the framework of Peng and Yang \cite{Peng2009anticipated} 
failed to address some general delayed stochastic optimal control problems, such as
the delayed  stochastic linear-quadratic (LQ, for short) optimal control problems with random coefficients. The reason is that  the related stochastic Riccati equations are essentially    ABSDEs with quadratic growth (see  Chen, Wu, and Yu \cite{chen2012delayed}).
%and  Buckdahn et al. \cite{Buckdahn2024meanfield}
On one hand,  the related  problem of BSDEs with quadratic growth plays  an intrinsic mathematical interest  and Peng \cite[Section 5, page 270]{peng1999open} listed it as an open problem.
On the other hand, this topic has important applications in various fields, such as  nonzero-sum risk-sensitive stochastic differential games, 
financial price-impact models, financial market equilibrium problems for several interacting
agents, and stochastic equilibria problems in incomplete financial markets (see Hu, Liang, and Tang \cite{HuLiangTang2020,HuLiangTang2024}). 
Therefore, it is necessary and valuable to study ABSDEs with quadratic growth.
In general, in multi-dimensional situation,
ABSDE is called of quadratic growth if the generator $f$ is of quadratic growth in $Z$ for  {each $i = 1,..., n$}; and    it is  called of diagonally quadratic growth if $f^i$ is quadratic growth in the $i$-th row of $Z$ and grows at a rate strictly slower than quadratic in the  $j$-th component of $Z$ when  $j \neq i$.
For ABSDE \rf{eq::1.1} with quadratic growth, Fujii and Takahashi \cite{Fujii2018anticipated} study one-dimensional case when the generator $f$ is independent of the anticipated term $Z_{t+\zeta_t}$. Hu, Li, and Wen \cite{hu2021anticipated} studied the  one-dimensional case with bounded terminal value. Based on the work of Hu, Li, and Wen \cite{hu2021anticipated}, Pak, Kim, and Kim \cite{pak2024wellposedness}  studied one-dimensional  case with unbounded terminal value when  the generator $f$ is independent of the anticipated term $Z_{t+\zeta_t}$. 
 
Unfortunately, due to the fact that the comparison property does not hold for the general ABSDE  \rf{eq::1.1} (see the counter
Example 5.3 of Peng and Yang \cite{Peng2009anticipated}), the topic of multi-dimensional ABSDEs, the focus of the present paper, poses a great challenge.
Next, in order to present our work in more detail, let's recall some results about BSDEs.

In ABSDE \rf{eq::1.1}, if the generator $f$ is independent of the anticipated term $Y_{t+\delta_t}$ and  $Z_{t+\zeta_t}$, then it becomes the following classical BSDE:
\begin{equation*}
		Y_t = \xi_{T} + \int_{t}^{T}f(s, Y_s, Z_s)ds - \int_{t}^{T} Z_s dW_s,  \quad t\in [0,T].
	\end{equation*}
For the classical BSDE, Pardoux and Peng \cite{Pardoux1990BSDE} firstly solved the nonlinear situation when the generator has uniformly Lipschitz and linear growth in 1990. Since then, BSDEs have attracted significant attention due to their wide-ranging applications in fields such as stochastic control, partial differential equations, and mathematical finance (see Yong and Zhou \cite{Yong1999control} and Zhang \cite{zhang2017backward}).
Meanwhile, many efforts have been made to relax the assumptions on the generator for the existence and/or uniqueness of adapted solutions. 
For example, Kobylanski \cite{Kobylanski2000PDE}  established the existence and uniqueness  for one-dimensional  BSDEs with quadratic growth and  bounded terminal value. 
Tevzadze \cite{Tevzadze2008quadratic} applied fixed point principle to give a simple proof for the existence of solution. 
Briand and Hu \cite{Briand2006quadratic,Briand and Hu} proved the existence and uniqueness of BSDE with quadratic growth and  unbounded terminal value. 
Note that, compared with one-dimensional BSDEs, some important tools, such as Gisanov's transform and monotone convergence, fail to work with multi-dimentional BSDEs. 
So for the multi-dimensional BSDEs with quadratic growth,  Frei and Dos Reis \cite{Frei} provided a counterexample demonstrating that multidimensional quadratic BSDEs with simple generator and terminal value assumptions may lack global solutions on $[0, T]$. 
 Hu and Tang \cite{Hu-Tang-16} studied multi-dimensional BSDEs with diagonally quadratic generators and bounded terminal value, and Fan, Hu, and Tang \cite{fan2023multi} extended the previous works  by relaxing conditions and providing a novel method for handling unbounded terminal values. 
Some other recent developments of quadratic BSDEs can be found in Bahlali, Eddahbi and Ouknine \cite{Bahali-Eddahbi-Ouknine-17}, Barrieu and El Karoui \cite{Briand-Karoui-13}, Cheridito and Nam \cite{Cheridito-Nam-14,Cheridito-Nam-15},   Richou \cite{Richou-12}, 
Xing and Zitkovic \cite{Xing-Zitkovic},  Madoui, Eddahbi, and  Khelfallah \cite{madoui2022quadratic},  Jiang, Li, and Wei \cite{Juan2024mean-field}, and the references cited therein.

 In this work,  we study the multi-dimensional ABSDEs with diagonally quadratic growth. By introducing new methods, three new results regarding the existence and uniqueness of local and global solutions are established. 
 First, the existence and uniqueness of local solutions with bounded terminal value is proved by virtue of the fixed-point theorem (see \autoref{thm:2.2}), which generalizes both the multi-dimensional result of Fan, Hu, and Tang \cite[Theorem 2.1]{fan2023multi}  and the one-dimensional result of Hu, Li, and Wen \cite[Theorem 4.4]{hu2021anticipated}.
Second, by splicing local solutions, two  existence and uniqueness  results of global solutions  with bounded terminal value are established under different conditions: (i) the $i$-th component of generator $f$ is bounded with respect to $Z_{t+\zeta_{t}}$ and the $j$-th $(j\neq i)$ row of the matrix $Z_t$ (see \autoref{global_bdd}), which generalizes  Hu, Li, and Wen \cite[Theorem 4.8]{hu2021anticipated}; (ii) the $i$-th component of the generator $f$ is of linear growth with respect to $Z_{t+\zeta_{t}}$ and the $j$-th $(j\neq i)$ row of the matrix $Z_t$ (see \autoref{global_bdd2}), where we use a new method concerning John-Nirenberg inequality (different from the ones used in  Fan, Hu, and Tang \cite[Theorem 2.5]{fan2023multi}) to tackle the difficulties.
Third, by using an iterative method, a priori estimate, and $\theta$-method, the existence and uniqueness of global solutions with unbounded terminal value are proved when the generator $f$ is independent of the anticipated term $Z_{t+\zeta_{t}}$ (see \autoref{thm:5.2}),  which extends  Fan, Hu, and Tang \cite[Theorem 2.8]{fan2023multi}.
Finally, when the generator  appears as  
$f(t,Z_t)+ \dbE[g(t, Y_{t}, Z_{t}, Y_{t + \delta_{t}}, Z_{t+\zeta_{t}})]$, where $f$ is of quadratic growth in $Z_t$ and $g$ depends on  the anticipated terms $Y_{t + \delta_{t}}$ and $Z_{t + \zeta_{t}}$, we establish the existence and uniqueness of global solutions with unbounded terminal values (see \autoref{thm:5.3}).
This result relaxes the conditions imposed by Hu, Li, and Wen  \cite[Theorem 5.4]{hu2021anticipated}, where the generator is required to be bounded with respect to $Y_{t+\delta_{t}}$ and $Z_{t+\zeta_{t}}$.
 In contrast, our framework does not necessitate such boundedness, even in the one-dimensional case.

The rest of the paper is organized as follows.  \autoref{section2} introduces key notations and propositions related to BMO martingales. In   \autoref{section3}, we establish the existence and uniqueness of local solutions with bounded terminal values. Morevoer, based on  \autoref{section3},  we study the solvability of global solutions with bounded terminal value in \autoref{section4}. 
In  \autoref{section5}, we prove the existence and uniqueness  of global solutions with unbounded terminal value.  Finally,  \autoref{section6} provides proofs for auxiliary results concerning one-dimensional quadratic BSDEs, addressing both bounded and unbounded terminal value cases.

\section{Preliminaries}\label{section2}
Let $\{W_t;t\ges0\}$ be a $d$-dimensional standard Brownian motion defined on some complete probability space $(\Om,\sF,\dbP)$ and  $\dbF \triangleq \{\sF_{t};t \ges 0\}$ be the natural filtration of ${W}$ augmented by all the $\dbP$-null sets in $\sF$.
The notation $\mathbb{R}^{m \times d}$ denotes the space of $m \times d$-matrix $C$, equipped with the Euclidean norm defined as $|C| = \sqrt{\text{tr}(CC^{\top})}$, where $C^{\top}$ is the transpose of $C$.
Let $a\wedge b$ and $a \vee b$ be the minimum and maximum of two real numbers $a$ and $b$, respectively. Set $a^{+} \triangleq a \vee 0$ and  $a^{-} \triangleq -(a \wedge 0)$. Denote by ${\bf 1}_A$ by the indicator of set $A$, and $\sgn(x) \triangleq {\bf 1}_{\{x>0\}} - {\bf 1}_{\{x\leq 0\}}$.  For any $t\in [0,T]$,  $p \geq 1$, and Euclidean space $\dbH$, we introduce the following spaces:
\begin{align*}
		L^{2}_{\sF_{t}}(\Om; \dbH)
		&= \Big\{\xi:\Om\to\dbH\Bigm|\xi ~\hb{is $\sF_{t}$-measurable and~}  \|\xi \|_2 \triangleq (\dbE[\left | \xi \right |^2] )^\frac{1}{2} \textless \infty \Big\}, \\
		L^\i_{\sF_{t}}(\Om; \dbH)
		&= \Big\{\xi :\Om\to\dbH\Bigm|\xi ~\hb{is $\sF_{t}$-measurable and~}  \|\xi \|_{\i} \triangleq \mathop{\esssup}\limits_{\omega \in \Om}|\xi (\omega)| \textless \infty \Big\} , \\
		\mathcal{H}^p_{\dbF}(t,T;\dbH)
		&= \Big\{Z: \Om\times [t,T] \to\dbH\Bigm| Z~ \hb{is $\dbF$-progressively measurable and~} \\
		& \quad \quad \quad \quad \quad \quad\quad\ \quad \quad \quad   \|Z\|_{\mathcal{H}^p_{\dbF}(t,T)} \triangleq \Big(\dbE\Big[\int_t^T|Z_s|^2ds\Big]^{\frac{p}{2}}\Big)^\frac{1}{p}<\i\Big\},\\
		S^p_{\dbF}(t,T;\dbH)
		&= \Big\{Y:\Om\times [t,T] \to\dbH\Bigm| Y~ \hb{is $\dbF$-progressively measurable, continuous, and~} \\
		& \quad \quad \quad \quad \quad \quad\quad\ \quad \quad \quad   \|Y\|_{S^p_{\dbF}(t,T)} \triangleq \Big(\dbE[\mathop{\sup}\limits_{s \in [t,T]} |Y_{s}|^p]\Big)^\frac{1}{p}<\i \Big\}, \\
		\mathcal{H}^\i_{\dbF}(t,T;\dbH)
		&= \Big\{Y:\Om\times [t,T] \to\dbH\Bigm| Y~ \hb{is $\dbF$-progressively measurable, continuous, and~} \\
		& \quad \quad \quad \quad \quad \quad\quad\ \quad \quad \quad   \|Y\|_{\mathcal{H}^\i_{\dbF}(t,T)} \triangleq \mathop{\esssup}\limits_{(s,\omega) \in [t,T]\times\Om} |Y_{s}(w)| = \Big\|\mathop{\sup}\limits_{s \in [t,T]} |Y_{s}|\Big\|_{\i}<\i \Big\},\\
		 \mathcal{E}(t,T; \dbH)
		&=\Big\{Y : \Om\times [t,T] \to\dbH \Bigm| \exp(|Y|)\in \mathop{\bigcap}\limits_{q \geq 1} S^q_{\dbF}(t,T;\dbH)   \Big\}, \\
		\mathcal{M}(t, T; \dbH)
		&=\Big\{ Z : \Om\times [t,T] \to\dbH\Bigm| Z\in \mathop{\bigcap}\limits_{q \geq 1}	\mathcal{H}^q_{\dbF}(t,T;\dbH)\Big\}.
	\end{align*}
%
%Additionally, we will say that a real process $Y\in \mathcal{E}(t,T; \dbR^{n})$ if the
%random variable $\exp(|Y|)\in \mathop{\bigcap}\limits_{p \geq 1} S^p_{\dbF}(t,T;\dbR^{n})$, and $Z\in \mathcal{M}(t, T; \dbR^{n\times d}),$ if $Z\in \mathop{\bigcap}\limits_{p \geq 1}	\mathcal{H}^p_{\dbF}(t,T;\dbR^{n\times d})$. Define $\mathcal{E}(N)_0 ^T \triangleq \exp\{N_{T}-N_{0} - \frac{1}{2}\langle N\rangle_{T} - \frac{1}{2}\langle N\rangle_{0}\}.$
%
\subsection{BMO martingale}
 
On one hand, we define the following space concerning $BMO$ martingale: for any $t\in \mathscr{T}[0, T]$,
\begin{equation*}
	\begin{aligned}
		\mathcal{Z}^p_{\dbF}(t,T;\dbH)
		= \Big\{Z\in \mathcal{H}^p_{\dbF}(t,T;\dbH)\Bigm| 
		 \|Z\|_{\mathcal{Z}^p_{\dbF}} \triangleq \mathop{\sup}\limits_{\tau \in \mathscr{T}[t,T]} \Big\| \dbE_{\tau}\Big[\Big(\int_\tau^T|Z_s|^2ds\Big)^{\frac{p}{2}}\Big]\Big\|^{\frac{1}{p}}_{\i}   <\i \Big\}, \\
	\end{aligned}
\end{equation*}
where $\mathscr{T}[t,T]$ denotes the set of all $\dbF$-stopping times $\tau$ with values in $[t,T]$, and $\dbE_{\tau}$ is conditional expectation with respect to $\sigma$-field $\sF_{\tau}$. On the other hand, we let $M = (M_t, \sF_t)$ be a uniformly integrable martingale with $M_{0} = 0$ and define 
$$\|M\|_{{BMO_{p}}(\dbP)} \triangleq \mathop{\sup}\limits_{\tau \in \mathscr{T}[0,T]} \Big\| \dbE_{\tau}[|M_{T} - M_{\tau} |^{p}]^{\frac{1}{p}}\Big\|_{\i}.$$
In the sequel, the class $\{M: \|M\|_{{BMO_{p}}(\dbP)} < \i \}$ is denoted by ${BMO_{p}}(\dbP)$, which  is a Banach space equipped with the norm $\|.\|_{BMO_{p}(\dbP)}$. For simplicity, we denote $BMO_p^{[a, b]}(\dbP)$ the space over the time interval $[a, b]$ and $BMO(\dbP)$  the space of ${BMO_{2}}(\dbP)$.
Then,  for $Z \in \mathcal{Z}^2_{\dbF}(0,T;\dbH)$, the process $t \mapsto \int_{0}^{t} Z_s dW_s$ (denoted by $Z\cdot W$)  belongs to $BMO(\dbP)$ and 
$$ \| Z \|_{\mathcal{Z}^2_{\dbF}}\equiv\|Z\cdot W\|_{BMO(\dbP)} . $$
Denote by $\mathcal{E}(M)$ the Dol\'{e}ans-Dade exponential of a continuous local martingale $M$, that is, $\mathcal{E}(M)_t \triangleq \exp\{M_t - M_0  - \frac{1}{2}\langle M \rangle_0^t\}$.
% \tc{${BMO^{[a, b]}_{p}}(\dbP)$ is defined for stochastic processes over the time interval $[a, b]$.}
 
\begin{definition} \sl 
A pair $(Y, Z) \in S^2_{\dbF}(0,T+K;\dbR^{n}) \times \mathcal{H}^2_{\dbF}(0,T+K;\dbR^{n\times d})$ is called an adapted solution of BSDE (\ref{eq::1.1}), if $\dbP$-almost surely, it satisfies (\ref{eq::1.1}). Additionally,  it is called a bounded adapted solution if $(Y, Z) \in \mathcal{H}^\i_{\dbF}(0,T+K;\dbR^{n}) \times \mathcal{Z}^2_{\dbF}(0,T+K;\dbR^{n\times d})$.
\end{definition}  

Next, we present some existing results concerning $BMO$ martingale, which can  be found in Kazamaki \cite{Kazamaki N}  and will be used frequently in the sequel.

\begin{proposition}[John-Nirenberg Inequality]\label{John-Nirenberg inequality}\sl 
 Let $M \in BMO(\dbP)$ with $\|M\|_{BMO(\dbP)} < 1$. Then for every stopping time $\tau \in \mathscr{T}[0,T]$, we have 
$$\dbE_{\tau}[\exp\{\langle M\rangle_{T} - \langle M\rangle_{\tau}\} ] \leq \big\{1 - \|M\|_{BMO(\dbP)}^{2}\big\}^{-1}.$$
\end{proposition}

\begin{proposition}\label{prop:1.3} \sl 
	For every $p\in [1,\i)$, there exists a positive constant $L_p$ such that for all uniformly integrable martingale $M$ with respect to $\tilde{\dbP}$, the following inequality holds true: 
	$$\|M\|_{BMO_{p}(\tilde{\dbP})} \leq L_p \|M\|_{BMO(\tilde{\dbP})},$$
	where 	  
	%= W_t - \langle W, (\Lambda^{i})^{\top}\cdot W\rangle_t 
	$d\tilde{\dbP}\triangleq \mathcal{E}(\Lambda \cd W)_T d\dbP$ with $\Lambda \in BMO(\dbP)$.
\end{proposition}

\begin{proposition}\label{prop:1.4} \sl 
	For $K > 0$, there exist two positive constants $c_1 $ and $c_2 $ depending only on $K$ such that for each $M \in BMO(\dbP)$, we have 
	$$c_1 \|M\|_{BMO(\dbP)} \leq \|\tilde{M}\|_{BMO(\tilde{\dbP})} \leq c_2 \|M\|_{BMO(\dbP)},$$
	where $\tilde{M} \triangleq M - \langle M, N \rangle$  with $N \in BMO(\dbP)$ satisfying $\|N\|_{BMO(\dbP)} \leq K$,
	and $d\tilde{\dbP} \triangleq \mathcal{E}(N)_T d\dbP.$
\end{proposition}
Finally, we present the following (conditional) Burkholder-Davis-Gundy's inequality.
\begin{proposition}\label{BDG}
	For any $p>0$ and $M\in \mathcal{Z}^{p}_{\dbF}(0, T; \dbR^{m})$,  there exist constants $0 < c_p < C_p$, depending only on $p$ and $m$, such that for any $\tau\in \mathscr{T}[0, T]$,
	\begin{align*}
		c_p \dbE^{\tilde{\dbP}}_{\tau} \Big[\Big(\int_{\tau}^{T}|M_s|^2ds\Big)^{\frac{p}{2}}\Big] 
		\leq \dbE^{\tilde{\dbP}}_{\tau}\Big[ \mathop{\sup}\limits_{t \in \mathscr{T} [\tau,T]}\Big|\int_{\tau}^{t}M_sd\tilde{W}_s\Big|^p\Big] 
		\leq C_p \dbE^{\tilde{\dbP}}_{\tau} \Big[\Big(\int_{\tau}^{T}|M_s|^2ds\Big)^{\frac{p}{2}}\Big],
	\end{align*}
	where 	$\tilde{W} \triangleq W - \langle W, \Lambda \cd W\rangle$ with $\Lambda \in BMO(\dbP)$ 
		%= W_t - \langle W, (\Lambda^{i})^{\top}\cdot W\rangle_t 
		and $d\tilde{\dbP}\triangleq \mathcal{E}(\Lambda \cdot W)_T d\dbP$.
	
\end{proposition}
\begin{proof}
By a similar argument as in the proof of Zhang \cite [Theorem 2.4.1]{zhang2017backward}, we deduce that for any $A\in \sF_\tau$,
\begin{align*}
c_p \dbE^{\tilde{\dbP}} \Big[\Big(\int_{\tau}^{T}|M_s {\bf 1}_A|^2ds\Big)^{\frac{p}{2}}\Big] 
\leq \dbE^{\tilde{\dbP}}\Big[ \mathop{\sup}\limits_{t \in \mathscr{T} [\tau,T]}\Big|\int_{\tau}^{t}M_s {\bf 1}_Ad\tilde{W}_s\Big|^p\Big] 
\leq C_p \dbE^{\tilde{\dbP}} \Big[\Big(\int_{\tau}^{T}|M_s {\bf 1}_A|^2ds\Big)^{\frac{p}{2}}\Big].
\end{align*}
Finally, by the definition of conditional expectation, we get the desired result immediately. 
\end{proof}
%---------------------------------------------------
%\begin{proposition}\label{prop:1.3} \sl 
%For every $p\in [1,\i)$, there exists a positive constant $L_p$ such that for all uniformly integrable martingale $M$, the following inequality holds true: 
%$$\|M\|_{BMO_{p}(\dbP)} \leq L_p \|M\|_{BMO(\dbP)}.$$
%\end{proposition}

 \subsection{One-dimensional quadratic BSDEs}
 
The aim of this subsection is to present some preliminary results. In the following,  we always let $\{\theta_{t};t\in[0, T]\}$ be an $\dbF$-progressively measurable scalar-valued non-negative process,  let  $\lambda, \lambda_{0},  \sigma, \sigma_0$, and $\gamma$ be some positive constants, and  $\alpha \in [0, 1)$, and  let $\rho(\cdot), \rho_{0}(\cdot): [0, +\i) \rightarrow [0,+\i)$ be two deterministic non-decreasing continuous functions   with  $\rho(0) = \rho_{0}(0) = 0$.
 Moreover, we always assume that the deterministic continuous functions $\delta$ and $\zeta$ that appear in  BSDE (\ref{eq::1.1}) satisfy the following two conditions:
\begin{itemize}
	\item [$\rm(i)$]  There exists a positive constant $K$ such that
	\begin{align}\label{eqqq:3.2}
		t + \delta_t \leq T+K\quad\hb{and}\quad t + \zeta_t \leq T+K, \q \forall t\in[0,T].
	\end{align}
	
	\item [$\rm(ii)$] There exists a positive constant $L$ such that for all nonnegative and integrable function $h$,
	\begin{align}\label{eq::3.2}
		\int_{t}^{T} h_{s + \delta_s} ds \leq L \int_{t}^{T+K} h_s ds\quad\hb{and}\quad
		\int_{t}^{T} h_{s + \zeta_s} ds \leq L \int_{t}^{T+K} h_s ds, \quad  \forall t \in [0, T].
	\end{align}
\end{itemize}

To establish the solvability of the multi-dimensional BSDE \eqref{eq::1.1} in the following sections, we first present some preliminary results concerning one-dimensional BSDEs with quadratic growth in this subsection. The proofs of these results are all included in the appendix.
For this purpose,  for   $\xi \in L_{\sF _{T}}^{\i}(\Om; \dbR)$, $f(\omega, t, z) : \Om \times [0, T] \ts \dbR^{d}$, and    $f(\cdot, \cdot, z)$ is $\dbF$-progressively measurable,  we consider the following equation
\begin{align}\label{241211}
	Y_t = \xi + \int_{t}^{T}f(s, Z_s)ds - \int_{t}^{T}Z_s dW_s,\q t\in[0,T]. 
\end{align}

\begin{lemma}\label{prop:2.1}  \sl 
	For any given pair $(U, V) \in \mathcal{H}^\i_{\dbF}(0,T+K;\dbR^{n}) \times \mathcal{Z}^2_{\dbF}(0,T+K;\dbR^{n\times d})$,  and  for any $t\in [0,T]$ and  $z, \bar{z} \in \dbR^{ d},$  if the generator  $f$  satisfies the following conditions:
	\begin{align*}
		& |f(t, z)| \leq 
		\theta_{t} + \rho(|U_{t}|) +  n\lambda|V_{t}|^{1+\alpha}+ \rho_{0}(\dbE_t[|U_{t+\delta_t}|]) +\lambda_{0}\big(\dbE_{t}[|V_{t+\zeta_t}|]\big)^{1+\alpha}+ \frac{\gamma}{2}|z|^2,\\
		& |f(t, z) - f(t, \bar{z})| \leq \rho(\|U\|_{\mathcal{H}^\i_{\dbF}(0,T+K;\dbR^n)})(\beta_t + |z| + |\bar{z}|)\times|z - \bar{z}|,
	\end{align*}
	where $\int_{0}^{T} \theta_tdt$ is (essentially) bounded.
	Then  BSDE \rf{241211}
	has a unique adapted solution $(Y, Z)\in \mathcal{H}^\i_{\dbF}(0,T;\dbR) \times \mathcal{Z}^2_{\dbF}(0,T;\dbR^{d}).$ Moreover,  we have the following estimates
	\begin{align}
		&	|Y_t| \leq  \frac{1}{\gamma}\ln2  + \|\xi\|_\i  +  \bigg\|\int_{0}^{T}\theta_{s}ds\bigg\|_{\i}  + C_{\alpha} \gamma^{\frac{1+\alpha}{1-\alpha}}x^{\frac{1+\alpha}{1-\alpha}} (T - t), \label{2.6}\\	
		&	 \dbE_{\tau}\Big[\int_{\tau}^{T}|Z_s|^2 ds\Big]
		\leq \frac{1}{\gamma}\exp\{{2\gamma\|Y\|_{\mathcal{H}^\i_{\dbF}(t,T; \dbR)}}\}\Big\{1+2\bigg\|\int_{0}^{T}\theta_{s}ds\bigg\|_{\i}+  2C_{\alpha}
		x^{\frac{1+\alpha}{1-\alpha}}(T-t)\Big\}\nonumber\\
		&\q\q\q\q\q\q\q\q +\frac{1}{\gamma^2}\exp\{2\gamma \|\xi\|_{\i}\}, \q \forall t\in[0,T],\q \forall \tau \in \mathscr{T}[t,T], \label{eq2.7}
	\end{align}
	where
	\begin{align*}
		&C_{\alpha} \triangleq \frac{1-\alpha}{2}(1+\alpha)^{\frac{1+\alpha}{1-\alpha}} \\ &x\triangleq 2\big[2 + (\rho+\rho_{0})\big(\|U\|_{\mathcal{H}^\i_{\dbF}(t,T+K; \dbR^n)}\big) \big]^2T + 2\big(1+n\lambda+(\lambda_{0} + 1)L\big)^2\|V\|^2_{\mathcal{Z}^2_{\dbF}(t,T+K; \dbR^{n\times d})}.
	\end{align*}
\end{lemma}

\begin{lemma}\label{lemma:6.2} \sl
	For any given pair $(U, V) \in \mathcal{H}^\i_{\dbF}(0,T+K;\dbR^{n}) \times \mathcal{Z}^2_{\dbF}(0,T+K;\dbR^{n\times d})$,   and for any $t\in[0,T]$ and $z\in\dbR^d$,
	if  the generator $f$ satisfies the following condition:
		\begin{align}\label{eq:6.6}
			|f(t, z)| \leq 
			\theta_{t} + \sigma|U_t|+  \lambda |V_t|^{1+\alpha}+ \sigma_0\dbE_t[|U_{t+\delta_t}|] +\lambda_{0}\dbE_{t}[|V_{t+\zeta_t}|] +\frac{\gamma}{2}|z|^2,
		\end{align}
		where $\int_{0}^{T} \theta_tdt$ is (essentially) bounded. Then, for any $t\in[0, T]$, the solution $Y$ of BSDE \rf{241211} has the following estimate
		\begin{align}\label{eq:6.7}
			\exp(\gamma|Y_\tau|) \leq \dbE_{\tau}&\exp\bigg\{\gamma\|\xi\|_{\i} + \gamma \Big\|\int_{0}^{T} \theta_s ds\Big\|_{\i} + (\sigma + \sigma_0)\gamma\|U\|_{\mathcal{H}^\i_{\dbF}(\tau,T+K;\dbR^{n})}(T-\tau)\nonumber \\
			&\qq\q + \lambda\gamma\int_{\tau}^{T}|V_s|^{1+\alpha}ds + \lambda_{0}\gamma \int_{\tau}^{T}\dbE_{s}[|V_{s+\zeta_s}|]ds
			\bigg\},\q \forall \tau \in \mathscr{T}[0,T].
		\end{align}
\end{lemma}		
		
\begin{lemma}\label{lemma:6.2-2} \sl
	For any given pair $(U, V) \in \mathcal{H}^\i_{\dbF}(0,T+K;\dbR^{n}) \times \mathcal{Z}^2_{\dbF}(0,T+K;\dbR^{n\times d})$,   and for any $t\in[0,T]$ and $z\in\dbR^d$,
	if   the generator $f$ satisfies 
		\begin{align}\label{eq:6.8}
			f(t, z) \geq \frac{{\gamma}}{2}|z|^2 - \theta_{t} - \sigma|U_t| - \lambda |V_t|^{1+\alpha} - \sigma_0\dbE_t[|U_{t+\delta_t}|] -\lambda_{0}\dbE_{t}[|V_{t+\zeta_t}|]
		\end{align}
		or
		\begin{align}\label{eq:6.9}
			\q f(t, z) \leq -\frac{{\gamma}}{2}|z|^2 + \theta_{t} + \sigma|U_t| + \lambda |V_t|^{1+\alpha} + \sigma_0\dbE_t[|U_{t+\delta_t}|] +\lambda_{0} \dbE_{t}[|V_{t+\zeta_t}|],
		\end{align}
		where $\int_{0}^{T} \theta_tdt$ is (essentially) bounded. Then, for any $\varepsilon \in (0, \frac{{\gamma}}{9}]$ and $t\in[0, T]$, 
		the solution $Z$ of BSDE \rf{241211} has the following estimate
		\begin{align}\label{eq:6.10}
			&\dbE_\tau\exp \Big\{ \frac{{\gamma}}{2}\varepsilon \int_{\tau}^{T}|Z_s|^2ds\Big\}\nonumber \\
			\leq &\dbE_\tau\exp\bigg\{6\varepsilon \|Y\|_{\mathcal{H}^\i_{\dbF}(\tau, T;\dbR)} +  3\varepsilon\Big\|\int_{0}^{T} \theta_{s}ds\Big\|_{\i} + 3\varepsilon(\sigma+\sigma_0)\|U\|_{\mathcal{H}^\i_{\dbF}(\tau, T+K;\dbR^{n})}(T - t)\\
			&\qq\q\q+ 3\varepsilon\lambda \int_{\tau}^{T}|V_s|^{1+\alpha}ds  + 3\varepsilon\lambda_{0}\int_{\tau}^{T}\dbE_{s}[|V_{s+\zeta_s}|]ds\bigg\},\q \forall \tau \in \mathscr{T}[0,T]\nonumber.
		\end{align}

\end{lemma}

%\begin{remark} \rm 
%\autoref{lemma:6.2} and \autoref{lemma:6.2-2} still holds true when the component of  $U_{t} \in \dbR^{n}$ contains $Y_t \in \dbR$  or the
%row of $V_t \in \dbR^{n\times d}$ contains $Z_t \in \dbR^{n}$, $\forall t\in [0, T]$.
%\end{remark}

%\tc{Combining the proof of \autoref{lemma:6.2} and \autoref{lemma:6.2-2}  in together.}
%\begin{remark}
%	\autoref{lemma:6.2} still holds true when the component of $U_t \in \dbR^{n}$ contains $Y_t \in \dbR$ or the row of $V_t \in \dbR^{n\times d}$ contains $Z_t \in \dbR^{n}$.
%\end{remark} 

Now we introduce the following lemma that establish certain bounds on the (potentially unbounded) solutions of scalar-valued quadratic BSDEs. For a comprehensive proof, the reader can refer to Fan--Hu--Tang \cite[Proposition 1]{Fan S}.
\begin{lemma}\label{lemma:4.1} \sl 
Assume that there exists an  $\dbF$-progressively measurable scalar-valued non-negative process  $\{\bar \zeta_{t};t\in[0, T]\}$ such that for any $t\in[0,T]$ and  $z\in \dbR^{ d}$, 
	\begin{align*}
		|f(t, z)| \leq \bar{\zeta_{t}} + \frac{\gamma}{2}|z|^2 \q 
		 \Big(\hbox{resp.}\q	f(t, z) \leq \bar{\zeta_{t}} + \frac{\gamma}{2}|z|^2\Big).
	\end{align*}
	If the solution $Y$ of BSDE \rf{241211} satisfies the following condition 
	\begin{align*}
		\dbE \exp \Big\{2\gamma\mathop{\sup}\limits_{t\in[0,T]}|Y_{t}| + 2\gamma\int_{0}^{T}\bar{\zeta_{s}}ds \Big\} < +\i \q
		\bigg(\hb{resp.}	\q \dbE \exp \Big\{2\gamma\mathop{\sup}\limits_{t\in[0,T]}|Y^{+}_{t}| + 2\gamma\int_{0}^{T}\bar{\zeta_{s}}ds \Big\}< +\i \bigg),
	\end{align*}
then	we have 
	\begin{align*}
		\exp\{\gamma|Y_t|\} \leq \dbE_{t} \exp\Big\{\gamma|\xi|+\gamma\int_{t}^{T}\bar{\zeta_{s}}ds\Big\} \q
		\bigg(\hb{resp.}	\q	\exp\{\gamma Y^{+}_t\} \leq  \dbE_{t}\exp\Big\{\gamma\xi^{+}+\gamma\int_{t}^{T}\bar{\zeta_{s}}ds\Big\} \bigg).
	\end{align*}
\end{lemma}

\section{Local solution with bounded terminal value}\label{section3}

 In this section, we study local  solutions of the quadratic BSDE (\ref{eq::1.1})  with bounded terminal value.
Before further specifying, we present an  assumption.
%

%Suppose $\theta : [0,T] \times \Om \rightarrow \dbR_{+}$ is an $\dbF$-progressively measurable process, $\psi, \psi_0 : [0, +\i) \rightarrow [0, +\i)$ are two increasing continuous function and $\gamma, \gamma_0, \lambda$ and $\alpha \in [0, 1)$ are positive constants.
\begin{assumption}\label{assumption1} \rm 
	Assume that  $f(\omega, t, y, z, \phi_r, \psi_{\bar{r}}):  \Om\times [0, T]\times \dbR^n \times \dbR^{n\times d}\times L^2_{\sF_r}(\Om; \dbR^n) \times L^2_{\sF_{\bar{r}}}(\Om; \dbR^{n\times d}) \rightarrow L^2_{\sF_t}(\Om;\dbR^n)$ and $f(\cdot, \cdot, y, z,\phi_r, \psi_{\bar{r}})$ is $\dbF$-progressively measurable, where $r, \bar{r} \in [t, T+K]$. In addition, $f(\cdot)$ satisfies the following conditions:
\begin{itemize}
\item [$\rm(i)$] For $i = 1, ..., n$, $f^{i}$ satisfies that $d\dbP \times dt$-a.e., for all $ y\in \dbR^n, z\in \dbR^{n\times d}, \phi\in \mathcal{H}^2_{\dbF}(t,T + K;\dbR^{n}),$ and $\psi\in \mathcal{H}^2_{\dbF}(t,T + K;\dbR^{n\times d})$, 
$$|f^{i}(t, y, z, \phi_{r},\psi_{\bar{r}})| \leq 
\theta_{t} + \rho(|y|) +  \lambda\sum_{j \neq i}|z^{j}|^{1+\alpha}+ \rho_{0}(\dbE_t[|\phi_{r}|]) +\lambda_{0}\big(\dbE_{t}[|\psi_{\bar{r}}|]\big)^{1+\alpha} +  \frac{\gamma}{2}|z^i|^2.$$
\item [$\rm(ii)$] For $i = 1, ..., n$, $f^{i}$ satisfies that $d\dbP \times dt$-a.e., for all  $y, \bar{y} \in \dbR^{n}, z, \bar{z}\in \dbR^{n\times d}, \phi, \bar{\phi}\in \mathcal{H}^2_{\dbF}(t,T + K;\dbR^{n}),$ and $\psi, \bar{\psi}\in \mathcal{H}^2_{\dbF}(t,T + K;\dbR^{n \times d})$,
\begin{align*}
		&|f^{i}(t, y, z, \phi_{r}, \psi_{\bar{r}}) - f^{i}(t, \bar{y}, \bar{z}, \bar{\phi}_{r}, \bar{\psi}_{\bar{r}})|
		\leq \rho\big(|y| \vee |\bar{y}| \vee \dbE_t[|\phi_{r}|] \vee \dbE_t[|\bar{\phi}_{r}|]\big)   \\
		&\q\times \Big[\big(1 + |z| + |\bar{z}| + \dbE_t[|\psi_{\bar{r}}|] + \dbE_t[|\bar{\psi}_{\bar{r}}|]\big)\big(|y - \bar{y}| + |z^{i} - \bar{z}^{i}| + \dbE_t[|\phi_{r} - \bar{\phi}_{r}|]\big)\\
		&\q \ \quad+\big[1 + |z|^{\alpha} + |\bar{z}|^{\alpha} + (\dbE_t[|\psi_{\bar{r}}|])^{\alpha} + (\dbE_t[|\bar{\psi}_{\bar{r}}|])^{\alpha}\big]\big( \sum_{j\neq i}|z^j - \bar{z}^j| + \dbE_t[|\psi_{\bar{r}} - \bar{\psi}_{\bar{r}}|] \big)\Big].
\end{align*}
\item [$\rm(iii)$] There exist some positive constants $M_{1}$, $M_{2}$, and $M_{3}$ such that
\begin{align}\label{121201}
\|\xi\|_{\mathcal{H}^\i_{\dbF}(T,T+K; \dbR^n)}\leq M_1,\quad \|\eta\|_{\mathcal{Z}^2_{\dbF}(T,T+K;\dbR^{n\times d})}\leq M_2, \quad\hb{and}\quad \Big\|\int_{0}^{T}\theta_{t}dt\Big\|_{\i} \leq M_3.
\end{align}
 
\end{itemize}
\end{assumption}

\begin{remark} \rm
In \autoref{assumption1}, there is no essential difference in replacing the terms $\sum_{j \neq i}|z^{j}|^{1+\alpha}$ and $\sum_{j\neq i}|z^j - \bar{z}^j|$ with $|z|^{1+\alpha}$ and $|z - \bar{z}|$, respectively.
\end{remark} 

\begin{remark}  \rm 
	\autoref{assumption1} is weaker than Assumption 4.1 of  Hu--Li--Wen \cite{hu2021anticipated} when $n=1$.
	 Moreover,
for $i= 1, ..., n$, $t \in [0, T]$, $y\in \dbR^{n}$,  $z\in \dbR^{n\times d}$, $\phi\in \mathcal{H}^2_{\dbF}(t,T + K;\dbR^{n})$, $\psi\in \mathcal{H}^2_{\dbF}(t,T + K;\dbR^{n\times d})$, the following generator satisfies \autoref{assumption1}:
\begin{align*}
f^{i}(t, y, z, \phi_{r}, \psi_{\bar{r}}) = (|y|^3 + \cos |z^i|)|z| + |z|^{\frac{3}{2}} + |z^i|^2 + \big(\dbE_t[|\phi_{r}|]\big)^3\sin \big(\dbE_{t}[|\psi_{\bar{r}}|]\big) 
 + \big(\dbE_{t}[|\psi_{\bar{r}}|]\big)^{\frac{4}{3}}.
\end{align*}
\end{remark}

%%%%%%20241128%%%%%%%%%%%%

In order to show the existence and uniqueness of  local  solutions of  BSDE (\ref{eq::1.1})  with bounded terminal value, we introduce the following space:
\begin{align*}
	&\mathcal{A}(0, T+K) 
	\triangleq \Big\{(U, V) \in \mathcal{H}^\i_{\dbF}(0,T+K;\dbR^{n}) \times \mathcal{Z}^2_{\dbF}(0,T+K;\dbR^{n\times d})\\
	&\qq\qq\qq\q\  \Big| U_t = \xi_t,\quad V_t = \eta_t, \quad t \in [T, T+K] \Big\}.
\end{align*}
Note that in $\mathcal{A}(0, T+K)$, the values of $(U, V)$ are determined by $(\xi,\eta)$ when $t\in[T,T+K]$.
Additionally, we define the following set using some positive constants $K_1, K_2$ and $\e$:
\begin{align}\label{eq:2.7}
	\mathcal{B}_{\e}(K_1, K_2) \triangleq \Big\{(Y, Z) \in \mathcal{A}(T-\e, T+K) \Big| \|Y\|_{\mathcal{H}^\i_{\dbF}(T - \e,T+K)} \leq K_1, \  \|Z\|^2_{\mathcal{Z}^2_{\dbF}(T-\e,T+K)} \leq K_2\Big\} 
\end{align}
equipped with the norm 
\begin{align*}
	\|(Y, Z)\|_{\mathcal{B}_{\e}(K_1, K_2)} \triangleq \Big\{\|Y\|^2_{\mathcal{H}^\i_{\dbF}(T - \e,T+K)} + \|Z\|_{\mathcal{Z}^2_{\dbF}(T-\e,T+K)}^2\Big\}^{\frac{1}{2}}.
\end{align*}

The following is the main result of this section, which implies the existence and uniqueness of local solutions to  BSDE (\ref{eq::1.1}). 

\begin{theorem}\label{thm:2.2} \sl 
	Under \autoref{assumption1},  there exists a positive constant $\e$ depending only on $(n, \gamma, \lambda, \lambda_0, L, M_1, M_2, M_3, \alpha, T)$ and functions $\rho(\cdot), \rho_{0}(\cdot)$
	% depending only on $(n, \gamma, \lambda, \lambda_{0}, L, M_1, M_2, M_3, \alpha, T,\rho(\cd),\rho_0(\cd))$  
	such that on the interval $[T-\e, T+K]$, BSDE \rf{eq::1.1}  has  a unique  adapted solution $(Y, Z) \in \mathcal{B}_{\e}(K_1, K_2)$ with 
	\begin{align*}
		K_1 \triangleq 	\frac{2n}{\gamma}\ln2 + 2(n+1)M_1 + 2nM_3 \quad \hb{and} \quad K_2 \triangleq \frac{2n}{\gamma^2} e^{2\gamma M_1} + 4M_2+ \frac{2n(1+2M_{3})}{\gamma}e^{2\gamma K_1} .
	\end{align*}
\end{theorem}

\begin{proof}

The proof will be divided into three steps.
In step one, we will construct a mapping in a suitable Banach space.
In step two, we will show that the constructed mapping is stable in a small ball.
In step three, we will prove that this mapping is a contraction. 

\ms

\noindent {\it Step 1: Construction of the mapping $\Pi$} 

%\tc{For $i = 1, ..., n, A \in \dbR^{n\times d}$ and $z\in \dbR^{ d},$ we denote by $A(z;i)$ the matrix in $\dbR^{n\times d}$ whose $i$th row is $z$ and whose $j$th row is $A^j$ for any $j\neq i$.}

For any given $(U, V) \in \mathcal{A}(0, T+K)$, we consider the following  BSDEs: for each $i =1,...,n$,
\begin{equation}\label{eq:2.141}
	\left\{\begin{aligned}
		&Y_{t}^i = \xi_{T}^{i} + \int_{t}^{T} f^{i}(s, U_s, V_s(Z_{s}^{i}; i), U_{s+\delta_s}, V_{s+\zeta_s})ds - \int_{t}^{T}Z^{i}_{s}dW_s,\quad t\in [0, T]; \\
		&Y^{i}_t = \xi^i_t, \quad    Z^{i}_t = \eta^i_t,   \quad    t\in[T, T+K],
	\end{aligned}\right.\end{equation}
where $V(z; i)$ denotes the matrix obtained from $V$ by replacing its $i$-th row with $z$.
For simplicity of  presentation, we denote  
\begin{align*}
f^{i, U, V}(t, z) = f^{i}(t, U_t, V_t(z; i), U_{t+\delta_t}, V_{t+\zeta_t}), \quad \quad(t, z)\in [0, T]\times \dbR^{ d}.
\end{align*} 
Then, under \autoref{assumption1}, $d\dbP \times dt$-a.e., for any $t\in[0,T]$ and $z, \bar{z}\in \dbR^{ d},$ we have
\begin{align*}
|f^{i, U, V}(t, z)| \leq 
\theta_{t} + \rho(|U_{t}|) +  n\lambda|V_{t}|^{1+\alpha}+ \rho_{0}(\dbE_t[|U_{t+\delta_t}|]) +\lambda_{0}\big(\dbE_{t}[|V_{t+\zeta_t}|]\big)^{1+\alpha}+ \frac{\gamma}{2}|z|^2
\end{align*} 
and 
\begin{align*}
|f^{i, U, V}(t, z) - f^{i, U, V}(t, \bar{z})| \leq \rho(\|U\|_{\mathcal{H}^\i_{\dbF}(0,T+K;\dbR^n)})\big(1 + |z| + |\bar{z}| + 2|V_{t}| + 2\dbE_{t}[|V_{t+\zeta_t}|]\big)|z - \bar{z}|.	
\end{align*}
Hence, according to  \autoref{prop:2.1},   BSDE  (\ref{eq:2.141}) possesses a unique adapted solution $(Y, Z)\in \mathcal{H}^\i_{\dbF}(0,T;\dbR^{n}) \times \mathcal{Z}^2_{\dbF}(0,T;\dbR^{n\times d})$ on the interval $[0,T]$. 
Moreover, on the interval $[0,T+K]$,  we have
 $$\|Y\|_{\mathcal{H}^\i_{\dbF}(0,T+K; \dbR^n)} \leq \|Y\|_{\mathcal{H}^\i_{\dbF}(0,T; \dbR^n)} + \|\xi\|_{\mathcal{H}^\i_{\dbF}(T,T+K; \dbR^n)}\leq \|Y\|_{\mathcal{H}^\i_{\dbF}(0,T; \dbR^n)} + M_1$$ and 
\begin{align}\label{eq:2.15}
		&\|Z\|^{2}_{\mathcal{Z}^2_{\dbF}(0,T + K; \dbR^{n\times d})}
		 = \mathop{\sup}\limits_{\tau \in \mathscr{T}[0,T+K]} \Big\|\dbE_{\tau}\Big[\int_\tau^{T+K}|Z_s|^2ds\Big]\Big\|_{\i}\nonumber\\
		%&\leq \mathop{\sup}\limits_{\tau \in \mathscr{T}[0,T+K]} \Big\{\Big\|\dbE_{\tau}\Big[\int_\tau^{T}|Z_s|^2ds\Big]\Big\|_{\i} +  \Big\|\dbE_{\tau}\Big[\int_{T}^{T+K}|\eta_s|^2ds\Big]\Big\|_{\i}\Big\}\\
		&\q \leq \mathop{\sup}\limits_{\tau \in \mathscr{T}[0,T]} \Big\{\Big\|\dbE_{\tau}\Big[\int_\tau^{T+K}|Z_s|^2ds\Big]\Big\|_{\i}\Big\} +\mathop{\sup}\limits_{\tau \in \mathscr{T}[T,T+K]}\Big\{\Big\|\dbE_{\tau}\Big[\int_\tau^{T+K}|Z_s|^2ds\Big]\Big\|_{\i}\Big\}\nonumber\\
		&\q \leq \mathop{\sup}\limits_{\tau \in \mathscr{T}[0,T]} \Big\{\Big\|\dbE_{\tau}\Big[\int_\tau^{T}|Z_s|^2ds\Big]\Big\|_{\i} +\Big\|\dbE_{\tau}\Big[\int_{T}^{T+K}|\eta_s|^2ds\Big]\Big\|_{\i}\Big\} +M_2\nonumber\\
	&\q	=
		 \mathop{\sup}\limits_{\tau \in \mathscr{T}[0,T]} \Big\{\Big\|\dbE_{\tau}\Big[\int_\tau^{T}|Z_s|^2ds\Big]\Big\|_{\i} +\Big\|\dbE_{\tau}\Big\{\dbE_{T}\Big[\int_{T}^{T+K}|\eta_s|^2ds\Big]\Big\}\Big\|_{\i}\Big\} +M_2\nonumber\\
		&\q 	\leq
	 \mathop{\sup}\limits_{\tau \in \mathscr{T}[0,T]} \bigg\{\Big\|\dbE_{\tau}\Big[\int_\tau^{T}|Z_s|^2ds\Big]\Big\|_{\i} +\Big\|\dbE_{\tau}\mathop{\sup}\limits_{\tilde{\tau} \in \mathscr{T}[T,T+K]}\Big\|\Big\{\dbE_{\tilde{\tau}}\Big[\int_{\tilde{\tau}}^{T+K}|\eta_s|^2ds\Big]\Big\|_{\i}\Big\}\Big\|_{\i}\bigg\} +M_2\nonumber\\
		&\q 	\leq
	\|Z\|^{2}_{\mathcal{Z}^2_{\dbF}(0,T; \dbR^{n\times d})} +2M_2,
\end{align}
%&= \mathop{\sup}\limits_{\tau \in \mathscr{T}[0,T]} \Big\{\Big\|\dbE_{\tau}\Big[\int_\tau^{T}|Z_s|^2ds\Big]\Big\|_{\i}\Big\} +3\mathop{\sup}\limits_{\tau \in \mathscr{T}[T,T+K]}\Big\{\Big\|\dbE_{\tau}[\int_{\tau}^{T+K}|\eta_s|^2ds\Big]\Big\|_{\i}\Big\}\\
%
which implies that $(Y, Z)\in \mathcal{H}^\i_{\dbF}(0,T + K;\dbR^{n}) \times \mathcal{Z}^2_{\dbF}(0,T + K;\dbR^{n\times d})$.
 Therefore, we can define a mapping $\Pi$  as follows:
\begin{align*}
\Pi(U, V) \triangleq (Y, Z), \quad\quad (U, V)\in \mathcal{H}^\i_{\dbF}(0,T+K;\dbR^{n}) \times \mathcal{Z}^2_{\dbF}(0,T+K;\dbR^{n\times d}).
\end{align*}

%%%%20241129%%%%%%%%%

\noindent{{\it  Step 2:  The mapping $\Pi$ is stable on $\mathcal{B}_{\e}(K_1, K_2)$}}

Note that by \autoref{prop:2.1}, for each $ i = 1, ..., n$,   the following estimates hold:
\begin{align*}
		|Y^{i}_t| \leq&\  \frac{1}{\gamma}\ln2 + \|\xi^{i}_{T}\|_\i + \Big\|\int_{0}^{T}\theta_{s}ds\Big\|_{\i} + C_{\alpha}\gamma^{\frac{1+\alpha}{1-\alpha}}x^{\frac{1+\alpha}{1-\alpha}}(T - t),\\
		 \dbE_{\tau}\Big[\int_{\tau}^{T}|Z^{i}_s|^2 ds\Big]\leq&\  \frac{1}{\gamma}\exp\{{2\gamma\|Y^{i}\|_{\mathcal{H}^\i_{\dbF}(0,T; \dbR)}}\}\Big\{1+ \Big\|\int_{0}^{T}\theta_{s}ds\Big\|_{\i} + 2C_{\alpha}
		 x^{\frac{1+\alpha}{1-\alpha}}(T-t)\Big\} \\
		 &\ + \frac{1}{\gamma^2}\exp\{ 2\gamma \|\xi^{i}_{T}\|_{\i}\}, \q \forall t\in [0, T], \q \forall \tau \in \mathscr{T}[t,T],
\end{align*}
	where
\begin{align*}
	&C_{\alpha} \triangleq \frac{1-\alpha}{2}(1+\alpha)^{\frac{1+\alpha}{1-\alpha}} \\ &x\triangleq 2\big[2 + (\rho+\rho_{0})\big(\|U\|_{\mathcal{H}^\i_{\dbF}(t,T+K; \dbR^n)}\big) \big]^2T + 2\big(1+n\lambda+(\lambda_{0} + 1)L\big)^2\|V\|^2_{\mathcal{Z}^2_{\dbF}(t,T+K; \dbR^{n\times d})}.
	\end{align*}
%similarly to (\ref{eq:2.15}), we have$\|Y\|_{\mathcal{H}^\i_{\dbF}(t,T+K; \dbR^n)} \leq \|Y\|_{\mathcal{H}^\i_{\dbF}(t,T; \dbR^n)} + \|\xi\|_{\mathcal{H}^\i_{\dbF}(T,T+K; \dbR^n)}$ and 
%\begin{equation*}
%\begin{aligned}
%\|Z\|^{2}_{\mathcal{Z}^2_{\dbF}(t,T + K; \dbR^{n\times d})}\leq\|Z\|^{2}_{\mathcal{Z}^2_{\dbF}(t,T; \dbR^{n\times d})} +3\|\eta\|^{2}_{\mathcal{Z}^2_{\dbF}(T,T + K; \dbR^{n\times d})}.
%\end{aligned}
%\end{equation*}
Then, we have that
\begin{align*}
&\|Y\|_{\mathcal{H}^\i_{\dbF}(t,T+K; \dbR^n)} \leq \|Y\|_{\mathcal{H}^\i_{\dbF}(t,T; \dbR^n)}+M_1 \leq \frac{n}{\gamma}\ln2 + (n+1)M_1 + nM_{3} + nC_{\alpha} \gamma^{\frac{1+\alpha}{1-\alpha}}x^{\frac{1+\alpha}{1-\alpha}}(T - t)
\end{align*}
and 
\begin{align*}
&\|Z\|^{2}_{\mathcal{Z}^2_{\dbF}(t,T+K; \dbR^{n\times d})}
\leq \|Z\|^{2}_{\mathcal{Z}^2_{\dbF}(t,T; \dbR^{n\times d})} + 2M_2\\
%\leq&\frac{n}{\gamma^{2}}\exp\{2\gamma \|\xi\|_{\mathcal{H}^\i_{\dbF}(T,T+K; \dbR^n)}\} +\frac{n}{\gamma^{2}}\exp\{{2\gamma\|Y\|_{\mathcal{H}^\i_{\dbF}(t,T; \dbR)}}\}\\
%&\quad\quad\quad\quad\quad\quad\quad\quad\quad\quad\cdot\Big\{\frac{1}{2}+\frac{1-\alpha}{2} (1+\alpha)^\frac{1+\alpha}{1-\alpha}(2\gamma)^{\frac{2}{1-\alpha}}\|x\|^{\frac{2(1+\alpha)}{1-\alpha}}_{\mathcal{Z}^2_{\dbF}(t,T)}(T-t)\Big\}\\
&\leq \frac{n}{\gamma^{2}}e^{2\gamma M_1 } +2M_2 +\frac{n}{\gamma}\exp\{{2\gamma\|Y\|_{\mathcal{H}^\i_{\dbF}(t,T; \dbR^{n})}}\}\Big\{1+ 2M_{3} + 2C_{\alpha}x^{\frac{1+\alpha}{1-\alpha}}(T-t)\Big\}.
%& \leq\frac{n}{\gamma^{2}}\exp\{2\gamma \|\xi\|_{\mathcal{H}^\i_{\dbF}(T,T+K; \dbR^n)}\} +\frac{2n}{\gamma^{2}}\exp\{{2\gamma\|Y\|_{\mathcal{H}^\i_{\dbF}(t,T; \dbR)}}\}\cdot\{1+\\
%&(1+\alpha)^{\frac{1+\alpha}{1-\alpha}+1}\gamma^{\frac{2}{1-\alpha}}2^{\frac{2(1+\alpha)}{1-\alpha}-1}\{[C(1 +2\|U\|_{\mathcal{H}^\i_{\dbF}(t,T+K; \dbR^n)})]^{\frac{2(1+\alpha)}{1-\alpha}}+[C(n+L)\|V\|_{\mathcal{Z}^2_{\dbF}(t,T + K; \dbR^{n\times d})}]^{\frac{2(1+\alpha)}{1-\alpha}}\}(T-t)\}.\\
\end{align*}
%and the last inequality is by $(a+b)^{\frac{2(1+\alpha)}{1-\alpha}} \leq 2^{\frac{2(1+\alpha)}{1-\alpha}-1}(a^{\frac{2(1+\alpha)}{1-\alpha}} + b^{\frac{2(1+\alpha)}{1-\alpha}})$
Now, we define
\begin{align*}
K_1 \triangleq 	\frac{2n}{\gamma}\ln2 + 2(n+1)M_1 + 2nM_3 \quad \hb{and} \quad K_2 \triangleq \frac{2n}{\gamma^2} e^{2\gamma M_1} + 4M_2+ \frac{2n(1+2M_{3})}{\gamma}e^{2\gamma K_1} .
\end{align*}
Additionally, let $x_1$ and $x_2$  denote the unique solutions to the following equations
\begin{equation*}
nC_{\alpha} \gamma^{\frac{1+\alpha}{1-\alpha}}\Big[2\Big(2 + (\rho+\rho_{0})(K_1) \Big)^2T + 2\Big(1+n\lambda+(\lambda_{0} + 1)L\Big)^2K_2\Big]^{\frac{1+\alpha}{1-\alpha}} x_1 = \frac{K_1}{2}
\end{equation*}
and
\begin{align*}
C_{\alpha}\Big[2\Big(2 + (\rho+\rho_{0})(K_1) \Big)^2T + 2\Big(1+n\lambda+(\lambda_{0} + 1)L\Big)^2K_2\Big]^{\frac{1+\alpha}{1-\alpha}}x_2 
= \frac{\gamma K_2}{4n}e^{-2\gamma K_1},
\end{align*}
respectively. Therefore,  we have that   for any $\e \in (0, x_0]$ with $x_0 \deq  x_1 \wedge x_2>0$, if
\begin{align*}
\|U\|_{\mathcal{H}^\i_{\dbF}(T-\e,T+K; \dbR^n)}\leq K_1 \quad {\rm and} \quad \|V\|^2_{\mathcal{Z}^2_{\dbF}(T-\e,T+K; \dbR^{n\times d})} \leq K_2,
\end{align*}
then
\begin{align*}
		\|Y\|_{\mathcal{H}^\i_{\dbF}(T-\e,T+K; \dbR^n)}\leq K_1 \quad {\rm and} \quad \|Z\|^2_{\mathcal{Z}^2_{\dbF}(T-\e,T+K; \dbR^{n\times d})} \leq K_2,
\end{align*}
which implies that   $\Pi$ is stable in the closed convex set $\mathcal{B}_{\e}(K_1, K_2)$ for each $\e \in (0, x_0]$, i.e., 
\begin{align*}
\Pi(U, V) \in \mathcal{B}_{\e}(K_1, K_2), \quad\ \forall(U,V)\in \mathcal{B}_{\e}(K_1, K_2).
\end{align*}

\noindent{\it Step 3:   The mapping $\Pi$ is  a contraction   on $\mathcal{B}_{\e}(K_1, K_2)$}

Denote that for any given $\e\in (0, x_0]$, $(U, V), (\tilde{U},\tilde{V})\in \mathcal{B}_{\e}(K_1, K_2),$ 
\begin{align*}
(Y, Z) &\triangleq \Pi(U,V), \quad (\tilde{Y}, \tilde{Z}) \triangleq \Pi(\tilde{U}, \tilde{V}), \q	\Delta Y \triangleq Y- \tilde{Y}, \\
 \Delta Z &\triangleq Z - \tilde{Z},
 \quad \Delta U\triangleq U - \tilde{U}, \quad \Delta V \triangleq V- \tilde{V}.
\end{align*}
Then,  for  $i = 1, ..., n$ and $\tau \in \mathscr{T}[T-\e, T]$, we have
\begin{equation}\left\{\begin{aligned}
		&Y^{i}_{\tau} = \xi^{i}_{T} + \int_{\tau}^{T}f^{i}(s, U_s, V_s(Z_{s}^{i}; i), U_{s+\delta_s}, V_{s+\zeta_s})ds - \int_{\tau}^{T}Z^{i}_s dW_s, \\
		&\tilde{Y}_{\tau}^{i} = \xi^{i}_{T} + \int_{\tau}^{T}f^{i}(s, \tilde{U}_s, \tilde{V}_s(\tilde{Z}_{s}^{i}; i), \tilde{U}_{s+\delta_s}, \tilde{V}_{s+\zeta_s})ds - \int_{\tau}^{T}\tilde{Z}^{i}_s dW_s.
	\end{aligned}\right.\end{equation}
Additionally, note that for all $\tau\in \mathscr{T}[T, T+K],$ 
\begin{align*}
Y_\tau = \tilde{Y}_\tau = U_\tau = \tilde{U}_\tau = \xi_{\tau}\q\hbox{and} \quad Z_\tau = \tilde{Z}_\tau = V_\tau = \tilde{V}_\tau = \eta_\tau.
\end{align*}
Then,  we have
\begin{align}\label{eq:2.31}
\Delta Y_{\tau}^{i} = \int_{\tau}^{T} (N_{s}^{1,i} + N_{s}^{2,i}) ds - \int_{\tau}^{T}\Delta Z_{s}^{i}dW_s, \q \tau\in \mathscr{T}[T-\e,T],
\end{align}
where 
\begin{align*}
N_{s}^{1,i} =&\ f^{i}(s, U_s, V_s(Z_{s}^{i}; i), U_{s+\delta_s}, V_{s+\zeta_s}) - f^{i}(s, U_s, V_s(\tilde{Z}_{s}^{i}; i), U_{s+\delta_s}, V_{s+\zeta_s}), \\
	N_{s}^{2,i} = &\ f^{i}(s, U_s, V_s(\tilde{Z}_{s}^{i}; i), U_{s+\delta_s}, V_{s+\zeta_s}) - f^{i}(s, \tilde{U}_s, \tilde{V}_s(\tilde{Z}_{s}^{i}; i), \tilde{U}_{s+\delta_s}, \tilde{V}_{s+\zeta_s}).
\end{align*}
Notice that for $i = 1, ..., n,$ the term $N_{s}^{1,i}$ can be written as $N_{s}^{1,i} = (Z_{s}^{i} - \tilde{Z}_{s}^{i})\Lambda_{s}^{i}$, where
\begin{equation*}\Lambda_{s}^{i}=\left\{\begin{aligned}
		&\frac{(Z_{s}^{i}-\tilde{Z}_{s}^{i})^{\top}}{|Z_{s}^{i}-\tilde{Z}_{s}^{i}|^2}\Big(f^{i}(s, U_s, V_s(Z_{s}^{i}; i), U_{s+\delta_s}, V_{s+\zeta_s}) \\
		&\quad\quad\quad\quad\quad\quad\ - f^{i}(s, U_s, V_s(\tilde{Z}_{s}^{i}; i), U_{s+\delta_s}, V_{s+\zeta_s})\Big),\quad \hb{if}\ Z_{s}^{i} \neq \tilde{Z}_{s}^{i}; \\
		&0, \quad\quad\quad\quad\quad\quad\quad\quad\quad\quad\quad\quad\quad\quad\quad\quad\quad\quad\quad\q\q\q \hb{if}\ Z_{s}^{i} = \tilde{Z}_{s}^{i}.\\
	\end{aligned}\right.\end{equation*}
From \autoref{assumption1}, one has
\begin{align}\label{eq:2.34}
|\Lambda_{s}^{i}| &\leq \rho(|U_{s}| \vee \dbE_{s}[|U_{s+\delta_s}|])\big(1 + 2|V_s|+|Z_s| + |\tilde{Z}_s| + 2\dbE_{s}[|V_{s+\zeta_{s}}|]\big), \q \forall s\in [t, T]\q \forall t\in [0, T].
\end{align}
Then, for $i = 1, ..., n$, we define
\begin{align}\label{eq:2.35}
\tilde{W}_{t}^{i} \triangleq W_t - \int_{0}^{t}\Lambda_{s}^{i}ds 
%= W_t - \langle W, (\Lambda^{i})^{\top}\cdot W\rangle_t 
\quad \hb{and} \quad d\dbQ^{i}\triangleq \mathcal{E}(P^{i})_T d\dbP,
\end{align}
where 
\begin{align*}
P_{t}^{i} \triangleq \int_{0}^{t}(\Lambda^{i}_{s})^{\top}dW_s, \q \hb{and} \q
\mathcal{E}(P^{i})_{t} \triangleq \exp\Big\{\int_{0}^{t} (\Lambda^{i}_{s})^{\top}dW_{s} - \frac{1}{2}\int_{0}^{t}|\Lambda^{i}_{s}|^2ds\Big\}.
\end{align*}
On one hand, by the similar deduction of \eqref{eq::3.6} and the definition of processes $(U, V)$, $(\tilde{U},\tilde{V})$, $(U, V)$, and  $(\tilde{U},\tilde{V})$, 
it is easy to verify that for each $i = 1, 2, ..., n$, the process $P_{t}^{i}$, $t\in [0, T]$, is a scalar BMO martingale under $\dbP$, which implies that $\mathcal{E}(P^{i})_{t}$
is a uniform integrable martingale (see  Kazamaki \cite[Theorem 2.3]{Kazamaki N}). Thus $\dbE^{\dbP}[\mathcal{E}(P^{i})_{T}] = 1$. Therefore, we can check that $\dbQ^{i}$ is a probability measure. 
On the other hand, by Theorem 5.22 and item (c) on page 136 of  Le Gall \cite{LeGall},
the process $\tilde{W}^{i}$ is a Brownian motion under $\dbQ^i.$ 
Note that
\begin{align*}
G_{t}^{i} \triangleq \int_{0}^{t}\Delta Z_{s}^{i}dW_{s} \q \hb{and} \q 
\tilde{G}_{t}^{i} \triangleq \int_{0}^{t}\Delta Z_{s}^{i}d\tilde{W}^{i}_{s}, \q t\in [0, T]
\end{align*}
are two BMO martingale under the two probability measures $\dbP$ and $\dbQ^i$, respectively.
Therefore,  BSDE (\ref{eq:2.31}) can be rewritten as follows:
\begin{align*}
	\Delta Y_{\tau}^{i} + \int_{\tau}^{T}\Delta Z_{s}^{i}d\tilde{W}^{i}_s= \int_{\tau}^{T} N_{s}^{2,i} ds,\q \tau\in \mathscr{T}[T-\e,T].
\end{align*}
Taking the square and the conditional expectation with respect to $\dbQ^i$ on both sides of the above equation, we deduce that
\begin{align}\label{eq:2.37}
|\Delta Y_{\tau}^{i}|^2 + \dbE_{\tau}^{\dbQ^{i}} \int_{\tau}^{T}|\Delta Z_{s}^{i}|^2ds = \dbE_{\tau}^{\dbQ^{i}}
\Big\{\int_{\tau}^{T} N_{s}^{2,i} ds\Big\}^2,\q \tau\in \mathscr{T}[T-\e,T].
\end{align}
Note that, from (\ref{eq:2.34}),  for each $\e \in (0, x_0]$, 
\begin{align*}
\dbE_{\tau}\int_{\tau}^{T}|\Lambda_{s}^{i}|^2ds
&\leq\ 4\rho(K_1)^2\Big\{x_0 + 4\|V\|^2_{\mathcal{Z}^2_{\dbF}(T-\e, T;\dbR^{n\times d})} + \|Z\|^2_{\mathcal{Z}^2_{\dbF}(T-\e, T;\dbR^{n\times d})} + \|\tilde{Z}\|^2_{\mathcal{Z}^2_{\dbF}(T-\e, T;\dbR^{n\times d})} \\
&\q\q+ 4L^2\|V\|^2_{\mathcal{Z}^2_{\dbF}(T-\e, T+K;\dbR^{n\times d})} \Big\}
 \leq\ 4\rho(K_1)^2 (x_0 + 10K_{2})\triangleq \bar{K}, \q \tau \in \mathscr{T}[T-\e,T].
\end{align*}
Thus $\|\Lambda^{i}\cdot W\|_{BMO (\dbP)} \leq \bar{K}$.
In view of \autoref{prop:1.4}, there exist positive constants $c_1$ and $c_2$ depending only on $\bar{K}$ such that for each $X \in \mathcal{Z}^2_{\dbF}(T-\e, T; \dbR^{n\times d})$, one has
\begin{align}\label{eq:::3.27}
	   & \mathop{\sup}\limits_{\tau \in \mathscr{T}[T-\e,T]}\Big\|\dbE_{\tau}^{\dbQ^{i}} \int_{\tau}^{T}|X_s|^2ds \Big\|^{\frac{1}{2}}_{\i}
		%&=\mathop{\sup}\limits_{\tau \in \mathscr{T}[t,T]}\Big\|\dbE_{\tau}^{\dbQ^{i}} \Big[\Big|\int_{\tau}^{T}X_sd\tilde{W}_{s}^{i}\Big|^2\Big]\Big\|_{\i}\\
		=\Big\|X\cdot \tilde{W}^i\Big\|_{BMO(\dbQ^i)}\nonumber \\
		%= \mathop{\sup}\limits_{\tau \in \mathscr{T}[t,T]}\Big\|\dbE_{\tau}^{\dbQ^{i}} \Big[\Big|\int_{\tau}^{T}X_sdW_{s} - \int_{\tau}^{T}X_s \Lambda_{s}^{i}ds\Big|^2\Big]\Big\|_{\i}
		&= \Big\|X\cdot W - \langle X\cdot W, \Lambda^{i}\cdot W\rangle\Big\|_{BMO(\dbQ^i)} \geq c_1 \|X\cdot W\|_{BMO(\dbP)}
\end{align}
and for any $\tau\in \mathscr{T}[T-\e, T]$
\begin{align}\label{eq::3.27}
\dbE_{\tau}^{\dbQ^{i}} \int_{\tau}^{T}|X_s|^2ds 
&\leq \Big\|X\cdot \tilde{W}^i\Big\|^{2}_{BMO(\dbQ^i)}
%&=\mathop{\sup}\limits_{\tau \in \mathscr{T}[t,T]}\Big\|\dbE_{\tau}^{\dbQ^{i}} \Big[\Big|\int_{\tau}^{T}X_sd\tilde{W}_{s}^{i}\Big|^2\Big]\Big\|^2_{\i}\\
%&= \mathop{\sup}\limits_{\tau \in \mathscr{T}[t,T]}\Big\|\dbE_{\tau}^{\dbQ^{i}} \Big[\Big|\int_{\tau}^{T}X_sdW_{s} - \int_{\tau}^{T}X_s \Lambda_{s}^{i}ds\Big|^2\Big]\Big\|^2_{\i}\\
= \Big\|X\cdot W - \langle X\cdot W, \Lambda^{i}\cdot W\rangle\Big\|^{2}_{BMO(\dbQ^i)} \nonumber\\
&\leq c^2_2 \|X\cdot W\|^2_{BMO(\dbP)}.
\end{align}
Moreover, by virtue of \autoref{prop:1.3}, there is a generic positive constant $L_4$ such that for each $X\cdot W \in BMO(\dbP)$, one has
\begin{align}\label{eq::3.28}
	\|X\cdot \tilde{W}^i\|^2_{BMO_{4}(\dbQ^{i})} \leq L^2_{4}\|X\cdot \tilde{W}^i\|^2_{BMO(\dbQ^i)}.
\end{align}
Then, combining   \autoref{assumption1}, the inequalities $(x^2+y^2)^{\frac{\alpha}{2}} \leq x^{\alpha} + y^{\alpha}$, and 
\begin{align*}
\sum_{j \neq i}|V_{s}^{j} - \tilde{V}_{s}^{j}| = \sqrt{\Big(\sum_{j \neq i}|V_{s}^{j} - \tilde{V}_{s}^{j}|\Big)^2}\leq \sqrt{n\sum_{i=1}^{n}|\Delta V_{s}^{i}|^2} = \sqrt{n}|\Delta V_{s}|,
\end{align*}
we deduce that for any $\tau \in \mathscr{T}[T-\e, T]$ and $s\in [\tau, T]$,
\begin{align*}
|N_{s}^{2,i}| &\leq \rho(K_1) \bigg[\Big(1 +|V_{s}| + |\tilde{V}_{s}| + 2|\tilde{Z}_{s}| + \dbE_{s}[|V_{s+\zeta_s}|] + \dbE_{s}[|\tilde{V}_{s+\zeta_s}|]\Big)2\|\Delta U\|_{\mathcal{H}^\i_{\dbF}(t,T+K; \dbR^n)}+\Big(1 +\\
& \q|V_{s}|^{\alpha}+ |\tilde{V}_{s}|^{\alpha} + 2|\tilde{Z}_{s}|^{\alpha} + \big(\dbE_{s}[|V_{s+\zeta_s}|]\big)^{\alpha} + \big(\dbE_{s}[|\tilde{V}_{s+\zeta_s}|]\big)^{\alpha}\Big)\Big(\sqrt{n}|\Delta V_{s}| + \dbE_{s}[|\Delta V_{s + \zeta_s}|]\Big)\bigg].
\end{align*}
Thus, from equation (\ref{eq:2.37}) and  H\"{o}lder's inequality, we derive that for any $\tau \in \mathscr{T}[T-\e, T]$, 
\begin{align}\label{eq:2.39}
&|\Delta Y_{\tau}^{i}|^2 + \dbE_{\tau}^{\dbQ^{i}}\int_{\tau}^{T}|\Delta Z_{s}^{i}|^2ds \nonumber\\
\leq&\ 12[\rho(K_1)]^2\e\|\Delta U\|^2_{\mathcal{H}^\i_{\dbF}(T-\e,T+K)}\dbE_{\tau}^{\dbQ^{i}}\int_{\tau}^{T}\Big(1 +|V_{s}| + |\tilde{V}_{s}| + 2|\tilde{Z}_{s}| + \dbE_{s}[|V_{s+\zeta_s}|] + \dbE_{s}[|\tilde{V}_{s+\zeta_s}|]\Big)^2ds\nonumber \\
&\ + 3n[\rho(K_1)]^2\dbE_{\tau}^{\dbQ^{i}}\Big\{\int_{\tau}^{T} \Big(1 +|V_{s}|^{\alpha} + |\tilde{V}_{s}|^{\alpha} + 2|\tilde{Z}_{s}|^{\alpha} + \big(\dbE_{s}[|V_{s+\zeta_s}|]\big)^{\alpha} + \big(\dbE_{s}[|\tilde{V}_{s+\zeta_s}|]\big)^{\alpha}\Big) |\Delta V_s| ds\Big\}^2\nonumber \\
&\ + 3[\rho(K_1)]^2\dbE_{\tau}^{\dbQ^{i}}\Big\{\int_{\tau}^{T} \Big(1 +|V_{s}|^{\alpha} + |\tilde{V}_{s}|^{\alpha} + 2|\tilde{Z}_{s}|^{\alpha} + (\dbE_{s}[|V_{s+\zeta_s}|])^{\alpha} \nonumber\\
&\ \q\q\q\q\qq\ \qq\qq+\big(\dbE_{s}[|\tilde{V}_{s+\zeta_s}|]\big)^{\alpha}\Big) \dbE_{s}[|\Delta V_{s + \zeta_s}|] ds\Big\}^2 \nonumber \\
\deq &\ 12[\rho(K_1)]^2\e\|\Delta U\|^2_{\mathcal{H}^\i_{\dbF}(T-\e,T+K)} I_1 + 3n[\rho(K_1)]^2I_2 + 3[\rho(K_1)]^2I_3.
\end{align}
For the term $I_1$, from (\ref{eq::3.27}) and (\ref{eq::3.2}), we have that

\begin{align}\label{eq::3.30}
I_1 \leq&\ 7\dbE_{\tau}^{\dbQ^{i}} \int_{\tau}^{T}\big( 1 +|V_{s}|^2 + |\tilde{V}_{s}|^2 + 2|\tilde{Z}_{s}|^2 + \dbE_{s}[|V_{s+\zeta_s}|^2] + \dbE_{s}[|\tilde{V}_{s+\zeta_s}|^2]\big)ds \nonumber\\
\leq&\ 7c_2^2\dbE_{\tau}\int_{\tau}^{T}\big( 1 +|V_{s}|^2 + |\tilde{V}_{s}|^2 + 2|\tilde{Z}_{s}|^2 + \dbE_{s}[|V_{s+\zeta_s}|^2] + \dbE_{s}[|\tilde{V}_{s+\zeta_s}|^2]\big) ds \nonumber\\
\leq&\ 7c_2^2\dbE_{\tau}\int_{\tau}^{T}\big( 1 +|V_{s}|^2 + |\tilde{V}_{s}|^2 + 2|\tilde{Z}_{s}|^2\big) ds
 + 7c^2_2L\dbE_{\tau}\int_{\tau}^{T +K}\big( |V_{s}|^2 + |\tilde{V}_{s}|^2\big) ds \nonumber\\
\leq&\  14c_2^2(L+1)\Big\{x_0 + \|V\|^2_{\mathcal{Z}^2_{\dbF}(T-\e,T + K;\dbR^{n\times d})} + \|\tilde{V}\|^2_{\mathcal{Z}^2_{\dbF}(T-\e,T + K;\dbR^{n\times d})} + 2\|\tilde{Z}\|^2_{\mathcal{Z}^2_{\dbF}(T-\e,T + K;\dbR^{n\times d})}\Big\} \nonumber \\
\leq&\ 14c_2^2(L+1)(x_0 + 4K_2).
\end{align}
For the term $I_3$, using H\"{o}lder's inequality again, we have that
\begin{equation}\label{eq:2.28}
\begin{aligned}
I_3\leq&\ \dbE_{\tau}^{\dbQ^{i}}\bigg\{\int_{\tau}^{T}\Big(1 + |V_{s}|^{\alpha} + |\tilde{V}_{s}|^{\alpha} + 2|\tilde{Z}_{s}|^{\alpha} +\big(\dbE_s[|V_{s+\zeta_s}|]\big)^{\alpha}
+\big(\dbE_s[|\tilde{V}_{s+\zeta_s}|]\big)^{\alpha}\Big)^2ds\\
&\qq\q  \cd \int_{\tau}^{T}\big(\dbE_{s}[|\Delta V_{s+\zeta_s}|]\big)^2ds\bigg\}\\
\leq&\ 7\bigg\{\dbE_{\tau}^{\dbQ^{i}}\Big[\Big(\int_{\tau}^{T}1 + |V_{s}|^{2\alpha} + |\tilde{V}_{s}|^{2\alpha} + 2|\tilde{Z}_{s}|^{2\alpha}+\big(\dbE_s[|V_{s+\zeta_s}|]\big)^{2\alpha}
+ \big(\dbE_s[|\tilde{V}_{s+\zeta_s}|]\big)^{2\alpha}ds\Big)^2\Big]\bigg\}^{\frac{1}{2}}\\
&\ \q\cd \bigg\{\dbE_{\tau}^{\dbQ^{i}}\Big[\Big(\int_{\tau}^{T}\big(\dbE_{s}[|\Delta V_{s+\zeta_s}|]\big)^2ds\Big)^2\Big]\bigg\}^{\frac{1}{2}}.
\end{aligned}
\end{equation}
In order to estimate the last term in \rf{eq:2.28}, we denote that
\begin{align*}
M_{t} = \int_{0}^{t}\dbE_{s}[|\Delta V_{s+\zeta_s}|]dW_{s}.
\end{align*}
Then 
\begin{align*}
&\|M\|^2_{BMO^{[T-\e, T]}(\dbP)} = \mathop{\sup}\limits_{\tau \in \mathscr{T}[T-\e,T]}\Big\|\dbE_{\tau}\Big[\int_{\tau}^{T}\Big(\dbE_{s}[|\Delta V_{s+\zeta_s}|]\Big)^2ds\Big]\Big\|_{\i} \leq \mathop{\sup}\limits_{\tau \in \mathscr{T}[T-\e,T]}\Big\|\dbE_{\tau}\Big[\int_{\tau}^{T}|\Delta V_{s+\zeta_s}|^2ds\Big]\Big\|_{\i}\\
&\leq  L\mathop{\sup}\limits_{\tau \in \mathscr{T}[T-\e,T]}\Big\|\dbE_{\tau}\Big[\int_{\tau}^{T+K}|\Delta V_{s}|^2ds\Big]\Big\|_{\i} \leq L\|\Delta V\|^{2}_{\mathcal{Z}^2_{\dbF}(T-\e,T + K; \dbR^{n\times d})}.
\end{align*}
Moreover, from (\ref{eq::3.27}), (\ref{eq::3.28}), and \autoref{BDG},  there exists a positive constant $b_1 $ such that for any $\tau \in \mathscr{T}[T-\e, T]$,
\begin{equation}\label{eq:2.30}
	\begin{aligned}
		&\bigg\{\dbE_{\tau}^{\dbQ^{i}}\Big(\int_{\tau}^{T}(\dbE_{s}[|\Delta V_{s+\zeta_s}|])^2ds\Big)^2\bigg\}^{\frac{1}{2}}
		\leq b_1 \mathop{\sup}\limits_{r \in \mathscr{T}[\tau, T]} \bigg\{\Big\|\dbE_{\tau}^{\dbQ^{i}}\Big[\Big(\int_{\tau}^{r}\dbE_{s}[|\Delta V_{s+\zeta_s}|]d\tilde{W}^i_s\Big)^{4}\Big]\Big\|^{\frac{1}{4}}_{\i}\bigg\}^2\\ 
		&\leq b_1 \mathop{\sup}\limits_{r \in \mathscr{T}[T-\e, T]}\|\tilde{M}^i\|^2_{BMO^{[T-\e, r]}_{4}(\dbQ^{i})}
		%   \leq b_1L^2_{4}\|M\|^2_{BMO(\dbQ^{i})}
		\leq b_1L^2_{4}\|\tilde{M}^i\|^2_{BMO^{[T-\e, T]}(\dbQ^i)} \leq b_1L_{4}^2c_{2}^2\|M\|^2_{BMO^{[T-\e, T]}(\dbP)} \\
		&\leq Lb_1L_{4}^2c_{2}^2\|\Delta V\|^{2}_{\mathcal{Z}^2_{\dbF}(T-\e,T + K; \dbR^{n\times d})},
	\end{aligned}
\end{equation}
where $\tilde{M}^i = M - \langle M, \Lambda^i\cd W\rangle$.
Now,  using H\"{o}lder's inequality and similar to  (\ref{eq:2.30}),  there exists a positive constant $b_2 $ such that for any $\tau \in \mathscr{T}[T-\e, T]$,
\begin{align}
	&\dbE_{\tau}^{\dbQ^{i}}\Big\{\int_{\tau}^{T}1 + |V_{s}|^{2\alpha} + |\tilde{V}_{s}|^{2\alpha} + 2|\tilde{Z}_{s}|^{2\alpha}+(\dbE_s[|V_{s+\zeta_s}|])^{2\alpha}+ (\dbE_s[|\tilde{V}_{s+\zeta_s}|])^{2\alpha}ds\Big\}^2 \nonumber \\
	&\leq \e^{2(1-\alpha)}\dbE_{\tau}^{\dbQ^{i}} \bigg\{x_0^{\alpha} + \Big(\int_{\tau}^{T}|V_s|^2ds\Big)^{\alpha}+\Big(\int_{t}^{T}|\tilde{V}_s|^2ds\Big)^{\alpha}+2\Big(\int_{\tau}^{T}|\tilde{Z}_s|^2ds\Big)^{\alpha} \nonumber\\
	&\qq\q\q\q\q\q+\Big(\int_{t}^{T}|\dbE_s[|V_{s+\zeta_s}|]|^2ds\Big)^{\alpha} + \Big(\int_{\tau}^{T}|\dbE_s[|\tilde{V}_{s+\zeta_s}|]|^2ds\Big)^{\alpha}\bigg\}^2\nonumber\\
	&\leq 7\e^{2(1-\alpha)}\dbE_{\tau}^{\dbQ^{i}}
	\bigg\{x_0^{2} + 7 + \Big(\int_{\tau}^{T}|V_s|^2ds\Big)^{2}+\Big(\int_{t}^{T}|\tilde{V_s}|^2ds\Big)^{2}+2\Big(\int_{t}^{T}|\tilde{Z_s}|^2ds\Big)^{2} \nonumber\\
	&\qq\q\q\q\q\q+\Big(\int_{\tau}^{T}|\dbE_s[|V_{s+\zeta_s}|]|^2ds\Big)^{2} + \Big(\int_{\tau}^{T}|\dbE_s[|\tilde{V}_{s+\zeta_s}|]|^2ds\Big)^{2}\bigg\} \nonumber\\
	&\leq 7\e^{2(1-\alpha)} \Big\{x_0^{2} + 7 + 6(L^2+1)b_2^2L_{4}^{4}c_2^4K_2^4\Big\}. \label{eqq:3.23}
\end{align}
Then, combining (\ref{eq:2.28}),  (\ref{eq:2.30}),  and (\ref{eqq:3.23}),  we obtain that
\begin{equation}\label{eq::3.35}
	\begin{aligned}
		I_3 \leq 7Lb_1L_{4}^2c_{2}^2\e^{1-\alpha} \sqrt{7\Big[x_0^{2} + 7 + 6(L^2+1)b_2^2L_{4}^{4}c_2^4K_2^4\Big]}\|\Delta V\|^{2}_{\mathcal{Z}^2_{\dbF}(T-\e,T + K;\dbR^{n\times d})}.
	\end{aligned}
\end{equation}
For the term $I_2$, 
similarly to the deduction of $I_3$,  there exists a positive constant $b_3$ such that
\begin{align}\label{eq::3.36}
	I_2 \leq 7b_3L_{4}^2c_{2}^2\e^{1-\alpha} \sqrt{7\Big[x_0^{2} + 7 + 6(L^2+1)b_2^2L_{4}^{4}c_2^4K_2^4\Big]}\|\Delta V\|^{2}_{\mathcal{Z}^2_{\dbF}(T-\e,T + K;\dbR^{n\times d})}.
\end{align}
Finally, combining (\ref{eq::3.30}), (\ref{eq::3.35}), and (\ref{eq::3.36}), we obtain that for any $\tau\in \mathscr{T}[T-\e, T]$,
\begin{align}
	&|\Delta Y_{\tau}^{i}|^2 + \dbE_{\tau}^{\dbQ^{i}}\int_{\tau}^{T}|\Delta Z_{s}^{i}|^2ds \nonumber\\
	&\leq 168[\rho(K_1)]^2 c_2^2(L+1)(x_0 + 4K_2)\e\big\|\Delta U\big\|^2_{\mathcal{H}^\i_{\dbF}(T-\e,T+K;\dbR^{n})} + 42(n+1)(L+1)(b_1 +b_3)[\rho(K_1)]^2  \nonumber\\
	&\q \cdot L_{4}^2c_{2}^2\e^{1-\alpha}\sqrt{7\Big[x_0^{2} + 7 + 6(L^2+1)b_2^2L_{4}^{4}c_2^4K_2^4\Big]}\big\|\Delta V\big\|^{2}_{\mathcal{Z}^2_{\dbF}(T-\e,T + K;\dbR^{n\times d})}\nonumber\\ 
	&\leq C_1\e\|\Delta U\|^2_{\mathcal{H}^\i_{\dbF}(T-\e,T+K;\dbR^n)} + C_2 \e^{1-\alpha}\|\Delta V  \|^2_{\mathcal{Z}^2_{\dbF}(T-\e,T + K;\dbR^{n\times d})}, \label{eqq:3.25}
\end{align}
where 
$$C_1 = 168[\rho(K_1)]^2 c_2^2(L+1)(x_0 + 4K_2) $$
and
$$C_2 = 42(n+1)(L+1)(b_1 +b_3)[\rho(K_1)]^2 L_{4}^2c_{2}^2\sqrt{7\Big[x_0^{2} + 7 + 6(L^2+1)b_2^2L_{4}^{4}c_2^4K_2^4\Big]}.$$
In view of (\ref{eq:::3.27}) and (\ref{eqq:3.25}), we derive that
\begin{align}\label{eq:2.45}
	&\|\Delta Y\|^2_{\mathcal{H}^\i_{\dbF}(T-\e,T+K;\dbR^{n})} + c_1^2\|\Delta Z\|^{2}_{\mathcal{Z}^2_{\dbF}(T-\e,T+K ;\dbR^{n\times d})} \nonumber\\
	&=\|\Delta Y\|^2_{\mathcal{H}^\i_{\dbF}(T-\e,T ;\dbR^{n})} + c_1^2\|\Delta Z\|^{2}_{\mathcal{Z}^2_{\dbF}(T-\e,T;\dbR^{n\times d})} \nonumber \\
	&\leq nC_1\e\|\Delta U\|^2_{\mathcal{H}^\i_{\dbF}(T-\e,T+K;\dbR^{n})} + nC_2 \e^{1-\alpha}\|\Delta V  \|^2_{\mathcal{Z}^2_{\dbF}(T-\e,T + K;\dbR^{n\times d})}.
\end{align}
Therefore, by taking $\e$ sufficiently small, we deduce that the mapping $\Pi$ is a contraction on  $\mathcal{B}_{\e}(K_1, K_2)$. This completes the proof.
%\qed
%
\end{proof}

\section{Global solution with bounded terminal value}\label{section4}

In this section, we study global  solutions of the quadratic BSDE (\ref{eq::1.1})  with bounded terminal value.
Recall that  $\lambda, \lambda_{0}, \sigma, \sigma_0, \gamma$, and $\alpha \in [0, 1)$ are some positive constants, and  $\{\theta_{t};t\in[0, T]\}$ is an $\dbF$-progressively measurable scalar-valued non-negative process. 
Additionally, note that $f(\omega, t, y, z, \phi_r, \psi_{\bar{r}}):  \Om\times [0, T]\times \dbR^n \times \dbR^{n\times d}\times L^2_{\sF_r}(\Om; \dbR^n) \times L^2_{\sF_{\bar{r}}}(\Om; \dbR^{n\times d}) \rightarrow L^2_{\sF_t}(\Om;\dbR^n)$ and $f(\cdot, \cdot, y, z,\phi_r, \psi_{\bar{r}})$ is $\dbF$-progressively measurable, where $r, \bar{r} \in [t, T+K]$. 
\begin{assumption}\label{assumption2}\rm 
		
For $i = 1, ..., n$,  and for all   $t \in [0, T]$, $r, \bar{r} \in [t, T+K]$,  $y\in \dbR^n, z\in \dbR^{n\times d}, \phi\in \mathcal{H}^2_{\dbF}(t,T + K;\dbR^{n})$, and $\psi\in \mathcal{H}^2_{\dbF}(t,T + K;\dbR^{n\times d})$,  the generator $f^i$ satisfies that $d\dbP \times dt$-a.e.,
$$|f^i(t, y, z, \phi_{r},\psi_{\bar{r}})| \leq 
\theta_{t} + \sigma|y| +   \sigma_0\dbE_t[|\phi_{r}|] +  \frac{\gamma}{2}|z^i|^2.
$$
\end{assumption}
\begin{assumption}\label{assumption5} \rm 
	For $i = 1, ..., n$, and for all $t \in [0, T]$, $r, \bar{r} \in [t, T+K]$, $y\in \dbR^n, z\in \dbR^{n\times d}, \phi\in \mathcal{H}^2_{\dbF}(t,T + K;\dbR^{n})$, and $\psi\in \mathcal{H}^2_{\dbF}(t,T + K;\dbR^{n\times d})$,  the generator $f^i$ satisfies that $d\dbP \times dt$-a.e.,
	\begin{itemize}
		\item [$\rm(i)$] $$|f^{i}(t, y, z, \phi_{r},\psi_{\bar{r}})| \leq 
		\theta_{t} + \sigma|y| +  \lambda |z|^{1+\alpha}+ \sigma_0\dbE_t[|\phi_{r}|] +\lambda_{0}\dbE_{t}[|\psi_{\bar{r}}|] +  \frac{\gamma}{2}|z^i|^2.
		$$
		\item [$\rm(ii)$] \begin{align*}
			f^{i}(t, y, z, \phi_{r},\psi_{\bar{r}}) \geq \frac{\gamma}{2}|z^i|^2 - \theta_{t} - \sigma|y| - \lambda |z|^{1+\alpha} - \sigma_0\dbE_t[|\phi_{r}|] -\lambda_{0}\dbE_{t}[|\psi_{\bar{r}}|]
		\end{align*}
		or
		\begin{align*}
			f^{i}(t, y, z, \phi_{r},\psi_{\bar{r}}) \leq -\frac{\gamma}{2}|z^i|^2 + \theta_{t} + \sigma|y| + \lambda |z|^{1+\alpha} + \sigma_0\dbE_t[|\phi_{r}|] +\lambda_{0}\dbE_{t}[|\psi_{\bar{r}}|].
		\end{align*}
		\item [$\rm(iii)$] 
			There exists constant $M_4$ such that
			\begin{align}\label{20241216}
				\mathop{\sup}\limits_{\tau\in[T, T+K]}\bigg\|\dbE_{\tau}\exp\bigg\{\frac{\gamma\e_0}{2n} \int_\tau^{T+K}|\eta_s|^2ds\bigg\}\bigg\|_{\i} \leq M_4 
				,
			\end{align}
			where 
			\begin{equation}\label{definitione_01}
				\e_0 \triangleq \min \Big\{ \Big(\frac{\gamma}{16n(18T\lambda_{0}^2L^2 + 1)} \Big)^{\frac{1+\alpha}{1-\alpha}}, \frac{\gamma}{24\big(2+(\sigma+\sigma_0)T\big)}, \frac{\gamma}{9} \Big\}.
			\end{equation}
			%\begin{align*} 
			%		\varepsilon_1 \triangleq  .
			%\end{align*}
	\end{itemize}
\end{assumption}
%
%\begin{remark}
%	The following generator satisfies such an assumption: for $i = 1, ..., n, y\in \dbR^n, z\in \dbR^{n\times d}, \phi_{r}\in L^{2}_{\sF_r}(\Om; \dbR^{n}),\psi_{\bar{r}}\in L^{2}_{\sF_{\bar{r}}}(\Om; \dbR^{n\times d}),$ where $r, \bar{r} \in [t, T+K],$
%	\begin{equation}
	%		f^{i}(t, y, z, \phi_{r},\psi_{\bar{r}})  = 1 + |y| + |z^{i}|^2  + \dbE_{t}\Big[|\phi_{r}|\Big].
	%	\end{equation}
%\end{remark}
%and 
%\begin{equation*}\begin{aligned}
%		&|f^{i}(w, t, y, z, \phi_{r}, \psi_{\bar{r}}) - f^{i}(w, t, y, z, 0, 0)|\\
%		&\leq C \Big\{1 + \dbE_t[|\phi_{r} ](w)\Big\}
%\end{aligned}\end{equation*}

\begin{remark} \rm 
	In \autoref{assumption5},
on one hand, the reason why we introduce the item (iii) is because  items (i) and (ii) involve the anticipated term of $Z_t$. In fact,  the method  used to handle the term of $Z_t$   in  Fan--Hu--Tang \cite [Theorem 2.5]{fan2023multi} fails to address  the anticipated term of $Z_t$.
	On the other hand,  the item (iii) could be removed if $\lambda_{0} =0$, and  both items (ii) and (iii) could  be removed  if $\lambda = \lambda_{0} =0$ {(in this special case, \autoref{assumption5} becomes \autoref{assumption2})}. 
\end{remark}

%\begin{remark} \rm 
%	On one hand, we can delete item (iii) if $\lambda_{0} =0$ and delete item (ii)-(iii) if $\lambda = \lambda_{0} =0$ \tc{(i.e., it is just \autoref{assumption2})}. On the other hand, the reason why we introduce the item (iii) is that item (i) and item (ii) contain the anticipated term of $Z_t$, actually, the method in \autoref{global_bdd2} (which will be presented below) used to tackle the term of $Z_t$ fails to deal with the anticipated term of $Z_t$.
%\end{remark}

\begin{theorem}\label{global_bdd} \sl 
	Let \autoref{assumption1} and \autoref{assumption2}  hold. Then  BSDE \rf{eq::1.1}  admits a unique global adapted solution $(Y, Z) \in \mathcal{H}^\i_{\dbF}(0, T+K;\dbR^{n}) \times \mathcal{Z}^2_{\dbF}(0,T+K;\dbR^{n\times d})$. Moreover, there exist  positive constants $J_1$ and $J_2$ depending only on $(n, T, \sigma, \sigma_0, M_1, M_3)$ and $(n, T, \sigma, \sigma_0, M_1, M_2, M_3)$ respectively such that 
	\begin{align*}
		\|Y\|_{\mathcal{H}^\i_{\dbF}(0,T+K; \dbR^n)}\leq J_1 \q~ {\rm and}\q~
		\|Z\|^{2}_{\mathcal{Z}^2_{\dbF}(0,T + K; \dbR^{n\times d})}\leq J_2.
	\end{align*}
\end{theorem}
\begin{theorem}\label{global_bdd2}  \sl 
	Let \autoref{assumption1} and \autoref{assumption5}  hold. Then BSDE  \rf{eq::1.1} admits a unique global adapted solution $(Y, Z) \in \mathcal{H}^\i_{\dbF}(0, T+K;\dbR^{n}) \times \mathcal{Z}^2_{\dbF}(0,T+K;\dbR^{n\times d})$. Moreover, there exist  positive constants $J_3$ and $J_4$, depending only on $(n, T, L, \alpha, \sigma,  \sigma_0, \lambda, \lambda_0, \gamma,  M_1, M_3, M_4)$ such that 
	\begin{align*}
		\|Y\|_{\mathcal{H}^\i_{\dbF}(0,T+K; \dbR^n)}\leq J_3 \q~ {\rm and}\q~
		\|Z\|^{2}_{\mathcal{Z}^2_{\dbF}(0,T + K; \dbR^{n\times d})}\leq J_4.
	\end{align*}
\end{theorem}
In order to prove \autoref{global_bdd} and \autoref{global_bdd2}, we need the following lemma, which provides a priori estimate for  BSDE \eqref{eq::1.1}. Note that $\kappa\in(0,T]$ is a proper constant  considered below.

\begin{lemma}[A priori estimate]\label{lemma:4.4} \sl 
	
Let	 the condition  \rf{121201} hold.  On one hand, if the generator $f$ satisfies \autoref{assumption2}, then the solution  $(Y,Z)$ of BSDE  \eqref{eq::1.1} admits the following estimate
\begin{align} \label{eq::4.11}
\|Y\|_{\mathcal{H}^\i_{\dbF}(T-\kappa, T+K; \dbR^n)} \leq  Q_{1}\q \hbox{and}\q
\|Z\|^{2}_{\mathcal{Z}^2_{\dbF}(T-\kappa,T + K; \dbR^{n \times d})} \leq Q_{2},
\end{align}
where $Q_1$ is a positive constant depending only on $(n, T, \sigma, \sigma_0, M_1, M_3)$ and $Q_{2}$ is also a positive constant depending only on $(n, T, \sigma, \sigma_0, M_1, M_2, M_3)$.
On the other hand, for positive constant $\kappa \leq \frac{\e_0}{8n^3\lambda_{0}^2\gamma L^2}$, if the generator $f$ satisfies \autoref{assumption5}, then  $(Y,Z)$  of BSDE  \eqref{eq::1.1} admits the following estimate
\begin{align} \label{eq:::4.13}
	\|Y\|_{\mathcal{H}^\i_{\dbF}(T-\kappa, T+K; \dbR^n)} \leq  Q_{3}\q\hbox{and}\q 
	\|Z\|^{2}_{\mathcal{Z}^2_{\dbF}(T-\kappa,T + K; \dbR^{n\times d})} \leq Q_{4}, 
\end{align}
where 
%$$
%	M_5 \triangleq \frac{1}{n}\ln \frac{36^2 M_4^4 \bar{c}}{\gamma}   + M_3 + \frac{12\e_0 M_3}{n\gamma} + \frac{4C_1 T}{n\gamma} +  \frac{C_2 T}{n\gamma}
%$$ 
%
$Q_{3}$ and $Q_{4}$ are positive constants depending only on $(n, T, L, \alpha, \sigma,  \sigma_0, \lambda, \lambda_0, \gamma,  M_1, M_3, M_4)$.
\end{lemma}
\begin{proof}
For simplicity of presentation,	the proof will be divided into two steps. 
	%For $i = 1, ..., n, A \in \dbR^{n\times d}$ and $z\in \dbR^{ d},$ we denote by $A(z;i)$ the matrix in $\dbR^{n\times d}$ whose $i$th row is $z$ and whose $j$th row is $A^j$ for any $j\neq i$.
%

\ms

\noindent {{\it Step 1:  The proof of the estimate \eqref{eq::4.11}.} }
On one hand,
by \autoref{assumption2},  we have that  for $i = 1, ..., n,$
\begin{align*}
 |f^{i}(t, Y_t, Z_t, Y_{t+\delta_t}, Z_{t+\zeta_t})|
\leq\theta_{t} + \sigma|Y_t| + \sigma_0\dbE_t[|Y_{t + \delta_t}|] 
 +  \frac{\gamma}{2}|Z_t^i|^2,\q t \in [T-\kappa, T].
\end{align*}
On the other hand,
it is obvious that $f^i$,  the  $i$-th  component of the generator $f$,  satisfies (\ref{eq:6.6}). 
Then, it  follows from \autoref{lemma:6.2},   noting the condition   \rf{121201},  that for each $i = 1, ..., n,$
\begin{align}\label{eq::4.12}
		\exp(\gamma|Y^i_t|) \leq &\exp\Big\{\gamma (M_1 + M_3) + (\sigma + \sigma_0)\gamma\|Y\|_{\mathcal{H}^\i_{\dbF}(t,T+K;\dbR^{n})}(T-t) \Big\}, \q \forall t\in [0, T],
\end{align}
which implies that  for $i = 1, ...,n,$
\begin{align*}
|Y^i_t| \leq  M_1 + M_3 + (\sigma + \sigma_0)\|Y\|_{\mathcal{H}^\i_{\dbF}(t,T+K;\dbR^n)}(T-t) , \q  t\in [T - \kappa, T].
\end{align*}
Hence, by combining  each component $Y^i$ of $Y$, we have that on the interval $[T - \kappa, T]$,
\begin{align*}
\|Y\|_{\mathcal{H}^\i_{\dbF}(t,T;\dbR^n)} \leq  n(M_1 + M_3) + n(\sigma + \sigma_0)\|Y\|_{\mathcal{H}^\i_{\dbF}(t,T+K;\dbR^n)}(T-t),
\end{align*}
and thus on the interval $[T - \kappa, T+K]$, 
\begin{align}\label{eq::4.13}
	\|Y\|_{\mathcal{H}^\i_{\dbF}(t,T+K;\dbR^n)} \leq  2n(M_1 + M_3) + n(\sigma + \sigma_0)\|Y\|_{\mathcal{H}^\i_{\dbF}(t,T+K;\dbR^n)}(T-t).
\end{align}

Now, we prove the first estimate of (\ref{eq::4.11}). 
Note that if $\sigma = \sigma_0 = 0$, then we have from \rf{eq::4.13}  that 
\begin{align*}
\|Y\|_{\mathcal{H}^\i_{\dbF}(T-\kappa,T+K;\dbR^n)} \leq  2n(M_1 + M_3).
\end{align*} 
Otherwise,  in order to show the first estimate of (\ref{eq::4.11}) holds, we denote $\epsilon = \frac{1}{2n(\sigma + \sigma_0)}$ and let  $m_0$ be the unique positive integer such that 
\begin{align}\label{241216}
2n(\sigma + \sigma_0)\kappa \leq m_0 < 2n(\sigma + \sigma_0)\kappa + 1.
\end{align}
Note that, if $m_0 = 1$,  we have that  $n(\sigma+\sigma_0)(T - t)\leq n(\sigma+\sigma_0)\kappa \leq \frac{1}{2}$  when $t\in [T-\kappa, T]$. This, combined with  (\ref{eq::4.13}),  implies that 
\begin{align}\label{eq:::4.15}
\|Y\|_{\mathcal{H}^\i_{\dbF}(T-\kappa,T+K;\dbR^n)} \leq  4n(M_1 + M_3).
\end{align}
If $m_0 \neq 1$,  we have that $n(\sigma+\sigma_0)(T - t)\leq n(\sigma+\sigma_0)\epsilon = \frac{1}{2}$ when
 $t\in [T-\epsilon, T]$. This  implies that
\begin{align}\label{eq::4.15}
	\|Y\|_{\mathcal{H}^\i_{\dbF}(T-\epsilon,T+K;\dbR^n)} \leq  4n(M_1 + M_3).
\end{align}
It also implies that  on the interval $[T-\epsilon,T-\epsilon+K$], 
\begin{align}\label{eq::4.17}
	\|Y\|_{\mathcal{H}^\i_{\dbF}(T-\epsilon,T-\epsilon+K;\dbR^n)} \leq  4n(M_1 + M_3).
\end{align}
Next, we would like to suture the interval $[T-\epsilon,T-\epsilon+K$] and assert that \rf{eq::4.17} still holds on the interval $[T-\k,T+K$]. 
%For this purpose, we consider the following BSDE
%%
%\begin{equation}\left\{\begin{aligned}\label{bsde}
%		&\tilde{Y}_t = Y_{T-\epsilon} + \int_{t}^{T-\epsilon}f(s, \tilde{Y}_s, \tilde{Z}_s, \tilde{Y}_{s+\delta_s}, \tilde{Z}_{s+\zeta_s})ds - \int_{t}^{T-\epsilon} \tilde{Z}_s dW_s,  \quad   t\in [T-\kappa,T-\epsilon]; \\
%		&\tilde{Y}_t = Y_t, \quad  \tilde{Z}_t = Z_t, \quad       t\in[T - \epsilon, T-\epsilon+K].
%	\end{aligned}\right.\end{equation}
%%
In fact,  similarly to the process used to obtain  (\ref{eq::4.17}), we have that on the interval $[T-2\epsilon,T-\epsilon+K]$,
\begin{align*}
\|Y\|_{\mathcal{H}^\i_{\dbF}(T-2\epsilon,T-\epsilon+K;\dbR^n)} \leq  4n\Big(4n(M_1 + M_3)+M_3\Big) = (4n)^2M_1 + \Big(4n + (4n)^2\Big)M_3.
\end{align*}
Proceeding with the above computation gives us that for  $1\leq j \leq m_0 -1$,
\begin{align*}
\|Y\|_{\mathcal{H}^\i_{\dbF}(T-j\epsilon,T-(j-1)\epsilon+K;\dbR^n)} \leq  (4n)^{j}M_1 + \Big(4n + (4n)^2 + ... + (4n)^j\Big)M_3.
\end{align*}
Note that
\begin{align*}
	2n(\sigma + \sigma_0)\{\kappa-(m_0 - 1)\epsilon\} \leq 1 \leq 2n(\sigma + \sigma_0)\{\kappa-(m_0 - 1)\epsilon\} + 1,
\end{align*}
then, similarly to the process used to obtain (\ref{eq:::4.15}), we have 
\begin{align*}
	\|Y\|_{\mathcal{H}^\i_{\dbF}(T-\kappa, T-(m_0 - 1)\epsilon+K;\dbR^n)} \leq  (4n)^{m_0}M_1 + \big\{4n + (4n)^2 + ... + (4n)^{m_{0}}\big\}M_3.
\end{align*}
Thus, note \rf{241216}, we have 
\begin{align*}
	\|Y\|_{\mathcal{H}^\i_{\dbF}(T-\kappa, T+K;\dbR^n)} &\leq  (4n)^{m_0}M_1 
	+ \big\{4n + (4n)^2 + ... + (4n)^{m_{0}}\big\} M_3\\
&\leq 		(4n)^{[2n(\sigma + \sigma_0)T] +1}M_1 + \big\{4n + (4n)^2 + ... + (4n)^{[2n(\sigma + \sigma_0)T]+1}\big\} M_3  \deq Q_1,
\end{align*}
which implies that the first estimate of (\ref{eq::4.11}) holds.

Next, in order to prove the second estimate of  (\ref{eq::4.11}), we define the following function
$$\Phi(x) \triangleq \frac{1}{\gamma^2} [\exp(\gamma|x|) - \gamma|x| - 1], \q x\in \dbR.$$
It is easy to check that for $x\in \dbR $, 
\begin{align}\label{eq::4.255}
\Phi'(x) = \frac{1}{\gamma} [\exp(\gamma|x|) - 1]\sgn(x),\q \Phi''(x) = \exp(\gamma|x|), \q  \Phi''(x)-\gamma|\Phi'(x)| = 1.
\end{align}
For any $\tau\in \mathscr{T}[0, T]$, by applying It\^{o} formula to $\Phi(Y^i)$ in $[\tau, T]$,   noting  \autoref{assumption2} and (\ref{eq::4.255}), we deduce that  
\begin{align*}
\Phi(Y_{\tau}^i)  + \frac{1}{2}\dbE_\tau\int_{\tau}^{T}|Z_s^i|^2ds
&\leq  \dbE_\tau\big[\Phi(\xi_{T}^i)\big] + \dbE_\tau \int_{\tau}^{T}|\Phi'(Y_s^i)|\big(\theta_s + \sigma|Y_s| + \sigma_0\dbE_s[|Y_{s+\delta_s}|]\big)ds \\
&\leq \Phi(M_1) + |\Phi'(Q_1)| \big\{ M_3T + (\sigma+\sigma_0)Q_{1}T\big\}.
\end{align*}
Therefore, we get that for any $\tau\in \mathscr{T}[0, T+K]$,
\begin{align*}
	\dbE_\tau\int_{\tau}^{T+K}|Z_s|^2ds \leq 2n\Phi(M_1) +M_2 + 2n|\Phi'(Q_1)| \big\{M_3T + (\sigma+\sigma_0)Q_{1}T\big\} \triangleq Q_2.
\end{align*}
Consequently, we deduce (\ref{eq::4.11}) easily.

\ms

\noindent {{\it Step 2: The proof of the estimate \eqref{eq:::4.13}.} }
By the item (ii) of \autoref{assumption5}, we have
\begin{align*}
	f^{i}(t, Y_t, Z_t, Y_{t+\delta_t},Z_{t+\zeta_t}) &\geq \frac{\gamma}{2}|Z_t^i|^2 - \theta_{t} - \sigma|Y_t| - \lambda |Z_t|^{1+\alpha} - \sigma_0\dbE_t[|Y_{t+\delta_t}|] -\lambda_{0}\dbE_{t}[|Z_{t+\zeta_t}|]
\end{align*}
or 
\begin{align*}\label{definitione_00}
	f^{i}(t, Y_t, Z_t, Y_{t+\delta_t},Z_{t+\zeta_t}) &\leq -\frac{\gamma}{2}|Z_t^i|^2 + \theta_{t} + \sigma|Y_t| + \lambda |Z_t|^{1+\alpha} + \sigma_0\dbE_t[|Y_{t+\delta_t}|] +\lambda_{0}\dbE_{t}[|Z_{t+\zeta_t}|].
\end{align*}
It then follows from the definition of $\e_0$, \autoref{lemma:6.2-2}, and H\"{o}lder's inequality that when $\tau\in \mathscr{T}[T-\kappa, T]$,
\begin{equation}\label{4.12}
\begin{aligned}
	&\ \dbE_\tau\exp \bigg\{ \frac{\gamma}{2}\e_0 \int_{\tau}^{T}|Z^i_s|^2ds\bigg\} \\
	&\leq \dbE_\tau\exp\bigg\{6\e_0 \|Y\|_{\mathcal{H}^\i_{\dbF}(\tau, T+K;\dbR^n)} +  3\e_0 M_3 + 3\e_0(\sigma+\sigma_0)\|Y\|_{\mathcal{H}^\i_{\dbF}(\tau, T+K;\dbR^n)}T \\
	&\q+ 3\e_0\lambda \int_{\tau}^{T}|Z_s|^{1+\alpha}ds  + 3\e_0\lambda_{0}\int_{\tau}^{T}\dbE_{s}[|Z_{s+\zeta_s}|]ds\bigg\}\\
	&\leq  \exp\Big\{3\e_0 M_3 + 3\e_0\big(2+(\sigma+\sigma_0)T\big)\|Y\|_{\mathcal{H}^\i_{\dbF}(\tau, T+K;\dbR^n)}\Big\}\\
	&\q\cdot\dbE_\tau\exp\bigg\{6\e_0\lambda \int_{\tau}^{T}|Z_s|^{1+\alpha}ds\bigg\}   \cd \dbE_\tau \exp\bigg\{6\e_0\lambda_{0}\int_{\tau}^{T}\dbE_{s}[|Z_{s+\zeta_s}|]ds\bigg\}.
\end{aligned}
\end{equation}
Thus, by H\"{o}lder's inequality we get that 
\begin{equation}\label{eq::4.22}
	\begin{aligned}
		&\ \dbE_\tau\exp \bigg\{ \frac{\gamma\e_0}{2n} \int_{\tau}^{T}|Z_s|^2ds\bigg\}\\
		&\leq  \exp\Big\{3\e_0 M_3 + 3\e_0\big(2+(\sigma+\sigma_0)T\big)\|Y\|_{\mathcal{H}^\i_{\dbF}(\tau, T+K;\dbR^n)}\Big\}\\
		&\q\cdot\dbE_\tau\exp\bigg\{6\e_0\lambda \int_{\tau}^{T}|Z_s|^{1+\alpha}ds\bigg\}\cd  \dbE_\tau \exp\bigg\{6\e_0\lambda_{0}\int_{\tau}^{T}\dbE_{s}[|Z_{s+\zeta_s}|]ds\bigg\}.
	\end{aligned}
\end{equation}
For the last term in the above inequality, noting that for any stopping time $\tau\in \mathscr{T}[T-\kappa, T]$, if
$\|\dbE_{.}[|Z_{.+\zeta_{.}}|]\|_{\mathcal{Z}^2_{\dbF}(\tau,T;\dbR)} = 0,$
then we deduce that $\dbP$-a.s.,
$
\int_{\tau}^{T}\dbE_{s}[|Z_{s+\zeta_s}|]ds = 0.
$
%In view of \eqref{4.12}, we obtain that
%\begin{equation} \label{4.13}
%\begin{aligned}
%	%
%	&\ \dbE_\tau\exp \bigg\{ \frac{\gamma}{2}\e_0 \int_{\tau}^{T}|Z^i_s|^2ds\bigg\}\\
%	&\leq  \exp\Big\{3\e_0 M_3 + 3\e_0\big(2+(\sigma+\sigma_0)T\big)\|Y\|_{\mathcal{H}^\i_{\dbF}(\tau, T+K;\dbR^n)}\Big\}
%	%
%	\dbE_\tau\exp\bigg\{6\e_0\lambda \int_{\tau}^{T}|Z_s|^{1+\alpha}ds\bigg\}.
%\end{aligned}
%\end{equation}
% 
Hence, we have
\begin{align*}
 \dbE_\tau \exp\bigg\{6\e_0\lambda_{0}\int_{\tau}^{T}\dbE_{s}[|Z_{s+\zeta_s}|]ds\bigg\} = 1\q \hb{and} \q \dbE_{\tau}\exp\bigg\{2n\lambda_{0}\gamma \int_{\tau}^{T}\dbE_{s}[|Z_{s+\zeta_s}|]ds\bigg\} = 1.
\end{align*}
Thus, when $ \|\dbE_{.}[|Z_{.+\zeta_{.}}|]\|_{\mathcal{Z}^2_{\dbF}(\tau,T;\dbR)} = 0$, and in view of \eqref{eq::4.22} and \eqref{eq::4.33}, the estimate for $Z$ and $Y$ can be deduced by applying arguments similarly to those in Fan--Hu--Tang \cite[Lemma 4.1]{fan2023multi}.
Therefore, without loss of generality, in the rest of the proof, we assume that
$\|\dbE_{.}[|Z_{.+\zeta_{.}}|]\|_{\mathcal{Z}^2_{\dbF}(\tau,T;\dbR)} \neq 0$.
Then similarly to the deduction of (\ref{eq::3.6}), and noting the inequality $|ab| \leq \frac{1}{2} (a^2 + b^2)$, we  have that  
\begin{equation}\label{eq:::4.25}
	\begin{aligned}
		 6\e_0\lambda_{0} \dbE_{s}[|Z_{s+\zeta_s}|]
		&=\frac{\dbE_{s}[|Z_{s+\zeta_s}|]}{\|\dbE_{.}[|Z_{.+\zeta_{.}}|]\|_{\mathcal{Z}^2_{\dbF}(\tau,T;\dbR)}} \cd  6\e_0\lambda_{0} \|\dbE_{.}[|Z_{.+\zeta_{.}}|]\|_{\mathcal{Z}^2_{\dbF}(\tau,T;\dbR)} \\
		&\leq \frac{\big\{\dbE_{s}[|Z_{s+\zeta_s}|]\big\}^2}{2\|\dbE_{.}[|Z_{.+\zeta_{.}}|]\|^2_{\mathcal{Z}^2_{\dbF}(\tau,T;\dbR)}}
		 +  18\e_0^2\lambda_{0}^2 L^2 \|Z\|^{2}_{\mathcal{Z}^2_{\dbF}(\tau,T+K;;\dbR^{n\times d})}.
\end{aligned}\end{equation}
Next, in order to better tackle the above terms, we define the following for some given time  $\bar t$ such that for
$T-\k\les \tau\les \bar t  \leq T$,
 \begin{equation}\label{241219}
 P_{\bar{t}} \deq  \int_{\bar{t}}^{T}\frac{\dbE_{s}[|Z_{s+\zeta_s}|]}{{\|\dbE_{.}[|Z_{.+\zeta_{.}}|]\|_{\mathcal{Z}^2_{\dbF}(\tau,T;\dbR)}}}dW_s.
 \end{equation}
Now, since   $Z\in\mathcal{Z}^2_{\dbF}(T-\kappa,T+K;\dbR^{n\times d})$,   \autoref{John-Nirenberg inequality} implies that for $\bar{\tau}\in \mathscr{T}[\tau, T]$,
\begin{equation}\label{J}
	\begin{aligned}
		&\ \dbE_{\bar{\tau}} \exp\bigg\{\frac{1}{2\|\dbE_{.}[|Z_{.+\zeta_{.}}|]\|^2_{\mathcal{Z}^2_{\dbF}(\tau ,T;\dbR)}}\int_{\bar{\tau}}^{T}|\dbE_{s}[|Z_{s+\zeta_s}|]|^2ds\bigg\} \\
		&\leq \dbE_{\tau}\exp\Big\{\frac{9}{16}\langle P \rangle_{T} - \frac{9}{16}\langle P \rangle_{\tau }\Big\} \leq \frac{1}{1-\frac{9}{16}\|P\|^2_{BMO^{[\tau, T]}(\dbP)}} = \frac{16}{7} \leq 3.
\end{aligned}\end{equation}
Thus we obtain that for any $\tau\in \mathscr{T}[T-\kappa, T]$,
\begin{equation}\label{eq:::4.26}
	\begin{aligned}
		&\ \dbE_{\tau} \exp\bigg\{\frac{1}{2\|\dbE_{.}[|Z_{.+\zeta_{.}}|]\|^2_{\mathcal{Z}^2_{\dbF}(\tau,T;\dbR)}}\int_{\tau }^{T}|\dbE_{s}[|Z_{s+\zeta_s}|]|^2ds\bigg\} \leq 3.
\end{aligned}\end{equation}
Therefore, it follows from (\ref{eq:::4.25}), (\ref{eq:::4.26}), and Jensen's inequality that
\begin{equation}\label{eq:::4.27}
	\begin{aligned}
		&\ \dbE_\tau\exp\bigg\{6\e_0\lambda_{0}\int_{\tau}^{T}\dbE_{s}[|Z_{s+\zeta_s}|]ds\bigg\} \leq 3\exp\Big\{18\e_0^2\lambda_{0}^2 L^2T \|Z\|^{2}_{\mathcal{Z}^2_{\dbF}(\tau,T+K;\dbR^{n\times d})}\Big\}\\
		&\leq 3\mathop{\sup}\limits_{\bar{\tau}\in \mathscr{T}[\tau, T+K]}\Big\|\dbE_{\bar{\tau}}\exp\Big\{18\e_0^2\lambda_{0}^2 L^2T\int_{\bar{\tau}}^{T+K} |Z_{s}|^{2}ds\Big\}\Big\|_{\i}.
\end{aligned}\end{equation}
Note that by Young's inequality, for any  positive constant $a$ and  $b$, one has 
\begin{align}\label{young}
	ab^{1+\alpha} = \bigg\{\Big(\frac{1+\alpha}{2}\Big)^{\frac{1+\alpha}{1-\alpha}} a^{\frac{2}{1-\alpha}}\bigg\}^{\frac{1-\alpha}{2}}
	\Big(\frac{2}{1+\alpha}b^2\Big)^{\frac{1+\alpha}{2}}
	 \leq b^2 + \frac{1-\alpha}{2}\Big(\frac{1+\alpha}{2}\Big)^{\frac{1+\alpha}{1-\alpha}}a^{\frac{2}{1-\alpha}}.
\end{align}
Then, by letting $a = 6 \lambda$ and $b = (\e_0)^{\frac{1}{1+\alpha}}|Z_s|$, we have 
\begin{align}\label{eq::4.23}
	\dbE_\tau\exp\bigg\{6\e_0\lambda \int_{\tau}^{T}|Z_s|^{1+\alpha}ds\bigg\} \leq \exp\{C_1 T\}\mathop{\sup}\limits_{\bar{\tau}\in \mathscr{T}[\tau, T+K]}\bigg\|\dbE_{\bar{\tau}}\exp \bigg\{\e_0^{\frac{2}{1+\alpha}}\int_{\bar{\tau}}^{T+K}|Z_s|^2ds\bigg\} \bigg\|_{\i},
\end{align} 
where 
\begin{align*}
	C_1 = \frac{1-\alpha}{2}\Big(\frac{1+\alpha}{2}\Big)^{\frac{1+\alpha}{1-\alpha}}(6 \lambda)^{\frac{2}{1-\alpha}}.
\end{align*}
Thus, by combining (\ref{eq::4.22}), (\ref{eq:::4.27}), and (\ref{eq::4.23}), we obtain  that for any $\tau\in \mathscr{T}[T-\kappa, T]$,
\begin{equation*}
	\begin{aligned}
		\dbE_\tau\exp \bigg\{ \frac{\gamma\e_0}{2n} \int_{\tau}^{T}|Z_s|^2ds\bigg\}
		&\leq 3\exp\Big\{3\e_0 M_3 +  C_1 T+
		3\e_0\big[2+(\sigma+\sigma_0)T\big]\|Y\|_{\mathcal{H}^\i_{\dbF}(\tau, T+K;\dbR^n)}\Big\}  \\
		&\ \q\cdot \mathop{\sup}\limits_{\bar{\tau}\in\mathscr{T}[\tau, T+K]}\bigg\|\dbE_{\bar{\tau}}\exp\bigg\{2\e_0^{\frac{2}{1+\alpha}}(18\lambda_{0}^2 L^2T+1)\int_{\bar{\tau}}^{T+K} |Z_{s}|^{2}ds\bigg\}\bigg\|_{\i},
\end{aligned}\end{equation*}
%\leq & 9\exp\Big(3\varepsilon M_3 +  
%3\varepsilon\big(2+(\sigma+\sigma_0)T\big)\|Y\|_{\mathcal{H}^\i_{\dbF}(t, T+K;\dbR^n)}\Big) \nonumber\\
%&\cdot\exp\Big(18\varepsilon^2(\lambda_{0}^2+\lambda^2) (L+1)^2T \|Z\|^{2}_{\mathcal{Z}^2_{\dbF}(t,T+K)}\Big)\nonumber\\
%\leq & 9\exp\Big(3\varepsilon M_3 + 3\varepsilon\big(2+(\sigma+\sigma_0)T\big)\|Y\|_{\mathcal{H}^\i_{\dbF}(t, T+K;\dbR^n)}\Big) \nonumber\\
%&\cdot\mathop{\sup}\limits_{\tau\in[t, T+K]}\exp\Big(18\varepsilon^2(\lambda_{0}^2+\lambda^2) (L+1)^2T \Big\|\dbE_{\tau}[\int_\tau^{T+K}|Z_s|^2ds]\Big\|_{\i}\Big)\nonumber\\
%\leq & 9\exp\Big(3\varepsilon M_3 + 3\varepsilon\big(2+(\sigma+\sigma_0)T\big)\|Y\|_{\mathcal{H}^\i_{\dbF}(t, T+K;\dbR^n)}\Big) \nonumber\\
%&\cdot\mathop{\sup}\limits_{\tau\in[t, T+K]}\Big\|\dbE_{\tau}\Big[\exp\Big(18\varepsilon^2(\lambda_{0}^2+\lambda^2) (L+1)^2T \int_\tau^{T+K}|Z_s|^2ds\Big)\Big]\bigg\|_{\i}.
%
which implies that
\begin{equation}\label{eq::4.27}
	\begin{aligned}
		&\mathop{\sup}\limits_{\bar{\tau}\in\mathscr{T}[\tau, T]}\bigg\|\dbE_{\bar{\tau}}\exp\bigg\{\frac{\gamma\e_0}{2n} \int_{\bar{\tau}}^T|Z_s|^2ds\bigg\} \bigg\|_{\i} \\
		&\leq 3\exp\Big\{3\e_0 M_3 +  C_1 T+
		3\e_0\big[2+(\sigma+\sigma_0)T\big]\|Y\|_{\mathcal{H}^\i_{\dbF}(\tau, T+K;\dbR^n)}\Big\} \\
		&\q\ \cdot \mathop{\sup}\limits_{\bar{\tau}\in\mathscr{T}[\tau, T+K]}\bigg\|\dbE_{\bar{\tau}}\exp\bigg\{2\e_0^{\frac{2}{1+\alpha}}(18\lambda_{0}^2 L^2T+1)\int_{\bar{\tau}}^{T+K} |Z_{s}|^{2}ds\bigg\}\bigg\|_{\i}.
\end{aligned}\end{equation}
Therefore, by \rf{20241216} and H\"{o}lder's inequality, we deduce that for $\tau\in \mathscr{T}[T-\kappa, T]$,
\begin{align}\label{eq::4.30}
	&\mathop{\sup}\limits_{\bar{\tau}\in\mathscr{T}[\tau, T+K]}\bigg\|\dbE_{\bar{\tau}}\exp\bigg\{\frac{\gamma\e_0}{4n} \int_{\bar{\tau}}^{T+K}|Z_s|^2ds\bigg\}\bigg\|_{\i}\nonumber\\
	& \leq \mathop{\sup}\limits_{\bar{\tau}\in \mathscr{T}[\tau, T]}\bigg\|\dbE_{\bar{\tau}}\exp\bigg\{\frac{\gamma\e_0}{4n} \int_{\bar{\tau}}^{T+K}|Z_s|^2ds\bigg\}\bigg\|_{\i} + \mathop{\sup}\limits_{\bar{\tau}\in\mathscr{T}[T, T+K]}\bigg\|\dbE_{\bar{\tau}}\exp\bigg\{\frac{\gamma\e_0}{4n} \int_{\bar{\tau}}^{T+K}|Z_s|^2ds\bigg\}\bigg\|_{\i}\nonumber\\
	&  \leq \mathop{\sup}\limits_{\bar{\tau}\in\mathscr{T}[\tau, T]}\Bigg(\bigg\|\dbE_{\bar{\tau}}\exp\bigg\{\frac{\gamma\e_0}{2n} \int_{\bar{\tau}}^{T}|Z_s|^2ds\bigg\}\bigg\|_{\i}\cdot\bigg\|\dbE_{\bar{\tau}}\dbE_{T}\exp\bigg\{\frac{\gamma\e_0}{2n} \int_T^{T+K}|\eta_s|^2ds\bigg\}\bigg\|_{\i}\Bigg)
	+ M_4\nonumber\\
	& \leq  \mathop{\sup}\limits_{\bar{\tau}\in \mathscr{T}[\tau, T]}\Bigg(\bigg\|\dbE_{\bar{\tau}}\exp\bigg\{\frac{\gamma\e_0}{2n} \int_{\bar{\tau}}^{T}|Z_s|^2ds\bigg\}\bigg\|_{\i}\cdot\bigg\|\mathop{\sup}\limits_{\bar{\tau}\in \mathscr{T}[T, T+K]}\dbE_{\bar{\tau}}\exp\bigg\{\frac{\gamma\e_0}{2n} \int_{\bar{\tau}}^{T+K}|\eta_s|^2ds\bigg\}\bigg\|_{\i} \Bigg)+ M_4\nonumber\\
	&\leq  M_4\mathop{\sup}\limits_{\bar{\tau}\in \mathscr{T}[\tau, T]}\bigg\|\dbE_{\bar{\tau}}\exp\bigg\{\frac{\gamma\e_0}{2n} \int_{\bar{\tau}}^{T}|Z_s|^2ds\bigg\}\bigg\|_{\i} +M_4,
\end{align}
%Since 
%$$\bigg\{\dbE_{\tau}\Big[\exp\bigg\{\frac{\gamma\e_0}{2n} \int_T^{T+K}|\eta_s|^2ds\bigg\}\Big]\bigg\}_{\tau \in [T-\kappa, T]}$$ is a martingale, then it follows from Doob's maximal inequality that
%\begin{equation}
%\bigg\|\mathop{\sup}\limits_{\tau\in[T-\kappa, T]}\dbE_{\tau}\Big[\exp\bigg\{\frac{\gamma\e_0}{2n} \int_T^{T+K}|\eta_s|^2ds\bigg\}\Big]\bigg\|_{\i} \leq 4 \bigg\|\dbE_{T}\Big[\exp\bigg\{\frac{\gamma\e_0}{4n} \int_T^{T+K}|\eta_s|^2ds\bigg\}\Big]\bigg\|_{\i}\leq 4M_4.
%\end{equation}
%
which combining  (\ref{eq::4.27}) and (\ref{eq::4.30}) deduce that 
\begin{align*}
	&\mathop{\sup}\limits_{\bar{\tau}\in \mathscr{T}[\tau, T+K]}\bigg\|\dbE_{\bar{\tau}}\exp\bigg\{\frac{\gamma\e_0}{4n} \int_{\bar{\tau}}^{T+K}|Z_s|^2ds\bigg\}\bigg\|_{\i}\\
	&\leq M_4 \mathop{\sup}\limits_{\bar{\tau}\in \mathscr{T}[\tau, T]}\bigg\|\dbE_{\bar{\tau}}\exp\bigg\{\frac{\gamma\e_0}{2n} \int_{\bar{\tau}}^{T}|Z_s|^2ds\bigg\}\bigg\|_{\i} +M_4\\
	&\leq 2M_4\mathop{\sup}\limits_{\bar{\tau}\in \mathscr{T}[\tau, T]}\bigg\|\dbE_{\bar{\tau}}\exp\bigg\{\frac{\gamma\e_0}{2n} \int_{\bar{\tau}}^{T}|Z_s|^2ds\bigg\}\bigg\|_{\i}\\
	&\leq  6M_4\exp\Big\{3\e_0 M_3 +  C_1 T+
	3\e_0\big(2+(\sigma+\sigma_0)T\big)\|Y\|_{\mathcal{H}^\i_{\dbF}(\tau, T+K;\dbR^n)}\Big\}\\
	&\q \cdot \mathop{\sup}\limits_{\bar{\tau}\in \mathscr{T}[\tau, T+K]}\bigg\|\dbE_{\bar{\tau}}\exp\bigg\{2\e_0^{\frac{2}{1+\alpha}}(18\lambda_{0}^2 L^2T+1)\int_{\bar{\tau}}^{T+K} |Z_{s}|^{2}ds\bigg\}\bigg\|_{\i}\\
	&\leq  6M_4\exp\Big\{3\e_0 M_3 + C_1 T + 3\e_0\big(2+(\sigma+\sigma_0)T\big)\|Y\|_{\mathcal{H}^\i_{\dbF}(\tau, T+K;\dbR^n)}\Big\} \\
	&\q \cdot\mathop{\sup}\limits_{\bar{\tau}\in \mathscr{T}[\tau, T+K]}\bigg\|\dbE_{\bar{\tau}}\exp\bigg\{\frac{\gamma\e_0}{8n} \int_{\bar{\tau}}^{T+K}|Z_s|^2ds\bigg\}\bigg\|_{\i}.
\end{align*}
Hence, on one hand, by applying Jensen's inequality, we have that for any $\tau\in \mathscr{T}[T-\kappa, T]$,
\begin{equation}\label{eq::4.32}
	\begin{aligned}
		&\mathop{\sup}\limits_{\bar{\tau}\in \mathscr{T}[\tau, T+K]}\bigg\|\dbE_{\bar{\tau}}\exp\bigg\{\frac{\gamma\e_0}{4n} \int_{\bar{\tau}}^{T+K}|Z_s|^2ds\bigg\}\bigg\|_{\i}\\
		&\leq 36M_4^2\exp\Big\{6\e_0 M_3 + 2C_1 T+ 6\e_0\big(2+(\sigma+\sigma_0)T\big)\|Y\|_{\mathcal{H}^\i_{\dbF}(\tau, T+K;\dbR^n)}\Big\}.
\end{aligned}\end{equation}
On the other hand, similarly to the deduction of (\ref{eq::4.12}) and by using H\"{o}lder's inequality, we have that for any $\tau\in \mathscr{T}[T - \kappa, T]$,
\begin{equation}\label{eq::4.33}
	\begin{aligned}
		\exp(\gamma|Y_\tau|) &\leq \exp\Big\{n\gamma (M_1 + M_3) + n(\sigma + \sigma_0)\gamma\|Y\|_{\mathcal{H}^\i_{\dbF}(\tau,T+K;\dbR^n)}(T-\tau) \Big\}\\
		&\q \cdot\dbE_{\tau}\exp\bigg\{n\lambda\gamma\int_{\tau}^{T}|Z_s|^{1+\alpha}ds + n\lambda_{0}\gamma \int_{\tau}^{T}\dbE_{s}[|Z_{s+\zeta_s}|]ds\bigg\}
		\\
		&\leq \exp\Big\{n\gamma (M_1 + M_3) + n(\sigma + \sigma_0)\gamma\|Y\|_{\mathcal{H}^\i_{\dbF}(\tau,T+K;\dbR^n)}(T-\tau) \Big\}\\
		&\q \cdot\dbE_{\tau}\exp\bigg\{2n\lambda\gamma\int_{\tau}^{T}|Z_s|^{1+\alpha}ds\bigg\}\cdot
		\dbE_{\tau}\exp\bigg\{2n\lambda_{0}\gamma \int_{\tau}^{T}\dbE_{s}[|Z_{s+\zeta_s}|]ds\bigg\}.\\
	\end{aligned}
\end{equation}
For the last term in the above inequality, noting that we have previously assumed  $\|\dbE_{.}[|Z_{.+\zeta_{.}}|]\|^2_{\mathcal{Z}^2_{\dbF}(\tau,T;\dbR)} \neq 0$, and using arguments similarly to those in \rf{eq:::4.25} combined with the fact that $\kappa \leq \frac{\e_0}{8n^3\lambda_{0}^2\gamma L^2}$, 
 we conclude that for any $\tau\in \mathscr{T}[T-\kappa, T]$, 
\begin{equation}\label{eq::4.34}
	\begin{aligned}
		2n\lambda_{0}\gamma \dbE_{s}[|Z_{s+\zeta_s}|]
%	&	=\frac{2n\lambda_{0}\gamma\sqrt{T} L\dbE_{s}[|Z_{s+\zeta_s}|]}{\sqrt{\frac{\gamma\e_0}{2n}}\|\dbE_{.}[Z_{.+\zeta_{.}}]\|_{\mathcal{Z}^2_{\dbF}(t,T;\dbR^{n\times d})}} 
%	\cd  \sqrt{\frac{\gamma\e_0}{2n}}\frac{1}{L\sqrt{T}} \|\dbE_{.}[Z_{.+\zeta_{.}}]\|_{\mathcal{Z}^2_{\dbF}(t,T;\dbR^{n\times d})}  \\
%		%
%		&
		&\leq  \frac{4n^3\lambda_{0}^2\gamma \kappa L^2}{\e_0\|\dbE_{.}[|Z_{.+\zeta_{.}}|]\|^2_{\mathcal{Z}^2_{\dbF}(\tau,T;\dbR)}}\Big\{\dbE_{s}[|Z_{s+\zeta_s}|]\Big\}^2 +  \frac{\gamma\e_0}{4n\kappa} \|Z\|^{2}_{\mathcal{Z}^2_{\dbF}(\tau,T+K;\dbR^{n\times d})}\\
		&\leq \frac{1}{2\|\dbE_{.}[|Z_{.+\zeta_{.}}|]\|^2_{\mathcal{Z}^2_{\dbF}(\tau,T;\dbR)}}\Big\{\dbE_{s}[|Z_{s+\zeta_s}|]\Big\}^2 +  \frac{\gamma\e_0}{4n\kappa} \|Z\|^{2}_{\mathcal{Z}^2_{\dbF}(\tau,T+K;\dbR^{n\times d})}.
\end{aligned}\end{equation}
%
%\tc{Now, in order to better tackle the above terms, we define the following for some given time  $\bar t$ such that
%	$T-\k\les t\les \bar t \leq T$, 
%	%
%	$$\bar{P}_{\bar{t}} = \frac{2n^{\frac{3}{2}}\lambda_{0}\sqrt{\gamma T}L}{\sqrt{\e_0}}\int_{\bar{t}}^{T}\frac{\dbE_{s}[|Z_{s+\zeta_s}|]}{{\|\dbE_{.}[Z_{.+\zeta_{.}}]\|_{\mathcal{Z}^2_{\dbF}(t,T;\dbR^{n\times d})}}}dW_s.$$}
%%
%
Similarly to get \rf{eq:::4.26}, we have that for $\tau\in \mathscr{T}[T-\kappa, T]$, 
\begin{equation}\label{eq::4.300}
	\begin{aligned}
		&\ \dbE_{\tau} \exp\bigg\{\frac{1}{2\|\dbE_{.}[|Z_{.+\zeta_{.}}|]\|^2_{\mathcal{Z}^2_{\dbF}(\tau,T;\dbR)}}\int_{\tau}^{T}|\dbE_{s}[|Z_{s+\zeta_s}|]|^2ds\bigg\} 
		\leq 3.
\end{aligned}\end{equation}
Thus, by (\ref{eq::4.32}), (\ref{eq::4.300}), and  Jensen's inequality, we obtain that for $\tau\in \mathscr{T}[T-\kappa, T]$,
\begin{equation}\label{eq::4.36}
	\begin{aligned}
		&\ \dbE_{\tau}\exp\bigg\{2n\lambda_{0}\gamma \int_{\tau}^{T}\dbE_{s}[|Z_{s+\zeta_s}|]ds\bigg\}\\
		&\leq 3\exp\Big\{\frac{\gamma\e_0}{4n} \|Z\|^{2}_{\mathcal{Z}^2_{\dbF}(\tau,T+K;\dbR^{n\times d})}\Big\}\\
		&\leq 3\mathop{\sup}\limits_{\bar{\tau}\in\mathscr{T}[\tau, T+K]}\bigg\|\dbE_{\bar{\tau}}\exp\bigg\{\frac{\gamma\e_0}{4n} \int_{\bar{\tau}}^{T+K}|Z_s|^2ds\bigg\}\bigg\|_{\i}\\
		& \leq 108M_4^2\exp\Big\{6\e_0 M_3 + 2C_1 T+ 6\e_0\big(2+(\sigma+\sigma_0)T\big)\|Y\|_{\mathcal{H}^\i_{\dbF}(\tau, T+K;\dbR^n)}\Big\}.
\end{aligned}\end{equation}
Now, by letting $a = \frac{8n^2 \lambda}{\e_0}$ and $b = |Z_s|$ in  the inequality (\ref{young}), we have
\begin{align}
	2n\lambda\gamma|Z_s|^{1+\alpha} = \frac{\gamma\e_0}{4n}
	\Big\{\frac{8n^2 \lambda}{\e_0}|Z_s|^{1+\alpha}\Big\} \leq \frac{\gamma\e_0}{4n}|Z_s|^2 + C_2,
\end{align}
where
\begin{align*}
	C_2  \triangleq \frac{\gamma \e_0(1-\alpha)}{8n}\Big(\frac{1+\alpha}{2}\Big)^{\frac{1+\alpha}{1-\alpha}}\Big(\frac{8n^2 \lambda}{\e_0}\Big)^{\frac{2}{1-\alpha}}.
\end{align*}
Thus, we have that for $\tau \in \mathscr{T}[T-\kappa, T]$,
\begin{equation}\label{eq::4.37}
	\begin{aligned}
		&\ \dbE_{\tau}\exp\bigg\{2n\lambda\gamma\int_{\tau}^{T}|Z_s|^{1+\alpha}ds\bigg\} \\
		&\leq \exp\{C_2 T\}\mathop{\sup}\limits_{\bar{\tau}\in \mathscr{T}[\tau, T+K]}\bigg\|\dbE_{\bar{\tau}}\exp\bigg\{\frac{\gamma\e_0}{4n} \int_{\bar{\tau}}^{T+K}|Z_s|^2ds\bigg\}\bigg\|_{\i}\\
		&\leq 36M_4^2\exp\Big\{6\e_0 M_3 + 2C_1 T+ C_2 T + 6\e_0\big(2+(\sigma+\sigma_0)T\big)\|Y\|_{\mathcal{H}^\i_{\dbF}(\tau, T+K;\dbR^n)}\Big\}.
\end{aligned}\end{equation}
By combining (\ref{eq::4.33}), (\ref{eq::4.36}), and (\ref{eq::4.37}), we obtain that
for any $\tau\in \mathscr{T} [T-\kappa, T]$,%
\begin{align*}
	|Y_\tau| &\leq n M_1 +\ln \frac{3888 M_4^4 }{\gamma} + \frac{4C_1 T}{\gamma} + \frac{C_2 T}{\gamma}+ nM_3+\frac{12\e_0 M_3}{\gamma}  \\
	&\q+ n(\sigma + \sigma_0)\|Y\|_{\mathcal{H}^\i_{\dbF}(\tau,T+K;\dbR^n)}(T-\tau)+\frac{12\e_0\big(2+(\sigma+\sigma_0)T\big)}{\gamma} \|Y\|_{\mathcal{H}^\i_{\dbF}(\tau, T+K;\dbR^n)}.
\end{align*}
Thus, for any $t\in  [T-\kappa, T]$,
\begin{align*}
	|Y_t| &\leq n M_1 +\ln \frac{3888 M_4^4 }{\gamma} + \frac{4C_1 T}{\gamma} + \frac{C_2 T}{\gamma}+ nM_3+\frac{12\e_0 M_3}{\gamma}  \\
	&\q+ n(\sigma + \sigma_0)\|Y\|_{\mathcal{H}^\i_{\dbF}(t,T+K;\dbR^n)}(T-t)+\frac{12\e_0\big(2+(\sigma+\sigma_0)T\big)}{\gamma} \|Y\|_{\mathcal{H}^\i_{\dbF}(t, T+K;\dbR^n)}.
\end{align*}
Moreover, noting  \rf{definitione_01}, we deduce that for $t\in [T-\kappa, T]$,
\begin{align}\label{eq::4.38}
	\|Y\|_{\mathcal{H}^\i_{\dbF}(t,T+K;\dbR^n)} \leq 2n (M_1 + M_5) + 2n(\sigma + \sigma_0)(T-t)
	\|Y\|_{\mathcal{H}^\i_{\dbF}(t,T+K;\dbR^n)},
\end{align}
where 
\begin{align*}
	M_5 \triangleq \frac{1}{n}\ln \frac{3888 M_4^4 }{\gamma}   + M_3 + \frac{12\e_0 M_3}{n\gamma} + \frac{4C_1 T}{n\gamma} +  \frac{C_2 T}{n\gamma}.
\end{align*}
Finally, noting that the form of (\ref{eq::4.38})  is consistent with (\ref{eq::4.13}),  so  similarly to the discussion of  step 1, we have that
\begin{align}
	\|Y\|_{\mathcal{H}^\i_{\dbF}(T-\kappa, T+K;\dbR^n)} \leq (4n)^{[4n(\sigma + \sigma_0)T] +1}M_1 + \big\{4n + (4n)^2 + ... + (4n)^{[4n(\sigma + \sigma_0)T]+1}\big\}M_5 \triangleq Q_3.\label{eq::4.43}
\end{align}
Hence, the first  estimate of (\ref{eq:::4.13}) holds.
Moreover, for the second estimate of  (\ref{eq:::4.13}),
by integrating (\ref{eq::4.32}) and (\ref{eq::4.43}), and noting  that
\begin{align*}
	\frac{\gamma\e_0}{4n}\|Z\|^{2}_{\mathcal{Z}^2_{\dbF}(T-\kappa,T + K; \dbR^{n\times d})}\leq \mathop{\sup}\limits_{\tau\in \mathscr{T}[T-\kappa, T+K]}\bigg\|\dbE_{\tau}\exp\bigg\{\frac{\gamma\e_0}{4n} \int_\tau^{T+K}|Z_s|^2ds\bigg\}\bigg\|_{\i},
\end{align*}
we have 
\begin{align*}
	\|Z\|^{2}_{\mathcal{Z}^2_{\dbF}(T-\kappa,T + K; \dbR^{n\times d})}\leq \frac{144nM_4^2}{\gamma\e_0}\exp\Big\{6\e_0 M_3 + 2C_1 T+ 6\e_0\big(2+(\sigma+\sigma_0)T\big)Q_3\Big\}\triangleq Q_4.
\end{align*}
This completes the proof.
%\qed
%\renewcommand{\qedsymbol}{}
\end{proof}
Based on the above result, we are now able to prove \autoref{global_bdd} and \autoref{global_bdd2}. 
\begin{proof}[{ Proof of \autoref{global_bdd} and \autoref{global_bdd2}}]
First,  on the interval $[0, T+K]$,  \autoref{thm:2.2}, with $M_1$ and $M_2$  replaced by $Q_1$ and $Q_2$, respectively,  implies that we can choose a positive constant $\kappa_0$ such that $\kappa_0\leq \frac{\e_0}{8n^3\lambda_{0}^2\gamma L^2}$, $\frac{T}{\kappa_0} \in \dbN$, and  BSDE (\ref{eq::1.1}) admits a unique local solution $(Y^{(1)}, Z^{(1)})$ on the  interval $[T-\kappa_0, T+K]$. Furthermore, \autoref{lemma:4.4} implies that  
$$\|Y^{(1)}\|_{\mathcal{H}^\i_{\dbF}(T-\kappa_0,T+K)}\leq Q_1\q\hbox{and}\q\|Z^{(1)}\|^{2}_{\mathcal{Z}^2_{\dbF}(T-\kappa_0,T + K)}\leq Q_2.$$
Second, on the interval $[0, T-\kappa_0+K]$,  we let  $Y_{t} = Y^{(1)}_t$ and $Z_{t} = Z^{(1)}_t$ when $t\in[T - \kappa_0, T - \kappa_0 + K]$. Then,  \autoref{thm:2.2}  implies again that  BSDE (\ref{eq::1.1}) admits a unique local solution $(Y^{(2)}, Z^{(2)})$ on the  interval $[T-2\kappa_0, T-\kappa_0+K]$. Now,  if we  define 
	\begin{align*}
		(\tilde{Y}_t, \tilde{Z}_t) \triangleq (Y^{(1)}_t, Z^{(1)}_t) {\bf 1}_{[T-\kappa_0 +K, T+K]}(t) + (Y^{(2)}_t, Z^{(2)}_t) {\bf 1}_{[T-2\kappa_0, T-\kappa_0+K]}(t), 
	\end{align*}
	then  the pair $(\tilde{Y}, \tilde{Z})$ is the unique solution of  BSDE (\ref{eq::1.1})  on the  interval $[T-2\kappa_0, T+K]$.
	 Moreover, 	 \autoref{lemma:4.4}  implies  that $$\|\tilde{Y}\|_{\mathcal{H}^\i_{\dbF}(T-2\kappa_0,T+K)}\leq Q_1\q {\rm and}\q \|\tilde{Z}\|^{2}_{\mathcal{Z}^2_{\dbF}(T-2\kappa_0,T + K)}\leq Q_2.$$ 
	Third, by proceeding with the above process for finite times, we conclude that \autoref{global_bdd} holds for $J_1 = Q_1$ and $J_2 = Q_2.$ 
	 Finally, similarly to the deduction of \autoref{global_bdd},  we have that \autoref{global_bdd2}  also holds  for $J_3 = Q_3$ and $J_4 = Q_4$.
\end{proof}

\section{Global solution with unbounded terminal value}\label{section5}

In this section, we study global solutions of quadratic ABSDEs with unbounded terminal values. Let us consider the following type of BSDEs:
\begin{equation}\left\{\begin{aligned}\label{eq:4.1}
		&Y_t = \xi_{T} + \int_{t}^{T}f(s, Y_s, Z_s, Y_{s+\delta_s})dt - \int_{t}^{T}Z_s dW_s,  \quad \quad  t\in [0,T]; \\
		&Y_t = \xi_t, \quad \quad  Z_t = \eta_t, \quad \quad    t\in[T, T+K],
	\end{aligned}\right.\end{equation}
where the  parameter $\delta$ still satisfy the conditions in \rf{eqqq:3.2} and \rf{eq::3.2}. Note that $f(\omega, t, y, z, \phi_r): { \Om \ts [0, 	T] \ts \dbR^n \times \dbR^{n\times d}\times L^2_{\sF_r}(\Om; \dbR^n) \rightarrow L^2_{\sF_t}(\Om;\dbR^n)}$ and $f(\cdot, \cdot, y, z, \phi_r)$ is $\dbF$-progressively measurable.

\begin{assumption}\label{assumption3} \rm 
For $i=1,...,n$ and $d\dbP \times dt$-a.e.,
for all $t\in [0, T]$,  $r\in [t,T+K]$, $y, \bar{y} \in \dbR^{n}, z\in \dbR^{n\times d}, \phi, \bar{\phi}\in \mathcal{H}^2_{\dbF}(t,T + K;\dbR^{n}),$ the generator $f^{i}(t, y, z, \phi_r)$ only varies with $(t, y, z^{i}, \phi)$ and satisfies the following conditions:
\begin{itemize}
	\item [$\rm(i)$] There exist a positive constant $C$ and a non-negative process $\G$ satisfying  $\dbE\exp\{p\int_{0}^{T}\Gamma_tdt\} <+\i$ for any $ p\geq 1$ such that 
	\begin{align*}
	&	|f^{i}(t, y, z, \phi_{r})|  \leq \Gamma_t + 
		C\{|y| + \dbE_t[|\phi_{r}|]\}+  \frac{\gamma}{2}|z^i|^2,\\
	&	|f^{i}(t, y, z, \phi_{r}) - f^{i}(t, \bar{y}, z, \bar{\phi}_{r})|
		\leq C\Big\{|y - \bar{y}| + \dbE_t[|\phi_{r} - \bar{\phi}_{r}|]\Big\}.
	\end{align*}
	\item [$\rm(ii)$] The generator  $f^{i}(t, y, \cdot, \phi_r)$ is either convex or concave for any $y\in \dbR^n$ and $\phi_r \in L^2_{\sF_r}(\Om; \dbR^n)$.
	\item [$\rm(iii)$] The coefficients $\xi\in\mathcal{E}(T,T+K; \dbR^n)$ and $\eta \in \mathcal{M}(T,T+K; \dbR^{n\times d})$.
\end{itemize}
\end{assumption}

%\begin{remark}
%The assumption (A10) is valid for the generator $f,$ provided that certain components of $f$ are convex in z while others are concave in $z$.
%\end{remark}

The following theorem is the main result of this section, which implies  the existence and uniqueness of global solutions of BSDE (\ref{eq:4.1}) with unbounded terminal value.
\begin{theorem}\label{thm:5.2} \sl 
Under \autoref{assumption3},  BSDE \rf{eq:4.1} possesses a unique global solution $(Y,Z)\in \mathcal{E}(0,T+K; \dbR^{n}) \times \mathcal{M}(0,T+K; \dbR^{n\times d})$ on the whole interval $[0, T+K]$.
\end{theorem}
\begin{proof}
For simplicity of presentation, we would like to divide the proof into six steps. 
First, we shall  define a sequence of processes $\{(Y^{m}, Z^{m})\}_{m=0}^{\i}$ in the space $\mathcal{E}(0,T+K; \dbR^{n}) \times \mathcal{M}(0,T+K; \dbR^{n\times d})$. 
%%
%In step 2, we prove the sequence $\{(Y^{m}, Z^{m})\}_{m=0}^{\i}$ is uniformly bounded. 
%%
%In step 3, we estimate $\Delta_{\theta}Y^{m, p}$, $m, p \geq 1$ and $\theta \in (0, 1)$, which will be defined later.
%%
% Step 4: We prove the existence of BSDE (\ref{eq:4.1}) by using the results of step 2 and step 3. Step 5: We prove the uniqueness of BSDE (\ref{eq:4.1}).\\

\ms

\noindent {\it Step 1:  Define a sequence of processes recursively}

\ms

%%%20241217Night%%%%%%

For any given pair $(U, V) \in \mathcal{E}(0,T+K; \dbR^{n}) \times \mathcal{M}(0,T+K; \dbR^{n\times d}),$ we consider the following system of decoupled BSDEs: for each $i = 1, ..., n,$
\begin{equation}\left\{\begin{aligned}\label{eq:4.5}
		&Y_{t}^{i} = \xi^{i}_{T} + \int_{t}^{T}f^{i}(s, U_s, Z^{i}_s, U_{s+\delta_s})dt - \int_{t}^{T}Z^{i}_s dW_s,  \quad   t\in [0,T]; \\
		&Y^{i}_t = \xi^{i}_t, \quad    Z^{i}_t = \eta^{i}_t, \quad       t\in[T, T+K].
	\end{aligned}\right.\end{equation}
In view of item (i) of \autoref{assumption3}, for any $s\in[t,T]$ and $z\in \dbR^{d}$, 
\begin{equation*}
|f^{i}(s, U_s, z, U_{s+\delta_s})| \leq \Gamma_s + 
C\{|U_s| + \dbE_t[|U_{s+\delta_s}|] \}+  \frac{\gamma}{2}|z|^2.
\end{equation*}
By applying H\"{o}lder's inequality, we have that for any $q > 1$, 
\begin{align*}
&\ \dbE\exp\bigg\{q|\xi_{T}^{i}| + q\int_{0}^{T}\big(\Gamma_s + C|U_s| + C\dbE_s[|U_{s+\delta_s}|]\big)ds\bigg\}\\
%=& \dbE\Big[\exp\Big\{q|\xi_{T}^{i}|\Big\} \exp\Big\{q\int_{0}^{T}\Gamma_s + C|U_s| + C\dbE_s[|U_{s+\delta_s}|]ds\Big\}\Big]\\
&\leq \Big(\dbE\exp\big\{2q\xi^{*}\big\}\Big)^{\frac{1}{2}} \bigg(\dbE \exp\bigg\{2q\int_{0}^{T}
\big(\Gamma_s + C|U_s| + C\dbE_s[|U_{s+\delta_s}|]\big)
ds\bigg\}\bigg)^{\frac{1}{2}} \\
&\leq \Big(\dbE\exp\big\{2q\xi^{*}\big\}\Big)^{\frac{1}{2}} \bigg(\dbE \exp\bigg\{4q\int_{0}^{T}\Gamma_sds\bigg\}\bigg)^{\frac{1}{4}}
\bigg(\dbE \exp\bigg\{4q\int_{0}^{T}
\big(C|U_s| + C\dbE_s[|U_{s+\delta_s}|]\big)
ds\bigg\}\bigg)^{\frac{1}{4}},
\end{align*}
where $\xi^{*} \triangleq \mathop{\sup}\limits_{t\in[T, T+K]}|\xi_{t}|$.
%Furthermore, we have 
%\begin{align*}
%&\dbE\Big[ \exp\Big\{4q\int_{0}^{T}C|U_s| + C\dbE_s[|U_{s+\delta_s}|]ds\Big\}\Big]\leq \dbE\Big[\exp\Big\{8qCT\mathop{\sup}\limits_{t\in[0, T+K]}|U_{t}|\Big\}\Big].
%\end{align*}
Moreover, from  item (iii) of \autoref{assumption3}, along with the fact that $U \in \mathcal{E}(0,T+K; \dbR^{n}),$ we have 
\begin{align*}
\dbE\exp\bigg\{q|\xi_{T}^{i}| + q\int_{0}^{T}\big(\Gamma_s + C|U_s| + C\dbE_s[|U_{s+\delta_s}|]\big) ds\bigg\}<+\i.
\end{align*}
By combining item (iii) of \autoref{assumption3} with the results from Briand--Hu \cite[Corollary 6]{Briand and Hu}, it can be inferred that   BSDE (\ref{eq:4.5}) admits a unique solution $(Y,Z)\in \mathcal{E}(0,T; \dbR^{n}) \times \mathcal{M}(0,T; \dbR^{n\times d}).$
Additionally, since for any $p > 1$,
\begin{align*}
\dbE\exp \Big\{p\mathop{\sup}\limits_{t\in[0, T+K]}|Y_{t}|\Big\} \leq \Big(\dbE\exp \Big\{2p\mathop{\sup}\limits_{t\in[0, T]}|Y_{t}|\Big\}\Big)^{\frac{1}{2}} \Big(\dbE\exp \big\{2p\xi^{*}\big\}\Big)^{\frac{1}{2}}
\end{align*}
and 
\begin{align*}
\dbE\Big(\int_0^{T+K}|Z_s|^2ds\Big)^{\frac{p}{2}}
\leq 2^{\frac{p}{2}} \bigg\{\dbE\Big(\int_0^{T}|Z_s|^2ds\Big)^{\frac{p}{2}}+\dbE\Big(\int_T^{T+K}|\eta_s|^2ds\Big)^{\frac{p}{2}}
\bigg\},
\end{align*}
it follows that BSDE (\ref{eq:4.5}) admits a unique solution $(Y,Z)\in \mathcal{E}(0,T+K; \dbR^{n}) \times \mathcal{M}(0,T+K; \dbR^{n\times d}).$

Finally,  we define $(Y^{0}_{t}, Z^{0}_{t}) = (0, 0)$ for $t\in [0,T]$ and $(Y^{0}_{t}, Z^{0}_{t}) = (\xi_{t}, \eta_{t})$ for  $t\in [T,T+K]$. We can then recursively define a sequence of processes $\{(Y^{m}, Z^{m})\}_{m=1}^{\i}$ in the space of 
$\mathcal{E}(0,T + K; \dbR^{n}) \times \mathcal{M}(0,T+K; \dbR^{n\times d})$ by the unique solution of the following  BSDEs: for each $i = 1, ...,n$,
\begin{equation}\left\{\begin{aligned}\label{eq:4.12}
		&Y_{t}^{m+1; i} = \xi^{i}_{T} + \int_{t}^{T}f^{i}(s, Y^{m}_s, Z^{m+1,i}_s, Y^{m}_{s+\delta_s})ds
		 - \int_{t}^{T}Z^{m+1;i}_s dW_s,  \quad   t\in [0,T]; \\
		&Y^{m+1;i}_t = \xi^{i}_t, \quad    Z^{m+1;i}_t = \eta^{i}_t, \quad    t\in[T, T+K],
	\end{aligned}\right.\end{equation}
where $Y^{m; i}$ and $Z^{m; i}$ represent the $i$-th component of
$Y^{m}$ and $Z^{m}$, respectively. 
%Next, we will show that for $q \geq 1$, $\{(Y^{m}, Z^{m})\}_{m=1}^{\i}$ is a Cauchy sequence in $S^q_{\dbF}(0,T + K;\dbR^{n}) \times \mathcal{H}^2_{\dbF}(0,T+K; \dbR^{n\times d})$, and thus converges to a pair of adapted processes $(Y,Z)\in S^q_{\dbF}(0,T + K;\dbR^{n}) \times \mathcal{H}^2_{\dbF}(0,T+K; \dbR^{n\times d}).$ Furthermore, we show that $(Y,Z)\in \mathcal{E}(0,T+K; \dbR^{n}) \times \mathcal{M}(0,T+K; \dbR^{n\times d})$.
%which is the unique solution of BSDE (\ref{eq:4.1}).

\ms

\noindent{\it Step 2: The sequence  $\{Y^{m}\}_{m=0}^{\i}$ is exponential uniformly   bounded}

\ms

The aim of this step is to prove that for any $q > 1,$
\begin{align}\label{eq:4.13}
\mathop{\sup}\limits_{m\geq 0} \dbE\exp\Big\{q\gamma \mathop{\sup}\limits_{t\in [0,T+K]}|Y_{t}^{m}|\Big\} \leq C(q),
\end{align}
where 
\begin{align*}
C(q)\triangleq& \Big\{R\big(2([8nCT]+2)q\big)\Big\}^{\big([8nCT]+1\big)\big([8nCT]+2\big)} \cd \dbE\exp\Big\{16n(16n)^{[8nCT]}\big([8nCT]+2\big)^2q\gamma\xi^{*}\Big\}\nonumber\\
&\cdot\dbE\exp\bigg\{16n(32n)^{^{[8nCT]}}\big([8nCT]+2\big)^2q\gamma\int_{0}^{T}\Gamma_sds\bigg\}\vee  I(q) < +\i
%\Big(R(2^{[8nCT]+3}q)\Big)^{[8nCT]-1} 
%\dbE\Big[\exp\Big\{16n(16n)^{^{[8nCT]}}\big([8nCT]+1\big)2^{[8nCT]+2}q\gamma|\xi_{T}|\Big\}\Big]\\ &\cdot\dbE\Big[\exp\Big\{16n(32n)^{^{[8nCT]}}\big([8nCT]+1\big)2^{[8nCT]+2}q\gamma\int_{0}^{T}\Gamma_sds\Big\}\Big] \vee \sqrt{I(q)} \vee I(q) < +\i,
\end{align*}
with
\begin{align}\label{eq:4.13(1)}
	R(q) \triangleq \Big(\frac{q}{q-1}\Big)^{2q}\q\hbox{and}\q 	I(q) \triangleq R(2q)\dbE\exp\Big\{8nq\gamma\xi^{*}+8nq\gamma\int_{0}^{T}\Gamma_sds\Big\}.
\end{align}
Note that $R(q)$ is greater than 1 and is decreasing.
From item (i) of \autoref{assumption3}, we have that for $i = 1, ..., n$ and $m\geq 0$,
\begin{align}\label{eq:4.14}
|f^{i}(s, Y^{m}_s, z, Y^{m}_{s+\delta_s})| \leq \Gamma_s + C|Y^{m}_s| + C\dbE_s[|Y^{m}_{s+\delta_s}|] + \frac{\gamma}{2}|z|^2, \quad s\in [t,T].
\end{align}
In view of the fact that both $Y^{m}$ and $Y^{m+1}$ belong to $\mathcal{E}(0,T; \dbR^{n})$. Then,   by using H\"{o}lder's inequality, we derive that for any $q>1$ and $m\geq 0$,
\begin{align}\label{eq:4.15}
	\dbE\exp\bigg\{q\mathop{\sup}\limits_{s\in [0,T]}|Y_{s}^{m+1;i}| + q\int_{0}^{T}
	\big(\Gamma_r + C|Y^{m}_r| + C\dbE_r[|Y^{m}_{r+\delta_r}|]\big)
	dr\bigg\}<+\i.
\end{align}
%
%%%%20241219%%%%%%%%%%
%
 Hence,  we deduce from  \autoref{lemma:4.1} that for any  $ s\in [t, T]$,
\begin{align*}
\exp\Big\{\gamma|Y_{s}^{m+1;i}|\Big\} \leq \dbE_{s}\exp\bigg\{\gamma|\xi^{i}_{T}|+\gamma\int_{t}^{T}
\big(\Gamma_r + C|Y^{m}_r| + C\dbE_r[|Y^{m}_{r+\delta_r}|]\big)
dr\bigg\}.
\end{align*}
Moreover, it follows from Jensen's inequality that for any  $ s\in [t, T]$,
\begin{align}\label{eqqq:5.16}
\exp\Big\{\gamma|Y_{s}^{m+1}|\Big\}=&\ \exp\Bigg\{ \gamma\sqrt{\sum_{i=1}^{n}|Y_{s}^{m+1;i}|^2}\Bigg\} \leq \exp\Bigg\{ \gamma\sqrt{(\sum_{i=1}^{n}|Y_{s}^{m+1;i}|)^2}\Bigg\}= \exp\Big\{\gamma\sum_{i=1}^{n}|Y_{s}^{m+1;i}|\Big\}\nonumber\\
\leq&\ \dbE_{s}\exp\bigg\{n\gamma|\xi_{T}|+n\gamma\int_{t}^{T}
\big(\Gamma_r + C|Y^{m}_r| + C\dbE_r[|Y^{m}_{r+\delta_r}|]\big)
dr\bigg\}.
\end{align}
Note that   the last term of the above inequality is a martingale, then by using Doob's maximal inequality, H\"{o}lder's inequality, and Jensen's inequality, we obtain that
\begin{align*}
&\ \dbE\Big[\exp\Big\{q\gamma\mathop{\sup}\limits_{s\in [t,T]}|Y_{s}^{m+1}|\Big\}\Big] 
= \dbE\Big[\mathop{\sup}\limits_{s\in [t,T]}\exp\Big\{q\gamma|Y_{s}^{m+1}|\Big\}\Big]\nonumber\\
&\leq \dbE\bigg\{\mathop{\sup}\limits_{s\in [t,T]}\bigg(\dbE_{s}\exp\Big\{n\gamma|\xi_{T}|+n\gamma\int_{t}^{T}
\big(\Gamma_r + C|Y^{m}_r| + C\dbE_r[|Y^{m}_{r+\delta_r}|]\big) 
dr\Big\}\bigg)^q\bigg\}\nonumber\\
&\leq \Big(\frac{q}{q-1}\Big)^{q}\dbE\exp\Big\{nq\gamma|\xi_{T}|+nq\gamma\int_{t}^{T}
\big(\Gamma_r + C|Y^{m}_r| + C\dbE_r[|Y^{m}_{r+\delta_r}|]\big)
dr\Big\}\nonumber\\
&\leq \Big(\frac{q}{q-1}\Big)^{q}\bigg(\dbE\exp\Big\{2nq\gamma|\xi_{T}|+2nq\gamma\int_{0}^{T}\Gamma_sds\Big\}\bigg)^{\frac{1}{2}} \nonumber\\
&\q\ \cdot\bigg(\dbE\exp\Big\{2nq\gamma C\int_{t}^{T}|Y_{s}^{m}|ds + 2nq\gamma C\int_{t}^{T}\dbE_s[|Y^{m}_{s+\delta_s}|]ds\Big\}\bigg)^{\frac{1}{2}}\nonumber\\
&\leq \Big(\frac{q}{q-1}\Big)^{q}\bigg(\dbE\exp\Big\{2nq\gamma|\xi_{T}|+2nq\gamma\int_{0}^{T}\Gamma_sds\Big\}\bigg)^{\frac{1}{2}} \nonumber\\
&\q\ \cdot\bigg(\dbE\Big[\exp\Big\{4nq\gamma C\int_{t}^{T}|Y_{s}^{m}|ds\Big\}\bigg)^{\frac{1}{4}} \bigg(\dbE\exp\Big\{4nq\gamma C\int_{t}^{T}\dbE_s[|Y^{m}_{s+\delta_s}|]ds\Big\}\bigg)^{\frac{1}{4}} \nonumber\\
&\leq \Big(\frac{q}{q-1}\Big)^{q}\bigg(\dbE\exp\Big\{2nq\gamma|\xi_{T}|+2nq\gamma\int_{0}^{T}\Gamma_sds\Big\}\bigg)^{\frac{1}{2}} \nonumber\\
&\q\ \cdot\bigg(\dbE\exp\Big\{4nq\gamma C\int_{t}^{T}|Y_{s}^{m}|ds\Big\}\bigg)^{\frac{1}{4}}  \bigg(\int_{t}^{T}\dbE\dbE_s\exp\Big\{4nq\gamma C[|Y^{m}_{s+\delta_s}|]ds\Big\}\bigg)^{\frac{1}{4}} \nonumber\\
&\leq\Big(\frac{q}{q-1}\Big)^{q}\bigg(\dbE\exp\Big\{2nq\gamma\xi^{*}+2nq\gamma\int_{0}^{T}\Gamma_sds\Big\}\bigg)^{\frac{1}{2}}\bigg(\dbE\exp\Big\{4nq\gamma C\mathop{\sup}\limits_{s\in [t,T + K]}|Y_{s}^{m}|(T-t)\Big\}\bigg)^{\frac{1}{2}}.
\end{align*}
Thus 
\begin{equation}\label{eq:4.18}
\begin{aligned}
&\dbE\exp\Big\{q\gamma\mathop{\sup}\limits_{s\in [t,T+K]}|Y_{s}^{m+1}|\Big\}
\leq  \dbE\exp\Big\{q\gamma\mathop{\sup}\limits_{s\in [t,T]}|Y_{s}^{m+1}|\Big\} \cdot\exp\big\{q\gamma \xi^{*}\big\}\\
&\leq \dbE\exp\Big\{2q\gamma\mathop{\sup}\limits_{s\in [t,T]}|Y_{s}^{m+1}|\Big\}\cdot\Big(\dbE\exp\big\{2q\gamma\xi^{*}\big\}\Big)^{\frac{1}{2}}\\
& \leq \sqrt{I(q)}\bigg(\dbE\exp\Big\{8nq\gamma C\mathop{\sup}\limits_{s\in [t,T + K]}|Y_{s}^{m}|(T-t)\Big\}\bigg)^{\frac{1}{2}}.
\end{aligned}
\end{equation}
For the case of $C=0,$ it is obvious from (\ref{eq:4.18})  one has that 
\begin{equation}\label{eq:4.19}
\mathop{\sup}\limits_{m\geq 0}\dbE\exp\Big\{q\gamma\mathop{\sup}\limits_{s\in [0,T + K]}|Y_{s}^{m+1}|\Big\} \leq \sqrt{I(q)}.
\end{equation}
For the case of $C\neq 0$, we denote $\e \triangleq \frac{1}{8nC}$, and  let   $m_0$ be the unique positive integer such that
\begin{equation*}
8nCT \leq m_0 < 8nCT + 1.
\end{equation*}
If $m_0 = 1$, then  we have that $8nq\gamma C(T-t) \leq 8nq\gamma CT \leq q\gamma$.  It follows from (\ref{eq:4.18}) that 
\begin{equation*}
\dbE\exp\Big\{q\gamma\mathop{\sup}\limits_{s\in [t,T+K]}|Y_{s}^{m+1}|\Big\} \leq \sqrt{I(q)} \bigg(\dbE\exp\Big\{q\gamma \mathop{\sup}\limits_{s\in [t,T+K]}|Y_{s}^{m}|\Big\}\bigg)^{\frac{1}{2}}.
\end{equation*}
By induction, we deduce that 
\begin{equation}\label{5.11}
\dbE\exp\Big\{q\gamma\mathop{\sup}\limits_{s\in [t,T+K]}|Y_{s}^{m+1}|\Big\} \leq \sqrt{I(q)}^{1+\frac{1}{2}+\cdot\cdot\cdot +\frac{1}{2^{m}}} \bigg(\dbE\exp\Big\{q\gamma \mathop{\sup}\limits_{s\in [t,T+K]}|Y_{s}^{0}|\Big\}\bigg)^{\frac{1}{2^{m+1}}}\leq I(q).
\end{equation}
Then,  by H\"{o}lder's inequality, we have that 
\begin{equation}\label{eq:4.23}
	\mathop{\sup}\limits_{m\geq 0}\dbE\exp\Big\{q\gamma\mathop{\sup}\limits_{s\in [0,T+K]}|Y_{s}^{m}|\Big\} \leq R(2q)\dbE\exp\big\{16nq\gamma\xi^{*}\big\}\cdot \dbE\exp\bigg\{16nq\gamma\int_{0}^{T}\Gamma_sds\bigg\}.
\end{equation}
If $m_{0}\neq 1$, then  $8nq\gamma C(T-t) \leq 8nq\gamma C\e = q\gamma$ when $t\in [T -\e, T]$.  It follows from the similar deduction of (\ref{eq:4.23}) that for all $q>1$,
\begin{align}\label{eqq:5.14}
	\mathop{\sup}\limits_{m\geq 0}\dbE\exp\Big\{q\gamma\mathop{\sup}\limits_{s\in [T-\e,T+K]}|Y_{s}^{m}|\Big\} \leq R(2q)\dbE\exp\big\{16nq\gamma\xi^{*}\big\}\cdot \dbE\exp\bigg\{16nq\gamma\int_{0}^{T}\Gamma_sds\bigg\},
\end{align}
which leads to 
\begin{equation}\label{eq:4.25}
\mathop{\sup}\limits_{m\geq 0}\dbE\exp\Big\{q\gamma|Y_{T-\e}^{m}|\Big\} \leq R(2q)\dbE\exp\big\{16nq\gamma\xi^{*}\big\}\cdot \dbE\exp\bigg\{16nq\gamma\int_{0}^{T}\Gamma_sds\bigg\}.
\end{equation}
Next, in order to obtain the estimate when  $t\in[0,T-\e]$, we consider the following BSDEs: for $i = 1, ..., n,$
\begin{equation}\left\{\begin{aligned}\label{BSDE5.15}
		&\tilde{Y}_{t}^{m+1; i} = Y^{m+1;i}_{T-\e} + \int_{t}^{T-\e}f^{i}(s, Y^{m}_s, \tilde{Z}^{m+1,i}_s, Y^{m}_{s+\delta_s})ds
		- \int_{t}^{T-\e}\tilde{Z}^{m+1;i}_s dW_s; \quad  t\in [0,T-\e],\\
		&\tilde{Y}_{t}^{m+1; i} = Y_{t}^{m+1; i}, \quad \quad\quad \tilde{Z}_{t}^{m+1; i} = Z_{t}^{m+1; i}, \quad       t\in[T - \e, T-\e+K].
	\end{aligned}\right.\end{equation}
Then, similarly to the deduction of (\ref{eq:4.18}),  and noting that the pair $(Y^{m+1; i}, Z^{m+1; i})$ is a solution of BSDE (\ref{BSDE5.15}), we obtain that for all $t\in [0,T-\e],$
\begin{equation*}
	\begin{aligned}
		\dbE\exp\Big\{q\gamma\mathop{\sup}\limits_{s\in [t,T-\e+K]}|Y_{s}^{m+1}|\Big\} 
		\leq \sqrt{I(q)} \bigg(\dbE\exp\Big\{8nq\gamma C\mathop{\sup}\limits_{s\in [t,T-\e+K]}|Y_{s}^{m}|(T-\e-t)\Big\}\bigg)^{\frac{1}{2}}.
	\end{aligned}
\end{equation*}
Then, similarly to the deduction of (\ref{eqq:5.14}), we have
\begin{equation*}
\begin{aligned}
& \mathop{\sup}\limits_{m\geq 0}\dbE\exp\Big\{q\gamma\mathop{\sup}\limits_{s\in [T-2\e,T-\e+K]}|Y_{s}^{m}|\Big\} \\
&\leq R(2q)\dbE\exp\Big\{16nq\gamma|Y_{T-\e}^{m}|\Big\} \cdot\dbE\exp\bigg\{16nq\gamma\int_{0}^{T-\e}\Gamma_sds\bigg\}\\
&\leq R(2q)R(32nq)\dbE\exp\big\{16^2n^{2}q\gamma\xi^{*}\big\} \cdot\bigg(\dbE\exp\Big\{16^2n^{2}q\gamma\int_{0}^{T}\Gamma_sds\Big\}\bigg)^2\\
&\leq R(2q)R(32nq)\dbE\exp\big\{16^2n^{2}q\gamma\xi^{*}\big\}\cdot \dbE\exp\bigg\{16n\times32 nq\gamma\int_{0}^{T}\Gamma_sds\bigg\}.
\end{aligned}
\end{equation*}
By induction, for each  $j = 1, ...,m_{0}-1$,
\begin{align}\label{eq::5.16}
&\mathop{\sup}\limits_{m\geq 0}\dbE\exp\Big\{q\gamma\mathop{\sup}\limits_{s\in [T-j\e,T -(j-1)\e+ K]}|Y_{s}^{m}|\Big\}\nonumber\\ 
&\leq R(2q)\prod_{1\leq i \leq j-1} R(2\times16^{i}n^iq)\dbE\exp\Big\{16n(16n)^{^{j-1}}q\gamma\xi^{*}\Big\}\cdot\dbE\exp\bigg\{16n(32n)^{^{j-1}}q\gamma\int_{0}^{T}\Gamma_sds\bigg\}\nonumber\\
&\leq \Big(R(2q)\Big)^{m_0 }\dbE\exp\Big\{16n(16n)^{^{m_0-1}}q\gamma\xi^{*}\Big\}\cdot\dbE\exp\bigg\{16n(32n)^{^{m_0-1}}q\gamma\int_{0}^{T}\Gamma_sds\bigg\}.
\end{align}
Moreover, we consider the following  BSDEs: for $i = 1, ..., n,$
\begin{equation}\left\{\begin{aligned}\label{BSDE 5.16}
		&\tilde{Y}_{t}^{m+1; i} = Y^{m+1;i}_{T-(m_{0}-1)\e} + \int_{t}^{T-(m_{0}-1)\e}f^{i}(s, Y^{m}_s, \tilde{Z}^{m+1,i}_s, Y^{m}_{s+\delta_s})ds
		- \int_{t}^{T-(m_{0}-1)\e}\tilde{Z}^{m+1;i}_s dW_s; \quad \quad \\
		&\q\q\q\q\q\q\q\q\q\q\q\q\q\q\q\q\q\q\q\q\q\q\q\q\q\q\q\q\q\q\q\q t\in [0,T-(m_{0}-1)\e],\\
		&\tilde{Y}_{t}^{m+1; i} = Y_{t}^{m+1; i}, \quad \tilde{Z}_{t}^{m+1; i} = Z_{t}^{m+1; i}, \quad      t\in[ T-(m_{0}-1)\e, T-(m_{0}-1)\e+K].
	\end{aligned}\right.\end{equation}
Since
$
8nCT\leq m_{0} < 8nCT + 1, 
$
we have
\begin{equation*}
	8nC\big[T-(m_{0}-1)\e\big]\leq 1 < 8nC\big[T-(m_{0}-1)\e\big] + 1.
\end{equation*}
Then,  when $t\in [0, T-(m_{0}-1)\e]$, we have $8nq\gamma C\big(T-(m_{0}-1)\e-t\big) \leq 8nq\gamma C\big(T-(m_{0}-1)\e\big) \leq q\gamma.$ Thus, it follows from the similar deduction of (\ref{eq:4.23})  that 
\begin{align}\label{eq::5.17}
	&\mathop{\sup}\limits_{m\geq 0}\dbE\exp\Big\{q\gamma\mathop{\sup}\limits_{s\in [0,T-(m_{0}-1)\e+K]}|Y_{s}^{m}|\Big\} \nonumber\\
	&\leq
	R(2q)\dbE\exp\big\{16nq\gamma|\xi_{T}|\big\} \cdot\dbE\exp\bigg\{16nq\gamma\int_{0}^{T}\Gamma_sds\bigg\}\nonumber\\
	&\leq \Big\{R(2q)\Big\}^{m_0 } \dbE\exp\Big\{16n(16n)^{^{m_0-1}}q\gamma\xi^{*}\Big\}\cdot\dbE\exp\bigg\{16n(32n)^{^{m_0-1}}q\gamma\int_{0}^{T}\Gamma_sds\bigg\}.
\end{align}
Therefore, by H\"{o}lder's inequality, (\ref{eq::5.16}), and (\ref{eq::5.17}), we have
\begin{align*}
&\mathop{\sup}\limits_{m\geq 0}\dbE\exp\Big\{q\gamma\mathop{\sup}\limits_{s\in [0,T+K]}|Y_{s}^{m}|\Big\} \nonumber\\
&\leq \dbE\exp\big\{(m_0 + 1)q\gamma\xi^{*}\big\} \cdot \mathop{\sup}\limits_{m\geq 0}\Bigg\{\dbE\exp\Big\{(m_0 + 1)q\gamma\mathop{\sup}\limits_{s\in [0,T-(m_{0}-1)\e]}|Y_{s}^{m}|\Big\}\nonumber\\
&\q\ \cd \prod_{j=1}^{m_{0}-1}\dbE\exp\Big\{(m_0+1)q\gamma\mathop{\sup}\limits_{s\in [T-j\e,T-(j-1)\e]}|Y_{s}^{m}|\Big\}\Bigg\}\nonumber\\
&\leq \Big(R\big(2(m_0+1)q\big)\Big)^{m_0(m_0+1)} \dbE\exp\Big\{16n(16n)^{^{m_0-1}}(m_0+1)^2q\gamma\xi^{*}\Big\}\nonumber\\
&\q\ \cdot\dbE\exp\bigg\{16n(32n)^{^{m_0-1}}(m_0+1)^2q\gamma\int_{0}^{T}\Gamma_sds\bigg\}.
\end{align*}
This together with (\ref{eq:4.19}) and (\ref{5.11}) yields the estimate (\ref{eq:4.13}).

\ms

\noindent\emph{Step 3: The sequence $\{Z^{m}\}_{m=0}^{\i}$ is uniformly bounded in $\mathcal{H}^q_{\dbF}(0,T + K;\dbR^{n\times d})$}

\ms

 In this step we  prove  that for any $q > 1,$
\begin{equation}\label{eq:4.38}
\mathop{\sup}\limits_{m\geq 0}\dbE\Big[\Big(\int_{0}^{T+K}|Z_{s}^{m}|^2ds\Big)^{\frac{q}{2}}\Big] < +\i.
\end{equation}
%In fact, it suffice to show that \rf{eq:4.38} holds when $q>2$, since we have that when $1<q \leq 2,$ 
%\begin{align*}
%\mathop{\sup}\limits_{m\geq 0}\dbE\Big[\Big(\int_{0}^{T}|Z_{s}^{m}|^2ds\Big)^{\frac{q}{2}}\Big] \leq \mathop{\sup}\limits_{m\geq 0}\dbE\Big[\Big(\int_{0}^{T}|Z_{s}^{m}|^2ds\Big)^{\frac{q}{2}+2}\Big] + 1.
%\end{align*} 
Actually, by applying It\^{o}-Tanaka's formula to the term $\exp\{2\gamma|Y_{t}^{m+1;i}|\}$ firstly and then by using inequality (\ref{eq:4.14}), we obtain that for $i = 1, ..., n$ and $m\geq 0,$
\begin{equation*}
	\begin{aligned}
		&\exp\{2\gamma |Y_{0}^{m+1;i}|\} + \gamma^2 \int_{0}^{T}\exp\{2\gamma |Y_{s}^{m+1;i}|\}|Z_{s}^{m+1;i}|^2 ds\\
		&\leq \exp\{2\gamma |\xi_{T}^i|\} + 2\gamma\int_{0}^{T} \exp\{2\gamma |Y_{s}^{m+1;i}|\}(\Gamma_s + C|Y_{s}^{m}| + C\dbE_{s}[|Y_{s+\delta_s}^{m}|])ds\\
		&\quad-2\gamma\int_{0}^{T}\exp\{2\gamma |Y_{s}^{m+1;i}|\}\sgn(Y_{s}^{m+1;i})Z_{s}^{m+1;i}dW_s.
	\end{aligned}
\end{equation*}
Then, noting the inequality $(x + y)^{{\frac{q}{2}}} \leq 2^{\frac{q}{2}}(x^{\frac{q}{2}}+y^{\frac{q}{2}})$ for $ x, y \geq 0$, we have 
\begin{align*}
	&\ \gamma^q\Big(\int_{0}^{T}|Z_{s}^{m+1;i}|^2 ds\Big)^{\frac{q}{2}}\\
	&\leq 2^{\frac{q}{2}} \Bigg\{\Big(\exp\{2\gamma \xi^{*}\} + 2\gamma\int_{0}^{T} \exp\{2\gamma |Y_{s}^{m+1;i}|\}(\Gamma_s + C|Y_{s}^{m}| + C\dbE_{s}[|Y_{s+\delta_s}^{m}|])ds\Big)^{\frac{q}{2}}\\
	&\q+2^{\frac{q}{2}}\gamma^{\frac{q}{2}}\mathop{\sup}\limits_{t\in[0,T]}\Big|\int_{0}^{t}\exp\{2\gamma |Y_{s}^{m+1;i}|\}\sgn(Y_{s}^{m+1;i})Z_{s}^{m+1;i}dW_s\Big|^{\frac{q}{2}}\Bigg\}.
\end{align*}
By taking the expectation on both sides of the above inequality, we have
\begin{align*}
	&\ \dbE\Big[\Big(\int_{0}^{T}|Z_{s}^{m+1;i}|^2 ds\Big)^{\frac{q}{2}}\Big]\\
	&\leq 2^{\frac{q}{2}}\gamma^{-q}\mathop{\sup}\limits_{m\geq0}\dbE\Big[\Big(\exp\{2\gamma \xi^{*}\} + 2\gamma\int_{0}^{T} \exp\{2\gamma |Y_{s}^{m+1;i}|\}(\Gamma_s + C|Y_{s}^{m}| + C\dbE_{s}[|Y_{s+\delta_s}^{m}|])ds\Big)^{\frac{q}{2}}\Big]\\
	&\q\ +2^{q}\gamma^{-\frac{q}{2}}\dbE\Big[\mathop{\sup}\limits_{t\in[0,T]}\Big|\int_{0}^{t}\exp\{2\gamma |Y_{s}^{m+1;i}|\}\sgn(Y_{s}^{m+1;i})Z_{s}^{m+1;i}dW_s\Big|^{\frac{q}{2}}\Big].
	%&\leq 2^{q+1}\gamma^{-q}\mathop{\sup}\limits_{m\geq0}\dbE\Big[\Big(\exp\{2\gamma \xi^{*}\} + \int_{0}^{T} \exp\{2\gamma |Y_{s}^{m+1;i}|\}(\Gamma_s + C|Y_{s}^{m}| + C\dbE_{s}[|Y_{s+\delta_s}^{m}|])ds\Big)^{\frac{q}{2}}\Big]\\
	%&\quad+ 2^{q+1}\gamma^{-q}\dbE\Big(\mathop{\sup}\limits_{t\in[0,T]}\Big|\int_{0}^{t}\exp\{2\gamma |Y_{s}^{m+1;i}|\}\sgn(Y_{s}^{m+1;i})Z_{s}^{m+1;i}dW_s\Big|^{\frac{q}{2}}\Big)\\
\end{align*}
For the last term, by Burkholder-Davis-Gundy's inequality, H\"{o}lder's inequality, and Jensen's inequality,  there exists a positive constant $M(q)$ such that
\begin{align*}
	& 2^{q}\gamma^{-\frac{q}{2}}\dbE\Big[\mathop{\sup}\limits_{t\in[0,T]}\Big|\int_{0}^{t}\exp\{2\gamma |Y_{s}^{m+1;i}|\}\sgn(Y_{s}^{m+1;i})Z_{s}^{m+1;i}dW_s\Big|^{\frac{q}{2}}\Big]\\
	&\leq 2^{q}\gamma^{-\frac{q}{2}}M(q)\dbE\Big[\Big|\int_{0}^{T}\exp\{4\gamma |Y_{s}^{m+1}|\}|Z_{s}^{m+1;i}|^2ds\Big|^{\frac{q}{4}}\Big]\\
    &\leq 2^{q}\gamma^{-\frac{q}{2}}M(q) \dbE\Big[\mathop{\sup}\limits_{t\in[0,T]}\exp\{q\gamma |Y_{t}^{m+1}|\}\Big|\int_{0}^{T}|Z_{s}^{m+1;i}|^2ds\Big|^{\frac{q}{4}}\Big]\\
    &\leq 2^{q}\gamma^{-\frac{q}{2}}M(q) \dbE\Big[\mathop{\sup}\limits_{t\in[0,T]}\exp\{2q\gamma |Y_{t}^{m+1}|\}\Big]\sqrt{\dbE\Big[\Big(\int_{0}^{T}|Z_{s}^{m+1;i}|^2ds\Big)^{\frac{q}{2}}\Big]}.
\end{align*}
Therefore
\begin{equation}\label{eq:4.43}
	\begin{aligned}
	\dbE\Big[\Big(\int_{0}^{T}|Z_{s}^{m+1;i}|^2 ds\Big)^{\frac{q}{2}}\Big] \leq A(q) + B(q)\sqrt{\dbE\Big[\Big(\int_{0}^{T}|Z_{s}^{m+1;i}|^2ds\Big)^{\frac{q}{2}}\Big]},
	\end{aligned}
\end{equation}
where 
\begin{equation*}
\begin{aligned}
A(q) = 2^{\frac{q}{2}}\gamma^{-q}& \mathop{\sup}\limits_{m\geq0} \dbE\Big[\Big(\exp\{2\gamma \xi^{*}\} + 2\gamma\int_{0}^{T} \exp\{2\gamma |Y_{s}^{m+1;i}|\}\\
&\cdot(\Gamma_s + C|Y_{s}^{m}| + C\dbE_{s}[|Y_{s+\delta_s}^{m}|])ds\Big)^{\frac{q}{2}}\Big]<+\i
\end{aligned}
\end{equation*}
and 
\begin{equation*}
B(q) = 2^{q}\gamma^{-\frac{q}{2}}M(q)\mathop{\sup}\limits_{m\geq0} \dbE\Big[\mathop{\sup}\limits_{t\in[0,T]}\exp\{2q\gamma |Y_{t}^{m+1}|\}\Big]<+\i.
\end{equation*}
Next, on one hand, it is easy to check that the solution of the following inequality  
\begin{equation}\label{241220}
x^2 \leq A(q) + B(q)x,\q x\ges 0, 
\end{equation}
satisfies that 
%with $x\geq 0$, it is easy to check that 
%
$$ 0 \leq x \leq \frac{B(q) + \sqrt{|B(q)|^2 + 4A(q)}}{2}.$$
On the other hand, 
from (\ref{eq:4.43}),   it is easy to see that the following term  $$\sqrt{\dbE\Big[\Big(\int_{0}^{T}|Z_{s}^{m+1;i}|^2ds\Big)^{\frac{q}{2}}\Big]}$$ is a solution of the inequality \rf{241220}. 
Thus, we  obtain that for all $m\geq 0$,
$$\dbE\Big[\Big(\int_{0}^{T}|Z_{s}^{m}|^2ds\Big)^{\frac{q}{2}}\Big] \leq n^{\frac{q}{2}} \bigg\{\frac{B(q) + \sqrt{|B(q)|^2 + 4A(q)}}{2}\bigg\}^2<+\i.$$
Hence for $q>1$, we have
\begin{equation*}
	\mathop{\sup}\limits_{m\geq 0}\dbE\Big[\Big(\int_0^{T+K}|Z_{s}^{m}|^2ds\Big)^{\frac{q}{2}}\Big]\leq 2^{\frac{q}{2}}\mathop{\sup}\limits_{m\geq 0} \Bigg\{\dbE\Big[\Big(\int_0^{T}|Z_{s}^{m}|^2ds\Big)^{\frac{q}{2}}\Big]+\dbE\Big[\Big(\int_T^{T+K}|\eta_s|^2ds\Big)^{\frac{q}{2}}\Big]\Bigg\}<+\i,
\end{equation*}
which implies that \rf{eq:4.38} holds.

\ms

\noindent{\it Step 4:  Estimation of the term $\Delta_{\theta}Y^{m, p}$}

\ms

In what follows, without loss of generality, we suppose that the generator $f$ is componentwise convex. In other words,  for any $i = 1, ..., n$ and $t\in [0, T]$, the component $f^{i}( t, y, \cdot, \phi_r)$ is convex  for any $y\in \dbR^n$ and $\phi_r \in L^2_{\sF_r}(\Om; \dbR^n)$.  Moreover, for simplicity of presentation, we denote that
\begin{equation}\label{eqqq::5.43}
\Delta_{\theta}Y^{m, p} \triangleq \frac{Y^{m+p}-\theta Y^{m}}{1-\theta}
\quad \hb{and} \quad \Delta_{\theta}Z^{m, p} \triangleq \frac{Z^{m+p}-\theta Z^{m}}{1-\theta}, \q m\ges 1, p \geq 1, \theta \in (0, 1).
\end{equation}
The aim of this step is to give an exponential estimate of the term $\Delta_{\theta}Y^{m, p}.$

We observe that the pair  $(\Delta_{\theta}Y^{m, p}, \Delta_{\theta}Z^{m, p})$ belongs to  $\mathcal{E}(0,T+K; \dbR^{n}) \times \mathcal{M}(0,T+K; \dbR^{n\times d})$ and satisfies the following  BSDEs: for $i = 1, ..., n,$ 
\begin{equation}\left\{\begin{aligned}\label{eq:4.42}
	&\Delta_{\theta}Y_{t}^{m, p; i} = \xi^{i}_{T} + \int_{t}^{T}\Delta_{\theta}f^{m,p;i}(s, \Delta_{\theta}Z_{s}^{m,p;i})ds - \int_{t}^{T}\Delta_{\theta}Z_{s}^{m,p;i}dW_{s}, \quad t\in[0,T],\\
	&\Delta_{\theta}Y_{t}^{m, p; i} = \xi_{t}^{i}, \quad \Delta_{\theta}Z_{t}^{m, p; i} = \eta_{t}^{i}, \quad t\in [T, T+K],
	\end{aligned}\right.\end{equation}
where 
\begin{equation*}
\Delta_{\theta}f^{m,p;i}(s, z) \triangleq \frac{1}{1-\theta} \Big\{f^{i}\big(s, Y^{m+p-1}_{s}, (1-\theta)z + \theta Z_{s}^{m;i}, Y^{m+p-1}_{s+\delta_s}\big) - \theta f^{i}\big(s, Y^{m-1}_{s}, Z_{s}^{m;i}, Y^{m-1}_{s+\delta_s}\big)\Big\}.
\end{equation*}
By  \autoref{assumption3}, we have that  for all $s\in[t,T]$ and $z\in \dbR^{d},$
\begin{equation*}
\begin{aligned}
\Delta_{\theta}f^{m,p;i}(s, z)
&\leq \Big|f^{i}(s, Y^{m+p-1}_{s}, z , Y^{m+p-1}_{s+\delta_s})\Big| \\
&\q\  + \frac{\theta}{1-\theta}\Big|f^{i}(s, Y^{m+p-1}_{s}, Z_{s}^{m;i} , Y^{m+p-1}_{s+\delta_s}) 
- f^{i}(s, Y^{m-1}_{s}, Z_{s}^{m;i} , Y^{m-1}_{s+\delta_s})\Big|\\
&\leq \Gamma_s + C\Big\{ |Y_{s}^{m+p-1}| +\dbE_{s}|Y_{s+\delta_s}^{m+p-1}|+ |Y_{s}^{m-1}| +\dbE_{s} |Y_{s+\delta_s}^{m-1}| \Big\} \\
&\q\ + C\Big\{ |\Delta_{\theta}Y^{m-1,p}_{s}| + \dbE_{s}|\Delta_{\theta}Y^{m-1,p}_{s+\delta_s}|\Big\} + \frac{\gamma}{2}|z|^2.                                        
\end{aligned}
\end{equation*}
Then, from  (\ref{eq:4.13}) and Jensen's inequality, it is easy to check that 
%\begin{equation}
\begin{align}
&\dbE\exp \bigg\{2\gamma\mathop{\sup}\limits_{s\in [t,T]}(\Delta_{\theta}Y^{m, p;i}_s)^{+} +2\gamma\int_{t}^{T} \Big[ \Gamma_r + C\big\{|Y_{r}^{m+p-1}| +\dbE_{r} |Y_{r+\delta_r}^{m+p-1}| + |Y_{r}^{m-1}| +\dbE_{r}|Y_{r+\delta_r}^{m-1}|\big\} \nonumber \\ 
&+ C\big\{|\Delta_{\theta}Y^{m-1,p}_{r}| + \dbE_{r}|\Delta_{\theta}Y^{m-1,p}_{r+\delta_r}|\big\}\Big] dr\bigg\}
<\i,\q i = 1, ..., n, \q t\in[0,T]. \label{eq:4.46}
\end{align}
%\end{equation}
%
%20241220Afternoom%%%%%%
%
Thus,  all conditions in \autoref{lemma:4.1} are satisfied by BSDE (\ref{eq:4.42}),  from which it follows   that for each $i = 1, ..., n$,
\begin{equation}\label{eqqq:5.48}
\begin{aligned}
&\ \exp\Big\{\gamma(\Delta_{\theta}Y^{m, p;i}_s)^{+}\Big\}\\
&\leq \dbE_{s}\exp\bigg\{\gamma(\xi_{T})^{+} + \gamma \int_{t}^{T}
\Big[\Gamma_r + C\big(|Y_{r}^{m+p-1}| +\dbE_{r}|Y_{r+\delta_r}^{m+p-1}| +|Y^{m-1,p}_{r}| + \dbE_{r}|Y^{m-1,p}_{r+\delta_r}|\big)\\
&\q+ C\big(|\Delta_{\theta}Y^{m-1,p}_{r}| + \dbE_{r} |\Delta_{\theta}Y^{m-1,p}_{r+\delta_r}|\big)\Big]
dr\bigg\}, \q s\in [t,T].																	
\end{aligned}
\end{equation}
For simplicity, next we denote 
\begin{equation*}
	\Delta_{\theta}\bar{Y}^{m, p} \triangleq \frac{Y^{m}-\theta Y^{m+p}}{1-\theta}
	\quad \hb{and} \quad \Delta_{\theta}\bar{Z}^{m, p} \triangleq \frac{Z^{m}-\theta Z^{m+p}}{1-\theta}.
\end{equation*}
Then, similarly to the deduction of (\ref{eqqq:5.48}), we have that for all $i = 1, ..., n$,
\begin{equation}\label{eqq::5.50}
	\begin{aligned}
		&\ \exp\Big\{\gamma\big(|\Delta_{\theta}\bar{Y}^{m, p;i}_s\big)^{+}\Big\}\\
		&\leq \dbE_{s}\exp\bigg\{\gamma(\xi_{T})^{+} + \gamma \int_{t}^{T}
	\Big[	\Gamma_r + C\big(|Y_{r}^{m+p-1}| +\dbE_{r}|Y_{r+\delta_r}^{m+p-1}| +|Y^{m-1,p}_{r}| + \dbE_{r}|Y^{m-1,p}_{r+\delta_r}|\big)\\
		&\q+ C\big(|\Delta_{\theta}\bar{Y}^{m-1,p}_{r}| + \dbE_{r}|\Delta_{\theta}\bar{Y}^{m-1,p}_{r+\delta_r}|\big)\Big]
		dr\bigg\}, \q s\in [t, T].										
\end{aligned}
\end{equation}
By using the inequality  $(a + b)^{-} \leq a^{-} + b^{-}$  and $(-a)^{+} = a^{-}$ for   $a, b \in \dbR$, we deduce that
\begin{equation}\label{eqq::5.51}
(\Delta_{\theta}Y^{m, p;i}_s)^{-} \leq (\Delta_{\theta}\bar{Y}^{m, p;i}_s)^{+} + 2|Y_{s}^{m+p}| \q\hb{and}\q (\Delta_{\theta}\bar{Y}^{m, p;i}_t)^{-} \leq  (\Delta_{\theta}Y^{m, p;i}_s)^{+} + 2|Y_{s}^{m}|.
\end{equation}
Now, by combining the above inequalities and by using Jensen's inequality, we have for each $i = 1, ..., n$ and  for any $0\les t\les s\les T$,
\begin{equation*}
	\begin{aligned}
		&\ \exp\Big\{\gamma|\Delta_{\theta}Y^{m, p;i}_s|\Big\}=\exp\Big\{\gamma(\Delta_{\theta}Y^{m, p;i}_s)^{+}\Big\}\cdot\exp\Big\{\gamma(\Delta_{\theta}Y^{m, p;i}_s)^{-}\Big\}\\
		&\leq \dbE_{s}\exp\bigg\{2\gamma \xi^{*} 	+ 2\gamma|Y_{s}^{m+p}| + 2\gamma 
		\int_{t}^{T}
	\Big[	\Gamma_r + C\big(|Y_{r}^{m+p-1}| +\dbE_{r}|Y_{r+\delta_r}^{m+p-1}| +|Y^{m-1,p}_{r}|\\
		&\quad  + \dbE_{r}|Y^{m-1,p}_{r+\delta_r}|\big)+ C\big(|\Delta_{\theta}Y^{m-1,p}_{r}| + \dbE_{r}|\Delta_{\theta}Y^{m-1,p}_{r+\delta_r}|+|\Delta_{\theta}\bar{Y}^{m-1,p}_{r}| + \dbE_{r}|\Delta_{\theta}\bar{Y}^{m-1,p}_{r+\delta_r}|\big)\Big] 
		dr\bigg\}									
	\end{aligned}
\end{equation*}
and
\begin{equation}\label{eq:::5.53}
	\begin{aligned}
		&\exp\Big\{\gamma|\Delta_{\theta}\bar{Y}^{m, p;i}_s|\Big\}=\exp\Big\{\gamma(\Delta_{\theta}\bar{Y}^{m, p;i}_s)^{+}\Big\}\cdot\exp\Big\{\gamma(\Delta_{\theta}\bar{Y}^{m, p;i}_s)^{-}\Big\}\\
		&\leq \dbE_{s}\exp\bigg\{2\gamma \xi^{*} + 2\gamma|Y_{s}^{m}| 
+2\gamma	\int_{t}^{T}
\Big[	\Gamma_r + C\big(|Y_{r}^{m+p-1}| +\dbE_{r}|Y_{r+\delta_r}^{m+p-1}| +|Y^{m-1,p}_{r}|\\
&\quad  + \dbE_{r}|Y^{m-1,p}_{r+\delta_r}|\big)+ C\big(|\Delta_{\theta}Y^{m-1,p}_{r}| + \dbE_{r}|\Delta_{\theta}Y^{m-1,p}_{r+\delta_r}|+|\Delta_{\theta}\bar{Y}^{m-1,p}_{r}| + \dbE_{r}|\Delta_{\theta}\bar{Y}^{m-1,p}_{r+\delta_r}|\big)\Big] 
dr\bigg\}											.								
	\end{aligned}
\end{equation}
Consequently, it follows from Jensen's inequality again that   for any $0\les t\les s\les T$,
\begin{equation}\label{eqq::5.54}
	\begin{aligned}
		&\ \exp\Big\{\gamma\Big(|\Delta_{\theta}Y^{m, p}_s|+|\Delta_{\theta}\bar{Y}^{m, p}_s|\Big)\Big\}\\
		&\leq \dbE_{s}\mathop{\sup}\limits_{\bar{r}\in [t, T]}\exp\Big\{4n\gamma \big(\xi^{*} + |Y_{\bar{r}}^{m}| + |Y_{\bar{r}}^{m+p}|\big)\Big\} + \exp \bigg\{4n\gamma\int_{t}^{T}
	\Big[	\Gamma_r + C\(|Y_{r}^{m+p-1}|+\dbE_{r}|Y_{r+\delta_r}^{m+p-1}| \\
		&\q   +|Y^{m-1,p}_{r}|+ \dbE_{r}|Y^{m-1,p}_{r+\delta_r}|+|\Delta_{\theta}Y^{m-1,p}_{r}| + \dbE_{r}|\Delta_{\theta}Y^{m-1,p}_{r+\delta_r}|+|\Delta_{\theta}\bar{Y}^{m-1,p}_{r}|+\dbE_{r}|\Delta_{\theta}\bar{Y}^{m-1,p}_{r+\delta_r}|\)\Big] dr\bigg\}. \\									
\end{aligned}
\end{equation}
%, in view of  (\ref{eqq::5.54}), applying Doob's maximal inequality, as well as H\"{o}lder's and Jensen's inequality, we get that, for every $q>1$ and $t\in [0, T],$
%\begin{equation}\label{eq:4.47}
%	\begin{aligned}
%		\dbE\Big[\exp\Big\{q\gamma\mathop{\sup}\limits_{s\in [t,T+K]}|\Delta_{\theta}Y_{s}^{m,p}|\Big\}\Big] 
%		\leq \sqrt{\tilde{I}(q)}\cdot \Big(\dbE\Big[\exp\Big\{8nq\gamma C\mathop{\sup}\limits_{s\in [t,T+K]}|\Delta_{\theta}Y_{s}^{m-1
%			,p}|(T-t)\Big\}\Big]\Big)^{\frac{1}{2}},
%	\end{aligned}
%\end{equation}
%where
%\begin{equation*}
%\begin{aligned}
%\tilde{I}(q) = R(2q)\dbE\Big[\exp\Big\{2q\gamma\xi^{*}\Big\}\Big]\mathop{\sup}\limits_{m,p\geq1} &\dbE\Big[\exp\Big\{4nq\gamma\xi^{*}\\
% +&8nq\gamma\mathop{\sup}\limits_{s\in[t,T+K]}|Y_{s}^{m+p-1}|T + 4nq\gamma\int_{0}^{T}\Gamma_sds \Big\}\Big],
%\end{aligned}
%\end{equation*}
%with $R(q)$ defined as specified in (\ref{eq:4.13(1)}).
%By comparing (\ref{eq:4.18}) and (\ref{eq:4.47}), we see that, using an iteration w.r.t. m, a similarly argument as that for (\ref{eq:4.13}) 
%
Finally, by comparing (\ref{eqqq:5.16}) and (\ref{eqq::5.54}), and by using   a similar deduction to the previous one, we conclude that there exists a positive constant $C^{*}(q)$,  depending only on $q$, such that for all $q > 1,$
\begin{align}\label{eq:4.48}
\dbE\exp\Big\{q\gamma\mathop{\sup}\limits_{s\in [0,T+K]}|\Delta_{\theta}Y_{s}^{m,p}|\Big\} \leq \dbE\exp\bigg\{q\gamma\Big(\mathop{\sup}\limits_{s\in [0,T+K]}|\Delta_{\theta}Y_{s}^{m,p}|+|\Delta_{\theta}\bar{Y}^{m, p}_s|\Big)\bigg\}
\leq C^{*}(q),
\end{align}
which gives an exponential estimate for the term $\Delta_{\theta}Y^{m, p}.$  This completes this step.

\ms

\noindent{\it Step 5:  $\{(Y^{m}, Z^{m})\}_{m=1}^{\i}$ is a Cauchy sequence}

\ms

In this step, we prove that for any $q > 1$, the sequence $ \{(Y^{m}, Z^{m})\}_{m=1}^{\i}$ is a Cauchy sequence in $S^q_{\dbF}(0,T + K;\dbR^{n}) \times \mathcal{H}^q_{\dbF}(0,T+K; \dbR^{n\times d})$. Additionally, its   limit belongs to the space $\mathcal{E}(0,T+K; \dbR^{n}) \times \mathcal{M}(0,T+K; \dbR^{n\times d}).$ From (\ref{eq:4.48}), we have that for all $q > 1$ and $\theta \in (0, 1),$
\begin{align*}
\varlimsup_{m\rightarrow \i} \mathop{\sup}\limits_{p \geq 1}\dbE\exp\Big\{q\gamma\mathop{\sup}\limits_{s\in [0,T+K]}|\Delta_{\theta}Y_{s}^{m,p}|\Big\} 
\leq C^{*}(q).
\end{align*}
Note that the inequality $x^q \leq e^{qx}$ for positive $x$ and for $q>1$, then we have 
\begin{align*}
    \varlimsup_{m \rightarrow \i} \mathop{\sup}\limits_{p \geq 1}\dbE\Big[\gamma^q\mathop{\sup}\limits_{t\in [0,T+K]}\frac{|Y_{t}^{m+p} - \theta Y_{t}^{m}|^q}{(1-\theta)^q}\Big] &\leq \varlimsup_{m\rightarrow \i} \mathop{\sup}\limits_{p \geq 1}\dbE\Big[\exp\Big\{q\gamma\mathop{\sup}\limits_{t\in [0,T+K]}\frac{|Y_{t}^{m+p} - \theta Y_{t}^{m}|}{1-\theta}\Big\}\Big] 
    \\
    &\leq C^{*}(q).
\end{align*}
Moreover, 
\begin{align*}
\varlimsup_{m \rightarrow \i} \mathop{\sup}\limits_{p \geq 1}\dbE\Big[\mathop{\sup}\limits_{t\in [0,T+K]}|Y_{t}^{m+p} - \theta Y_{t}^{m}|^q\Big] \leq (1-\theta)^q\gamma^{-q}C^{*}(q).
\end{align*}
Hence, we have that
\begin{align}
&\varlimsup_{m \rightarrow \i} \mathop{\sup}\limits_{p \geq 1}\dbE\Big[\mathop{\sup}\limits_{t\in [0,T+K]}|Y_{t}^{m+p} -  Y_{t}^{m}|^q\Big] \nonumber\\
&\leq 2^{q}\bigg\{\varlimsup_{m \rightarrow \i} \mathop{\sup}\limits_{p \geq 1}\dbE\Big[\mathop{\sup}\limits_{t\in [0,T+K]}|Y_{t}^{m+p} -  \theta Y_{t}^{m}|^q\Big] + \mathop{\sup}\limits_{m\geq0}\mathop{\sup}\limits_{p\geq1}\dbE\Big[\mathop{\sup}\limits_{t\in [0,T+K]}(1-\theta)^{q}|Y_{t}^{m}|^{q} \Big]\bigg\} \nonumber \\
& \leq 2^{q}(1-\theta)^q\bigg\{\gamma^{-q}C^{*}(q) + \mathop{\sup}\limits_{m\geq0}\mathop{\sup}\limits_{p\geq1}\dbE\Big[\mathop{\sup}\limits_{t\in [0,T+K]}|Y_{t}^{m}|^{q}\Big]\bigg\}. \label{eq:5.29}
\end{align}
By sending $\theta$ to 1 and by noting (\ref{eq:4.13}), we deduce  that the sequence $\{Y^{m}\}_{m=0}^{\i}$ is a Cauchy sequence in the space $S^q_{\dbF}(0,T + K;\dbR^{n})$, which implies that there exists an adapted process $Y\in S^q_{\dbF}(0,T + K;\dbR^{n})$ such that for any $q>1$,
\begin{align}\label{eq:4.58}
\lim_{m\rightarrow\i}\dbE\Big[\mathop{\sup}\limits_{t\in [0,T+K]}|Y_{t}^{m}- Y_{t}|^q\Big] = 0.
\end{align}
Now, from (\ref{eq:4.58}), we obtain that 
\begin{align*}
0\leq\varliminf_{m\rightarrow\i}\dbE\Big[\Big|\mathop{\sup}\limits_{t\in [0,T+K]}|Y_{t}^{m}|- \mathop{\sup}\limits_{t\in [0,T+K]}|Y_{t}|\Big|\Big] &\leq \varlimsup_{m\rightarrow\i}\dbE\Big[\Big|\mathop{\sup}\limits_{t\in [0,T+K]}|Y_{t}^{m}|- \mathop{\sup}\limits_{t\in [0,T+K]}|Y_{t}|\Big|\Big] \\
&\leq \lim_{m\rightarrow\i}\dbE\Big[\mathop{\sup}\limits_{t\in [0,T+K]}|Y_{t}^{m}- Y_{t}|\Big] = 0.
\end{align*}
Then we can deduce that $\mathop{\sup}\limits_{t\in [0,T+K]}|Y_{t}^{m}|$ converge in probability to $\mathop{\sup}\limits_{t\in [0,T+K]}|Y_{t}|$. 
Thus,  there exists a subsequence $m_k$ such that 
\begin{align}\label{q}
\lim_{k\rightarrow\i}\Big|\mathop{\sup}\limits_{t\in [0,T+K]}|Y_{t}^{m_k}| - \mathop{\sup}\limits_{t\in [0,T+K]}|Y_{t}|\Big| = 0,\q a.s..
\end{align}
%%---------------------------------------------------------
%$$\dbE\Big[\mathop{\sup}\limits_{t\in [0,T+K]}|Y_{t}^{m'_0}- Y_{t}|^q\Big] < \e_0^q,$$
%which implies that the following inequalities hold almost surely,
%%
% $$\mathop{\sup}\limits_{t\in [0,T+K]}|Y_{t}^{m'_0}- Y_{t}|^q <\i,\q
% %
% \exp\Big\{q\mathop{\sup}\limits_{t\in [0,T+K]}|Y_{t}^{m'_0}- Y_{t}|\Big\} <\i,\q
% \dbE\exp\Big\{q\mathop{\sup}\limits_{t\in [0,T+K]}|Y_{t}^{m'_0}- Y_{t}|\Big\} < \i.
% $$
%----------------------------------------------------------
 %
%
Therefore, by using H\"{o}lder's inequality, Fatou's Lemma, and \eqref{eq:4.13}, we deduce that
\begin{equation}
\begin{aligned}\label{eq:4.60}
\dbE\exp\Big\{q\mathop{\sup}\limits_{t\in [0,T+K]}| Y_{t}|\Big\}
&\leq \dbE\bigg[\exp\Big\{2q\varliminf_{k\rightarrow \i}\Big|\mathop{\sup}\limits_{t\in [0,T+K]}|Y_{t}^{m_k}| - \mathop{\sup}\limits_{t\in [0,T+K]}|Y_{t}|\Big|\Big\}\bigg]\\
&\q\cdot \mathop{\sup}\limits_{k\geq 0}\dbE\exp\Big\{2q\mathop{\sup}\limits_{t\in [0,T+K]}|Y_{t}^{m_k}|\Big\} = \mathop{\sup}\limits_{k\geq 0}\dbE\exp\Big\{2q\mathop{\sup}\limits_{t\in [0,T+K]}|Y_{t}^{m_k}|\Big\} < \i,
\end{aligned}
\end{equation}
which implies that  $Y\in \mathcal{E}(0,T + K; \dbR^{n}).$

On the other hand, by applying It\^{o}'s formula to $|Y_{t}^{m+p} - Y_{t}^{m}|^2$ firstly, and then by taking $\frac{q}{2}$ power, it can be deduced that for all $q> 1$,
\begin{align}\label{eq:4.61}
&\ \dbE\Big[\Big(\int_{0}^{T}|Z_{s}^{m+p} - Z_{s}^{m}|^2ds\Big)^{\frac{q}{2}}\Big]\nonumber\\
 &\leq 2^{\frac{q}{2}+1}\dbE\Bigg[ \mathop{\sup}\limits_{t\in [0,T]}|Y_{t}^{m+p} - Y_{t}^{m}|^{\frac{q}{2}}\cdot \Bigg\{ \Big(\sum_{i=1}^{n}\mathop{\sup}\limits_{t\in [0,T]}\Big|\int_{0}^{t}Z_{s}^{m+p;i}-Z_{s}^{m;i}dW_s\Big|\Big)^{\frac{q}{2}}\nonumber\\
 &\q +\Big(\sum_{i=1}^{n}\int_{0}^{T}\Big|f^{i}(s,Y_{s}^{m+p-1}, Z_{s}^{m+p;i}, Y^{m+p-1}_{s+\delta_s}) - f^{i}(s,Y_{s}^{m-1}, Z_{s}^{m;i}, Y^{m-1}_{s+\delta_s})\Big|ds\Big)^{\frac{q}{2}}\Bigg\}\Bigg],
\end{align}
where we have used the inequality   $(x+y)^{\frac{q}{2}}\leq 2^{\frac{q}{2}}(x^{\frac{q}{2}}+y^{\frac{q}{2}})$ for positive $x$ and $y$.
Then, by using the item (i) of \autoref{assumption3}, as well as (\ref{eq:4.13}) and (\ref{eq:4.38}), we get that 
\begin{align}\label{eq:4.62}
\mathop{\sup}\limits_{m,p\geq 1}\dbE\Big[\Big(\int_{0}^{T}\Big|f^{i}(s,Y_{s}^{m+p-1}, Z_{s}^{m+p;i}, Y^{m+p-1}_{s+\delta_s}) - f^{i}(s,Y_{s}^{m-1}, Z_{s}^{m;i}, Y^{m-1}_{s+\delta_s})\Big|ds\Big)^q\Big] < +\i.
\end{align}
It follows from Burkholder-Davis-Gundy's inequality, (\ref{eq:4.13}), and (\ref{eq:4.38}) that there exists  a positive constant $b$ such that for each $i = 1, ..., n$,
\begin{align}\label{eqq:5.29}
\mathop{\sup}\limits_{m,p\geq 1}\dbE\Big[\mathop{\sup}\limits_{t\in [0,T]}\Big|\int_{0}^{t}(Z_{s}^{m+p;i}-Z_{s}^{m;i})dW_s\Big|^{q}\Big]\leq b\mathop{\sup}\limits_{m,p\geq 1}\dbE\Big[\Big(\int_{0}^{T}|Z_{s}^{m+p;i}-Z_{s}^{m;i}|^2ds\Big)^{\frac{q}{2}}\Big]< +\i .
\end{align}
Moreover, note that $Z_{t}^{m+p} - Z_{t}^{m} = 0$ when $t\in [T,T+K]$. Then, by applying H\"{o}lder's inequality to (\ref{eq:4.61}), and  by combining (\ref{eq:4.58}), (\ref{eq:4.62}), and (\ref{eqq:5.29}), we obtain that 
\begin{align*}
	\lim_{m\rightarrow\i}\mathop{\sup}\limits_{p\geq 1}\dbE\Big[\Big(\int_{0}^{T+K}|Z_{s}^{m+p} - Z_{s}^{m}|^2ds\Big)^{\frac{q}{2}}\Big] = 0.
\end{align*}
It follows that there exists a process $Z\in \mathcal{H}^q_{\dbF}(0,T+K; \dbR^{n\times d})$ such that 
\begin{align}\label{eq:4.64}
	\lim_{m\rightarrow\i}\dbE\Big[\Big(\int_{0}^{T+K}|Z_{s}^{m} - Z_{s}|^2ds\Big)^{\frac{q}{2}}\Big] = 0,
\end{align}
thus $Z\in\mathcal{M}(0,T+K; \dbR^{n\times d}).$
Finally, In view of (\ref{eq:4.58}) and (\ref{eq:4.64}), by letting $n\rightarrow \i$ in BSDE (\ref{eq:4.12}), one can easily conclude that the pair $(Y, Z) \in \mathcal{E}(0,T+K; \dbR^{n}) \times \mathcal{M}(0,T+K; \dbR^{n\times d})$ is a solution of BSDE (\ref{eq:4.1}).

\ms

\noindent{\it Step 6:  Uniqueness of the solution} 

\ms

In this step, we prove the uniqueness of the solutions. Let $(\tilde{Y}, \tilde{Z}) \in \mathcal{E}(0,T+K; \dbR^{n}) \times \mathcal{M}(0,T+K; \dbR^{n\times d})$ be another solution of BSDE (\ref{eq:4.1}). To show the uniqueness, it suffice to show  that $Y = \tilde{Y}$ and $Z = \tilde{Z}$. Firstly, similarly to the deduction of (\ref{eq:4.61}), we obtain that for all $q>1$,
\begin{align}\label{eq:4.66}
&\ \dbE\Big[\Big(\int_{0}^{T}|Z_{s} - \tilde{Z}_{s}|^2ds\Big)^{\frac{q}{2}}\Big]\nonumber\\
&\leq 2^{\frac{q}{2}+1}\dbE\Bigg[ \mathop{\sup}\limits_{t\in [0,T]}|Y_{t} - \tilde{Y}_{t}|^{\frac{q}{2}}\cdot \Bigg\{ \mathop{\sup}\limits_{t\in [0,T]}\Big|\int_{0}^{t}Z_{s} - \tilde{Z}_{s}dW_s\Big|^{\frac{q}{2}}\nonumber\\
&\q +\bigg(\sum_{i=1}^{n}\int_{0}^{T}\Big|f^{i}(s,Y_{s}, Z_{s}, Y_{s+\delta_s}) - f^{i}(s,\tilde{Y}_{s}, \tilde{Z}_{s}, \tilde{Y}_{s+\delta_s})\Big|ds\bigg)^{\frac{q}{2}}\Bigg\}\Bigg].
\end{align}
Now,  we denote that 
\begin{align*}
\Delta_{\theta}P\triangleq \frac{Y - \theta\tilde{Y}}{1-\theta}\quad \hbox{and} \quad \Delta_{\theta}Q\triangleq \frac{Z - \theta\tilde{Z}}{1-\theta}, \q   \theta\in(0,1).
\end{align*}
By making use of the same argument as in the discussion from (\ref{eqqq::5.43}) to (\ref{eq:4.48}), one can easily show that for all $q > 1,$ there exists a positive constant $\tilde{C}(q)$ that depends only on $q$ such that 
\begin{align}\label{eq:4.67}
\dbE\exp\Big\{q\gamma\mathop{\sup}\limits_{s\in[0,T+K]}\{|\Delta_{\theta}P_{s}|\}\Big\} 
		\leq \tilde{C}(q).
		%\sqrt{\bar{I}(q)} \Bigg(\dbE\Big[\exp\Big\{8nq\gamma C\mathop{\sup}\limits_{s\in[0,T+K]}\{|\Delta_{\theta}P_{s}|\}(T-t)\Big\}\Big]\Bigg)^{\frac{1}{2}},
\end{align}
%where
%\begin{equation*}
%	\begin{aligned}
%		\bar{I}(q) = R(2q)\dbE\Big[\exp\Big\{2q\gamma\xi^{*}\Big\}\Big]\mathop{\sup}\limits_{m,p\geq1} \dbE\Big[\exp\Big\{4nq\gamma\xi^{*} + 8nq\gamma\mathop{\sup}\limits_{s\in[t,T+K]}|Y_{s}|T + 4nq\gamma\int_{0}^{T}\Gamma_sds \Big\}\Big]
%	\end{aligned}
%\end{equation*}
%with $R(q)$ defined as specified in (\ref{eq:4.13(1)}).
%For the case of $C = 0,$ it can be deduced from (\ref{eq:4.67}) that 
%\begin{equation}
%\dbE\Big[2\gamma\mathop{\sup}\limits_{s\in[0,T+K]}\{|\Delta_{\theta}P_{s}|\}\Big] \leq \dbE\Big[\exp\Big\{2\gamma\mathop{\sup}\limits_{s\in[0,T+K]}\{|\Delta_{\theta}P_{s}|\}\Big\}\Big] \leq \sqrt{\bar{I}(2)}.
%\end{equation}
Next, note the inequality  $|e^x - e^y| \leq (e^x + e^y)|x-y|$ for positive $x$ and $y$, which in fact can be derived using the mean value theorem. Then, by applying H\"{o}lder's inequality and  using a similar deduction as in (\ref{eq:5.29}), we have that 
\begin{equation}\label{5.40}
\begin{aligned}
&\ \dbE \mathop{\sup}\limits_{s\in[0,T+K]}|\exp(Y_{s}) - \exp(\tilde{Y}_{s})|^q \\ &\leq 2^{q}\dbE \mathop{\sup}\limits_{s\in[0,T+K]}\Big\{\exp(q|Y_s |)+ \exp(q|\tilde{Y}_{s}|)\Big\}|Y_{s} - \tilde{Y}_{s}|^q \\
&\leq2^{q+1}\dbE\mathop{\sup}\limits_{s\in[0,T+K]}\Big\{\exp(2q|Y_s |)+ \exp(2q|\tilde{Y}_{s}|)\Big\}
 \cdot\big\{\dbE |Y_{s} - \tilde{Y}_{s}|^{2q} \big\}^{\frac{1}{2}}\\
&\leq (1-\theta)^{2q}2^{3q+1} \dbE\mathop{\sup}\limits_{s\in[0,T+K]}\Big\{\exp(2q|Y_s |)+ \exp(2q|\tilde{Y}_{s}|)\Big\}
\cdot\Big\{\gamma^{-2q}\tilde{C}(2q) + \dbE\mathop{\sup}\limits_{s\in[0,T+K]}| \tilde{Y}_{s}|^{2q}\Big\}^\frac{1}{2}.
\end{aligned}
\end{equation}
Finally, on one hand, by taking $\theta \rightarrow 1$, we obtain that $\exp(Y)= \exp(\tilde{Y})$ on the whole interval $[0, T+K]$. 
On the other hand, by  (\ref{eq:4.66}), \eqref{5.40}, and H\"{o}lder's inequality, we conclude that  $Z = \tilde{Z}$ on the whole interval $[0, T+K]$. 
This completes the whole proof.
\end{proof}

%Add a remark to summarize the above result firstly, and then to introduce the following  rusult with a more general assumption.

Note that, in the context of multidimensional BSDEs with unbounded terminal values, \autoref{thm:5.2} establishes the existence and uniqueness of the quadratic BSDE given by \eqref{eq:4.1}, where the generator $f(\cd)$ is independent of the anticipated term of $Z$. We now consider the following BSDE in which the generator includes the anticipated term of $Z$:
\begin{equation}\label{eq:5.73}
	\left\{\begin{aligned}
		&Y_{t} = \xi_{T} + \int_{t}^{T}\Big\{f(s,Z_s)+ \dbE\big[g(s, Y_{s}, Z_{s}, Y_{s + \delta_{s}}, Z_{s+\zeta_{s}})\big]\Big\}
		ds  - \int_{t}^{T}Z_{s}dW_s, \q t\in [0, T];\\
		&Y_{t} = \xi_{t}, \q Z_{t} = \eta_{t}, \q t\in [T, T+K].
	\end{aligned}\right.
\end{equation}
 Note that  $g(\omega, t, y, z, \phi_r, \psi_{\bar{r}}): { \Om \ts [0, 	T] \ts \dbR^n \times \dbR^{n\times d}\times L^2_{\sF_r}(\Om; \dbR^n)\times L^2_{\sF_{\bar{r}}}(\Om; \dbR^{n\times d}) \rightarrow L^2_{\sF_t}(\Om;\dbR^n)}$ and $g(\cdot, \cdot, y, z, \phi_r, \psi_{\bar{r}})$ is $\dbF$-progressively measurable. In addition, $f(\omega, t,  z ): { \Om \ts [0, 	T]  \times \dbR^{n\times d}  \rightarrow L^2_{\sF_t}(\Om;\dbR^n)}$ and $f(\cdot, \cdot,  z)$ is $\dbF$-progressively measurable, where $r\in [t, T+K]$.
%In fact, the related result still holds when the generator $f(\cd)$ depends on the anticipated term of  $Z$. To illustrate this, now we consider BSDE \eqref{eq::1.1} again and present the following assumption:
%To illustrate this, we consider the following BSDE:
%%
%%
%%
%\begin{equation}\label{eq:5.73}
%\left\{\begin{aligned}
	%&Y_{t} = \xi_{T} + \int_{t}^{T} f(s, Y_{s}, Z_{s}, Y_{s + \delta_s}, Z_{s+\zeta_s})ds - \int_{t}^{T}Z_{s}dW_s, \q t\in [0, T];\\
	%&Y_{t} = \xi_{t}, \q Z_{t} = \eta_{t}, \q t\in [T, T+K].
	%\end{aligned}\right.
	%\end{equation}
	%%
	%where the  parameter $\delta$ and $\zeta$ still satisfy  the conditions in \rf{eqqq:3.2} and \rf{eq::3.2}. 
	%\tc{Now we will present a assumption.}
	%
	\begin{assumption}\label{assumption4} \rm
		For all $t\in[0, T]$, $d\dbP \times dt$-a.e., $y,\bar{y}\in  \dbR^{n}, z\in \dbR^{n\times d}, \phi, \bar{\phi}\in \mathcal{H}^2_{\dbF}(t,T+K;\dbR^{n}),$ and $\psi\in \mathcal{H}^2_{\dbF}(t,T+K;\dbR^{n\times d})$, the generators $f$, $g$ satisfies the following conditions:
		%
		%The deterministic continuous functions $\delta$ and $\zeta$ satisfy the following two items:
		%\begin{itemize}
		%	\item [$\rm(i)$]  for all $t \in [0, T]$, there exists a constant $K \geq 0,$ such that
		%	$$ \quad t + \delta_t \leq T+K;\quad\quad t + \zeta_t \leq T+K.$$
		%	\item [$\rm(ii)$] There exists a constant $L \geq 0$ such that nonnegative and integrable $h$,
		%	\begin{align}
			%		\int_{t}^{T} h_{s + \delta_s} ds \leq L \int_{t}^{T+ K} h_s ds; \quad \int_{t}^{T} h_{s + \zeta_s} ds \leq L \int_{t}^{T + K} h_s ds,  \quad \forall t \in [0, T].
			%	\end{align}
		%\end{itemize}
		%
		\begin{itemize}
			\item [$\rm(i)$] There exist  positive constants $C$, $\gamma > 0$, and $\Gamma$ with $p\geq 1, \dbE[\exp\{p\int_{0}^{T}|\Gamma_t|dt\}] <+\i$ such that 
			\begin{align*}
				&|g(t, y, z, \phi_r, \psi_{\bar{r}})| \leq C\big\{1 + |y| + |z|^2 + \dbE_t[|\phi_r|] + \dbE_t[|\psi_{\bar{r}}|]\big\},\\
				& |g(t, y, z, \phi_r, \psi_{\bar{r}}) - g(t, \bar{y}, z, \bar{\phi_r}, \psi_{\bar{r}})| \leq C\big\{ |y - \bar{y}| + \dbE_t[|\phi_r - \bar{\phi}_{r}|]\big\},\\
				&|f^{i}(t, z)| \leq \Gamma_{t} +  \frac{\gamma}{2}|z^{i}|^2.
			\end{align*}
			\item [$\rm(ii)$] For $i = 1, ..., n$, it holds that $f^{i}(w, t, \cdot)$ is either convex or concave.
			\item [$\rm(iii)$] The coefficients 
			$\xi\in\mathcal{E}(T,T+K; \dbR^n)$ 
			%$\xi\in \tc{\mathcal{E}_{\dbF}(T,T+K; \dbR^n)}$
			and 
			$\eta \in \mathcal{M}(T,T+K; \dbR^{n\times d})$.
			%$\eta \in \tc{L^4_{\dbF}(T,T+K; \dbR^{n\times d})}$
		\end{itemize}
	\end{assumption}
	%\begin{remark}\rm
	%	The generator $f$ in \autoref{assumption4} is quadratic growth with respect to $Z$, which are more general than diagonally quadratic growth. 
	%\end{remark}
	%
	%\tc{Finally, we will give a existence and uniqueness result based on the result in \autoref{thm:5.2}. }
	%
	%

	\begin{theorem}\label{thm:5.3} \sl 
		Under  \autoref{assumption4}, BSDE \eqref{eq:5.73} admits a unique solution $(Y, Z) \in S^2_{\dbF}(0,T + K;\dbR^{n}) \times \mathcal{M}(0,T+K; \dbR^{n\times d})$. 
		%Moreover, if $f$ is bounded, then the unique solution $(Y, Z) \in S^q_{\dbF}(0,T + K;\dbR^{n}) \times \mathcal{M}(0,T+K; \dbR^{n\times d})$, for all $q > 1$.
	\end{theorem}
	\begin{proof}
		First, we consider the following BSDE:
		\begin{align}\label{eq:5.78}
			\widetilde{Y}_{t} = \xi_{T}  + \int_{t}^{T}f(s, Z_{s})ds - \int_{t}^{T}Z_sdW_s,\q t\in[0, T].
		\end{align}
		It follows from \autoref{thm:5.2} and \autoref{assumption4} that BSDE (\ref{eq:5.78}) admits a unique solution $(\widetilde{Y}, Z)$ in the space $\mathcal{E}(0,T; \dbR^n) \times \mathcal{M}(0,T+K; \dbR^{n\times d})$.
		Second, we consider the following equation
		\begin{equation}\label{241229}
			\left\{\begin{aligned}
				&Y_{t} = \widetilde{Y}_{t} + \int_{t}^{T}\dbE\big[g(s, Y_{s}, Z_{s}, Y_{s + \delta_s}, Z_{s+\zeta_s})\big]ds, \q t\in [0, T];\\
				&Y_{t} = \xi_{t},\q Z_{t} = \eta_{t},\q t\in [T, T + K]. 
			\end{aligned}\right.
		\end{equation}
		Note that in \rf{241229}, the value of $\widetilde{Y}$ and $Z$ is already determined by the uniqueness of BSDE \rf{eq:5.78} for  $t\in [0, T]$ and $t\in [0, T+K]$, respectively. Next, we would like to prove that \rf{241229} admits a unique solution $Y$. To do this, we define  $Y_{t}^{0} = 0$ for $t\in[0, T]$ and $Y_{t}^{0} = \xi_t$ for $t\in[T, T+K]$. 
		Moreover,  for any $m \in\dbN,$ we define 
		\begin{equation}\label{eq:5.80}
			\left\{\begin{aligned}
				&Y^{m+1}_{t} = \widetilde{Y}_{t} + \int_{t}^{T}\dbE\big[g(s, Y^{m}_{s}, Z_{s}, Y^{m}_{s + \delta_s}, Z_{s+\zeta_s})\big]ds, \q t\in [0, T];\\
				& Y^{m+1}_{t} = \xi_{t}, \q t\in [T, T+K].
			\end{aligned}\right.
		\end{equation}
		%Then
		For the existence of Eq. \rf{241229}, we would like to prove that the sequence $\{Y^{m};m\in\dbN\}$ is Cauchy in the space $S^2_{\dbF}(0,T + K;\dbR^{n})$. From \autoref{assumption4}, we have that
		%
		%\begin{align*}
		%&g(s, Y^{m}_{s}, Z_{s}, Y^{m}_{s + \delta_s}, Z_{s+\zeta_s}) \\
		%\leq& C\bigg\{1 + |Y_{s}^{m}| + \Big(\int_{0}^{T}|Z_{r}|^2dr\Big)^{\frac{1}{2}} + \dbE_{s}\big[\mathop{\sup}\limits_{\tau\in [0, T]}|Y^{m}_{s\vee \tau + \delta(s\vee \tau)}|\big] + \dbE_{s}\Big[\Big(\int_{0}^{T}|Z_{r + \zeta_r}|^2dr\Big)^{\frac{1}{2}}\Big]\bigg\}\\
		%\leq& C\bigg\{1 + \mathop{\sup}\limits_{\tau\in [0, T + K]}|Y^{m}_{\tau}| + \Big(\int_{0}^{T}|Z_{r}|^2dr\Big)^{\frac{1}{2}} + \dbE_{s}\big[\mathop{\sup}\limits_{\tau\in [0, T+K]}|Y^{m}_{\tau}|\big] + \sqrt{L}\dbE_{s}\Big[\Big(\int_{0}^{T + K }|Z_{r}|^2dr\Big)^\frac{1}{2}\Big]\bigg\}.
		%\end{align*}
		\begin{equation}\label{1224}
			\begin{aligned}
				&\ \dbE\Big[\mathop{\sup}\limits_{t\in [0, T+K]}|Y_{t}^{m+1}|^2\Big] \leq \dbE\Big[\mathop{\sup}\limits_{t\in [0, T]}|Y_{t}^{m+1}|^2\Big] + \dbE\Big[\mathop{\sup}\limits_{t\in [T, T+K]}|\xi_t|^2\Big] \\
				&\leq 2^5\bigg\{\dbE\Big[\mathop{\sup}\limits_{\tau\in [0, T]}|\widetilde{Y}_{t}|^2\Big] + C^2 T^2\Big\{1+4\dbE\Big[\mathop{\sup}\limits_{s\in [0, T+K]}|Y_{s}^{m}|^2\Big]\Big\}+C^2\dbE\Big[\Big(\int_{0}^{T}|Z_s|^2ds\Big)^2\Big]\\
				&\q\q\q+C^2L\dbE\Big[\int_{0}^{T+K}|Z_s|^2ds\Big]\bigg\}+\dbE\Big[\mathop{\sup}\limits_{t\in [T, T+K]}|\xi_t|^2\Big].
			\end{aligned}
		\end{equation}
		Note that $(\widetilde{Y}, Z) \in  \mathcal{E}(0,T; \dbR^n) \times \mathcal{M}(0,T+K; \dbR^{n\times d}).$
		Thus, if  $Y^{m} \in S^2_{\dbF}(0,T + K;\dbR^{n}),$ then from (\ref{1224}) and the item (iii) of \autoref{assumption4}, 
		it is easy to deduce that  $Y^{m+1} \in S^2_{\dbF}(0,T + K;\dbR^{n}).$ 
		Therefore, by induction, we obtain that $Y^{m} \in S^2_{\dbF}(0,T + K;\dbR^{n})$ for any $m\geq 0.$ 
		For simplicity of presentation, next we denote that
		$\Delta Y^{m} \triangleq Y^{m} - Y^{m-1}.$ Then, when $t\in [0, T],$ we have that
		\begin{align*}
			\Delta Y_{t}^{m + 1} = \int_{t}^{T} \dbE\big[g(s, Y^{m}_{s}, Z_{s}, Y^{m}_{s + \delta_s}, Z_{s+\zeta_s}) - g(s, Y^{m-1}_{s}, Z_{s}, Y^{m-1}_{s + \delta_s}, Z_{s+\zeta_s})\big] ds.
		\end{align*}
		Applying It\^{o}'s formula to $e^{\alpha t}|Y_{t}^{m+1}|^2$,  where 
		$\a$ is a positive constant which will be determined later,
		we have that
		\begin{align}
			&\ e^{\alpha t}|\Delta Y_{t}^{m+1}|^2 + \alpha\int_{t}^{T}e^{\alpha s}|\Delta Y_{s}^{m+1}|^{2}ds \nonumber\\
			&= 2\int_{t}^{T}e^{\alpha s}\Delta Y_{s}^{m+1}\dbE\Big[g(s, Y^{m}_{s}, Z_{s}, Y^{m}_{s + \delta_s}, Z_{s+\zeta_s}) - g(s, Y^{m-1}_{s}, Z_{s}, Y^{m-1}_{s + \delta_s}, Z_{s+\zeta_s})\Big]ds\nonumber\\
			&\leq  \int_{t}^{T}e^{\alpha s}2\sqrt{\frac{\alpha}{2}}\Delta Y_{s}^{m+1} C\sqrt{\frac{2}{\alpha}}\dbE\Big[ |\Delta Y^{m}_{s}| + \dbE_{s}[|\Delta Y^{m}_{s + \delta_s}]| \Big]ds\nonumber\\
			&\leq \frac{\alpha}{2}\int_{t}^{T}e^{\alpha s}|\Delta Y_{s}^{m+1}|^{2}ds + \int_{t}^{T}\frac{4C^2}{\alpha}\Big\{ e^{\alpha s}\dbE[|\Delta Y^{m}_{s}|^2] + e^{\alpha s}\dbE[|\Delta Y^{m}_{s+\delta_s} |^2]\Big\}ds.
		\end{align}
		%
		%Then, it follows from Jensen's inequality that for each fixed $q> 2$,
		%
		%\begin{align}\label{eq:5.82}
		%e^{\frac{q\alpha t}{2}}|\Delta Y_{t}^{m+1}|^q \leq \int_{t}^{T}2^{\frac{q}{2}}\Big(\frac{4C^2}{\alpha}\Big)^{\frac{q}{2}}\Big\{ e^{\frac{q\alpha s}{2}}|\Delta Y^{m}_{s}|^q + e^\frac{q\alpha s}{2}\dbE_{s}[|\Delta Y^{m}_{s} |^q]\Big\}ds.
		%\end{align}
		%
		By taking supremum,  expectation, as well as applying Tower property, and noting the fact that $|\Delta Y_{s}^{m}| = 0$ when $s\in [T, T+K]$, we have that
		\begin{align*}
			\dbE\Big[\mathop{\sup}\limits_{t\in [0, T + K]}e^{\alpha t}|\Delta Y_{t}^{m+1}|^2\Big] 
			%&\leq \dbE\Big[\int_{0}^{T}\frac{16C^2}{\alpha}\mathop{\sup}\limits_{s\in [0, T + K]}e^{\alpha s}|\Delta Y^{m}_{s}|^2 ds\Big]
			%
			\leq \frac{16C^2T}{\alpha}\dbE\Big[\mathop{\sup}\limits_{s\in [0, T + K]}e^{\alpha s}|\Delta Y_{s}^{m}|^2\Big].
		\end{align*}
		%we have
		%\begin{equation*}
		%\dbE\Big[\mathop{\sup}\limits_{s\in [0, T + K]}|\Delta Y_{s}^{m+1}|^2\Big] \leq \frac{8C^2T}{\alpha}e^{\alpha T}\dbE\Big[\mathop{\sup}\limits_{s\in [0, T + K]}|\Delta Y_{s}^{m}|^2\Big].
		%\end{equation*}
		Now, by letting $\alpha \triangleq 64C^2 T$, we obtain that
		\begin{align}\label{eq:5.83}
			\dbE\Big[\mathop{\sup}\limits_{s\in [0, T + K]}e^{\alpha s}|\Delta Y_{s}^{m+1}|^2\Big] \leq \frac{1}{4}\dbE\Big[\mathop{\sup}\limits_{s\in [0, T + K]}e^{\alpha s}|\Delta Y_{s}^{m}|^2\Big].
		\end{align}
		%
%		Then, there exists a positive constant $M$ such that, for any $m\geq 1$,
%		%
%		\begin{align*}
%			\dbE\Big[\mathop{\sup}\limits_{s\in [0, T + K]}|\Delta Y_{s}^{m}|^2\Big] \leq\dbE\Big[\mathop{\sup}\limits_{s\in [0, T + K]}e^{\alpha s}|\Delta Y_{s}^{m}|^2\Big] \leq \frac{1}{4^{m-1}}\dbE\Big[\mathop{\sup}\limits_{s\in [0, T + K]}e^{\alpha s}|\Delta Y_{s}^{1}|^2\Big] \leq \frac{M^{2}}{4^{m-1}}.
%		\end{align*}
		%
		Thus, one has that for $m > k$,
		\begin{align*}
			\|Y^{m} - Y^{k}\|_{S^2_{\dbF}(0,T + K;\dbR^{n})} \leq \sum_{i = k+1}^{m}\|\Delta Y^{i}\|_{S^2_{\dbF}(0,T + K;\dbR^{n})} \leq \sum_{i = k+1}^{m}\frac{M}{2^{i-1}}\leq \frac{M}{2^{k-1}} \rightarrow 0, \q \hb{as}\q  k\rightarrow +\i,
		\end{align*}
		which implies that  $\{Y^{m};m\in\dbN\}$ is a Cauchy sequence in the space $S^2_{\dbF}(0,T + K;\dbR^{n})$.
		Therefore,  there exists a process $Y \in S^2_{\dbF}(0,T + K;\dbR^{n})$ such that $Y^{m} \rightarrow Y$ in $S^2_{\dbF}(0,T + K;\dbR^{n})$, as $m\rightarrow +\i.$
		It follows that  Eq. \rf{241229} admits a  solution $Y \in S^2_{\dbF}(0,T + K;\dbR^{n})$.

		For the uniqueness of BSDE \rf{eq::1.1}, we only need to prove that the solution of Eq. \rf{241229} is unique. If $\overline Y\in S^q_{\dbF}(0,T + K;\dbR^{n})$ is another solution of Eq. \rf{241229}. Then, similarly to the deduction of (\ref{eq:5.80}) - (\ref{eq:5.83}), we have that 
		\begin{align*}
			\dbE\Big[\mathop{\sup}\limits_{s\in [0, T + K]}e^{\alpha s}
			| Y_{s}-\overline{Y}_s|^2\Big] \leq \frac{1}{4}\dbE\Big[\mathop{\sup}\limits_{s\in [0, T + K]}e^{\alpha s}| Y_{s}-\overline{Y}_s|^2\Big],
		\end{align*}
		which implies the uniqueness of Eq. \rf{241229}. 
		This competes the proof.
		%Finally, if $f$ is bounded, it follows from Eq \rf{241229} and $\widetilde{Y}\in \mathcal{E}(0,T+K; \dbR^{n})$ that for $q>1$, $Y \in S^q_{\dbF}(0,T + K;\dbR^{n})$.  
		%
		%Meanwhile, if $g$ is bounded, then since 
		%\begin{align*}
		%	Y_{t} = \bar{Y}_{t} + \int_{t}^{T} g(s, Y_{s\vee \cdot}, Z_{\cdot}, Y_{s\vee \cdot + \delta_{s \vee \cdot}}, Z_{\cdot+\zeta_.})ds\leq \bar{Y}_{t} + CT, \q t\in [0, T],
		%\end{align*}
		%we know that there exists $(Y, Z)\in S^q_{\dbF}(0,T + K;\dbR^{n})\times \mathcal{H}^q_{\dbF}(0,T+K; \dbR^{n\times d})$ satisfy BSDE (\ref{eq:5.73}). The uniqueness follows from the similar deduction of the first case. 
	\end{proof}

	%\begin{remark} \rm 
	% \autoref{thm:5.3} still holds true  if the generator $f(s, Y_{s}, Z_{s}, Y_{s + \delta_s}, Z_{s+\zeta_s})$ in BSDE (\ref{eq::1.1}) is replaced by 
	%	$\dbE_t[f(s, Y_{s}, Z_{s}, Y_{s + \delta_s}, Z_{s+\zeta_s})]$ or $f(s, Y_{s\vee \cdot}, Z_{\cdot}, Y_{s\vee \cdot + \delta_{s \vee \cdot}}, Z_{\cdot+\zeta_.})$, where $f(s, Y_{s\vee \cdot}, Z_{\cdot}, Y_{s\vee \cdot + \delta_{s \vee \cdot}}, Z_{\cdot+\zeta_.})$ represents the path-dependent generator, and the proof is similar.
	%\end{remark}

\section{Appendix}\label{section6}
In this part, we will give the proofs of some auxiliary results on one-dimensional quadratic BSDE \eqref{241211}.

First, we present a lemma concerning one-dimensional BSDEs with quadratic growth,  which is essentially an extension of Hu--Tang \cite[Lemma 2.1]{Hu-Tang-16} and will be used later.

\begin{lemma}\label{prop2.1}  \sl 
	Suppose that there exists a positive constant $C$ such that for any  $t \in [0, T]$ and  $z,\bar{z}\in \dbR^{ d}$,  the scalar-value generator $f: \Om\times[0, T]\times \dbR^{d} \rightarrow \dbR$ satisfies the  conditions:
	\begin{align}  \label{eqq}
		|f(t, z)| \leq \theta_{t} + |H_t|^{1+\alpha} + \frac{\gamma}{2} |z|^2  \q\hbox{and}\q   
		|f(t,z) - f(t,\bar{z})| \leq C(\beta_{t} + |z| + |\bar{z}|) |z - \bar{z}|,  
	\end{align}
	where $\int_{0}^{T} \theta_tdt$ is (essentially) bounded,	$\beta_{t} \in \mathcal{Z}^2_{\dbF}(0,T; \dbR)$, and  $H\cdot W$ is a $BMO$ martingale. 
	Then, for any $\xi \in L_{\sF _{T}}^{\i}(\Om; \dbR),$ BSDE \eqref{241211}
%	\begin{align}\label{2.3}
%		Y_t = \xi + \int_{t}^{T}f(s, Z_s)ds - \int_{t}^{T}Z_s dW_s
%	\end{align}
	%$$Y_t = \xi + \int_{t}^{T}f(s, Z_s)ds - \int_{t}^{T}Z_s dW_s$$ 
	has a unique adapted solution $(Y, Z)\in \mathcal{H}^\i_{\dbF}(0,T;\dbR) \times \mathcal{Z}^2_{\dbF}(0,T;\dbR^{d})$. Moreover, we have
	\begin{align}
		&	|Y_t| \leq  \frac{1}{\gamma}\ln2  + \|\xi\|_\i  +  \bigg\|\int_{0}^{T}\theta_{s}ds\bigg\|_{\i}  + C_{\alpha} \gamma^{\frac{1+\alpha}{1-\alpha}}\|H\|^{\frac{2(1+\alpha)}{1-\alpha}}_{\mathcal{Z}^2_{\dbF}(t,T;\dbR)} (T - t), \label{2.4}\\	
		&	 \dbE_{\tau}\Big[\int_{\tau}^{T}|Z_s|^2 ds\Big]
		\leq \frac{1}{\gamma}\exp\{{2\gamma\|Y\|_{\mathcal{H}^\i_{\dbF}(t,T; \dbR)}}\}\Big\{1+2\bigg\|\int_{0}^{T}\theta_{s}ds\bigg\|_{\i}+  2C_{\alpha}
		\|H\|^{\frac{2(1+\alpha)}{1-\alpha}}_{\mathcal{Z}^2_{\dbF}(t,T;\dbR)}(T-t)\Big\}\nonumber\\
		&\q\q\q\q\q\q\q\q\q\q +\frac{1}{\gamma^2}\exp\{2\gamma \|\xi\|_{\i}\}, \q \forall t\in[0,T],\q \forall \tau \in \mathscr{T}[t,T], \label{new}
	\end{align}
where
\begin{align*}
C_{\alpha} \triangleq \frac{1-\alpha}{2}(1+\alpha)^{\frac{1+\alpha}{1-\alpha}}.
\end{align*}
\end{lemma}
We would like to omit the detailed proof of \autoref{prop2.1}, which is a special case of Fan--Hu--Tang \cite[Lemma A.1]{fan2023multi}.

\begin{proof}[\bf{Proof of \autoref{prop:2.1}}]
	Let 
	\begin{align*}
		\bar{g}_{t} \triangleq \rho(|U_{t}|) +  n\lambda|V_{t}|^{1+\alpha}+ \rho_{0}(\dbE_t[|U_{t+\delta_t}|]) +\lambda_{0}\big(\dbE_{t}[|V_{t+\zeta_t}|]\big)^{1+\alpha}, \q t\in [0, T].
	\end{align*}
	Then by the inequalities $x\leq (x+1)^{1+\alpha}$, $x^{1+\alpha} + y^{1+\alpha}\leq (x+y)^{1+\alpha}$, for any $x, y \geq 0$,	we have
	\begin{align*}
		\bar{g}_{t} &\leq \Big(\rho(|U_t|)+1\Big)^{1+\alpha} + \Big((n\lambda+1)|V_{t}|\Big)^{1+\alpha} + \Big(\rho_{0}(\dbE_t[|U_{t+\delta_t}|])+1\Big)^{1+\alpha} + \Big((\lambda_{0}+1)\dbE_{t}[|V_{t+\zeta_t}|]\Big)^{1+\alpha}\\
		&\leq \bigg(2+(\rho+\rho_{0})\big(\|U\|_{\mathcal{H}^\i_{\dbF}(t,T+K; \dbR^n)}\big) + (n\lambda+1)|V_{t}| + (\lambda_{0}+1)\dbE_{t}[|V_{t+\zeta_t}|]\bigg)^{1+\alpha}.
	\end{align*}
	Define	
	\begin{align*}
		\bar{H}_{t} \triangleq 2+(\rho+\rho_{0})\big(\|U\|_{\mathcal{H}^\i_{\dbF}(t,T+K; \dbR^n)}\big) + (n\lambda+1)|V_{t}| + (\lambda_{0}+1)\dbE_{t}[|V_{t+\zeta_t}|], \q t\in [0, T].
	\end{align*}
	It then follows from $V\in\mathcal{Z}^2_{\dbF}(0,T;\dbR^{n\times d})$ that $|V|\in \mathcal{Z}^2_{\dbF}(0,T;\dbR)$. And since
	\begin{align}\label{eq::3.6}
		\|\dbE_{.}[|V_{.+\zeta_.}|]\|_{\mathcal{Z}^2_{\dbF}(0,T;\dbR)}&=\mathop{\sup}\limits_{\tau \in \mathscr{T}[0,T]} \bigg\|\dbE_{\tau}\Big[\int_\tau^T\Big(\dbE_{s}[|V_{s+\zeta_s}|]\Big)^2ds\Big]\bigg\|^{\frac{1}{2}}_{\i}\nonumber\\
		&\leq \mathop{\sup}\limits_{\tau \in \mathscr{T}[0,T]} \bigg\|\dbE_{\tau}\Big[\int_\tau^T\dbE_{s}[|V_{s+\zeta_s}|^2]ds\Big]\bigg\|^{\frac{1}{2}}_{\i}\nonumber\\
		&=\mathop{\sup}\limits_{\tau \in \mathscr{T}[0,T]} \bigg\|\dbE_{\tau}\Big[\int_\tau^T|V_{s+\zeta_s}|^2ds\Big]\bigg\|^{\frac{1}{2}}_{\i}\\
		&\leq L \mathop{\sup}\limits_{\tau \in \mathscr{T}[0,T]} \bigg\|\dbE_{\tau}\Big[\int_\tau^{T+K}|V_{s}|^2ds\Big]\bigg\|^{\frac{1}{2}}_{\i}\nonumber\\
		&\leq L \mathop{\sup}\limits_{\tau \in \mathscr{T}[0,T+K]} \bigg\|\dbE_{\tau}\Big[\int_\tau^{T+K}|V_{s}|^2ds\Big]\bigg\|^{\frac{1}{2}}_{\i}
		=L\|V\|_{\mathcal{Z}^2_{\dbF}(0,T+K;\dbR^{n\times d})},\nonumber
		%&\leq L \mathop{\sup}\limits_{\tau \in \mathscr{T}[0,T]} \bigg\|\dbE_{\tau}[\int_\tau^{T}|V_{s}|^2ds]\bigg\|^{\frac{1}{2}}_{\i} + L \mathop{\sup}\limits_{\tau \in \mathscr{T}[0,T]} \bigg\|\dbE_{\tau}[\int_T^{T+K}|\eta_{s}|^2ds]\bigg\|^{\frac{1}{2}}_{\i} \\
		%&= L\Big(\|V\|_{\mathcal{Z}^2_{\dbF}(0,T;\dbR^{n\times d})} + \|\eta\|_{\mathcal{Z}^2_{\dbF}(T,T+K;\dbR^{n\times d})}\Big) < \i,
	\end{align}
	we have $\dbE_{.}[|V_{.+\zeta_.}|] \in \mathcal{Z}^2_{\dbF}(0,T;\dbR)$. Therefore, by homogeneity and triangular inequality properties of norm $\mathcal{Z}^2_{\dbF}(0,T;\dbR)$, we conclude that $\bar{H}\in \mathcal{Z}^2_{\dbF}(0,T;\dbR)$.
	Now we observe that the generator of BSDE (\ref{241211}) satisfies the conditions  in \autoref{prop2.1} (by letting $\bar{H}_{t} = H_{t}$). It then follows from \autoref{prop2.1} that BSDE (\ref{241211}) admits a unique adapted solution $(Y, Z)\in \mathcal{H}^\i_{\dbF}(0,T;\dbR) \times \mathcal{Z}^2_{\dbF}(0,T;\dbR^{d})$. Finally, by homogeneity and triangular inequality properties of norm $\mathcal{Z}^2_{\dbF}(0,T;\dbR)$, (\ref{eq::3.6}), as well as the inequality that $(x+y)^2 \leq 2x^2+2y^2$, for any $x, y \geq 0$, we obtain (\ref{2.6}) and (\ref{eq2.7}) easily. Thus the proof is complete.
\end{proof}
%
%\begin{proof}[\bf{Proof of \autoref{lemma:6.2} and \autoref{lemma:6.2-2}}]
%\tc{By applying It\^{o} -Tanaka's formula to compute
	%\begin{align*}
	%	\exp(M_{\bar{t}}) \triangleq \exp\Big\{\gamma|Y_{\bar{t}}| + \gamma \int_{t}^{_{\bar{t}}} \theta_s + \sigma|U_s| + \lambda |V_s|^{1+\alpha} + \sigma_0\dbE_{s}[|U_{s+\delta_s}|] + \lambda_{0}\dbE_{s}[|V_{s+\zeta_s}|]ds\Big\}
	%\end{align*}  
	%and then using Dominated convergence theorem, we can easily to get \autoref{lemma:6.2}. For \autoref{lemma:6.2-2}, we can use a similar deduction of Fan--Hu--Tang \cite[Proposition 2]{Fan S} to get the desired result. }
%\end{proof}

\begin{proof}[\bf{Proof of \autoref{lemma:6.2}}]
	For fixed $\tau\in \mathscr{T}[0, T]$. Define 
	\begin{align*}
		\tau_n \triangleq \inf\Big\{s\in [\tau, T] : \int_{0}^{s}|Z_r|^2\exp (2M_r) dr \geq n\Big\} \wedge T
	\end{align*}
	and
	\begin{align*}
		\exp(M_{\bar{t}}) \triangleq \exp\Big\{\gamma|Y_{\bar{t}}| + \gamma \int_{\tau}^{_{\bar{t}}} (\theta_s + \sigma|U_s| + \lambda |V_s|^{1+\alpha} + \sigma_0\dbE_{s}[|U_{s+\delta_s}|] + \lambda_{0}\dbE_{s}[|V_{s+\zeta_s}|])ds\Big\}. 
	\end{align*}       
	Then, for any $\tau \in \mathscr{T}[0, T]$, by using It\^{o}-Tanaka's formula to $\exp(M)$ in $[\tau, \tau_n]$,                
	we have that 
	\begin{align*}
		\exp(\gamma|Y_\tau|) = \dbE_\tau \exp(M_\tau) \leq \dbE_\tau \exp(M_{\tau_n}) - \gamma\dbE_\tau \int_{\tau}^{\tau_n}\sgn(Y_s)Z_s \exp(M_s)dW_s,
	\end{align*}
	where we  used the condition (\ref{eq:6.6}). Now, noting the definition of $\tau_n$ and Optional Sampling theorem, one can deduce that 
	$$\dbE_\tau\int_{\tau}^{\tau_n}\sgn(Y_s)Z_s \exp(M_s)dW_s = \dbE_\tau\Big[\int_{0}^{\tau_n}\sgn(Y_s)Z_s \exp(M_s)dW_s\Big] - \int_{0}^{\tau}\sgn(Y_s)Z_s \exp(M_s)dW_s =  0.$$ Thus, for any $\tau \in \mathscr{T}[0, T]$,
	\begin{align*}
		&\ \exp(\gamma|Y_\tau|) \leq \dbE_\tau \exp(M_{\tau_n})\\
		&\leq\dbE_{\tau}\exp\bigg\{\gamma\|\xi_{T}\|_{\i} + \gamma \bigg\|\int_{0}^{T} \theta_s ds\bigg\|_{\i} + (\sigma + \sigma_0)\gamma\|U\|_{\mathcal{H}^\i_{\dbF}(\tau,T+K)}\\
		&\q\q\q\q\q(T-\tau) + \lambda\gamma\int_{\tau}^{\tau_n}|V_s|^{1+\alpha}ds + \lambda_{0}\gamma \int_{\tau}^{\tau_n}\dbE_{s}[|V_{s+\zeta_s}|]ds
		\bigg\}.
	\end{align*}
	Now, note that  the pair $(U, V)$ belongs to  $\mathcal{H}^\i_{\dbF}(0,T+K;\dbR^{n}) \times \mathcal{Z}^2_{\dbF}(0,T+K;\dbR^{n\times d})$, $Y \in \mathcal{H}^\i_{\dbF}(0, T+K)$. Then, by  the dominated convergence theorem, we have that for any $\tau \in \mathscr{T}[0, T]$,
	\begin{align*}
		\exp(\gamma|Y_\tau|) \leq& \dbE_\tau \Big\{\lim\limits_{n\rightarrow \i}\exp(M_{\tau_n})\Big\}\\ \leq&\dbE_{\tau}\exp\bigg\{\gamma\|\xi_{T}\|_{\i} + \gamma \bigg\|\int_{0}^{T} \theta_s ds\bigg\|_{\i} + (\sigma + \sigma_0)\gamma\|U\|_{\mathcal{H}^\i_{\dbF}(\tau,T+K)}(T-\tau) \\
		&\qq\qq + \lambda\gamma\lim\limits_{n\rightarrow \i}\int_{\tau}^{\tau_n}|V_s|^{1+\alpha}ds + \lambda_{0}\gamma\lim\limits_{n\rightarrow \i} \int_{\tau}^{\tau_n}\dbE_{s}[|V_{s+\zeta_s}|]ds
		\bigg\},
	\end{align*}
	which implies that (\ref{eq:6.7}) holds.
\end{proof}
\begin{proof}[\bf{Proof of \autoref{lemma:6.2-2}}]
	We first prove the case of (\ref{eq:6.8}).
	Since $(Y, Z)$ is a solution of BSDE (\ref{241211}) and (\ref{eq:6.8}) holds, we have for each $n \geq 1$ and for any fixed $\tau\in \mathscr{T}[0, T]$,
	\begin{align*}
		\frac{\gamma}{2}\int_{\tau}^{\tilde{\tau}_n} |Z_s|^2ds \leq& Y_\tau - Y_{\tilde{\tau}_n} + \int_{\tau}^{T} (\theta_{s} + \sigma|U_s| + \lambda |V_s|^{1+\alpha} + \sigma_0\dbE_s[|U_{s+\delta_s}|] +\lambda_{0}\dbE_{s}[|V_{s+\zeta_s}|])ds\\
		&+\int_{\tau}^{\tilde{\tau}_n}Z_sdW_s\leq N + \int_{\tau}^{\tilde{\tau}_n}Z_sdW_s, 
	\end{align*}
	where $$N \triangleq 2\mathop{\sup}\limits_{s\in [t, T]}|Y_s| +  \int_{\tau}^{T} (\theta_{s} + \sigma|U_s| + \lambda |V_s|^{1+\alpha} + \sigma_0\dbE_t[|U_{s+\delta_s}|] +\lambda_{0}\dbE_{s}[|V_{s+\zeta_s}|])ds$$
	and
	\begin{align*}
		\tilde{\tau}_n \triangleq \inf\Big\{s\in [\tau, T] : \int_{0}^{s}|Z_r|^2dr \geq n\Big\} \wedge T.
	\end{align*}
	Then, for each $\varepsilon > 0$, we have that for any $\tau\in \mathscr{T}[0, T]$,
	\begin{align}\label{eq:6.13}
		\exp \Big( \frac{\gamma}{2}\varepsilon \int_{\tau}^{\tilde{\tau}_n}|Z_s|^2ds\Big)
		\leq \exp(\varepsilon N)\exp\Big(\varepsilon\int_{\tau}^{\tilde{\tau}_n}Z_sdW_s - \frac{3}{2}\varepsilon^2\int_{\tau}^{\tilde{\tau}_n}|Z_s|^2ds\Big)\exp\Big(\frac{3}{2}\varepsilon^2\int_{\tau}^{\tilde{\tau}_n}|Z_s|^2ds\Big).
	\end{align}
	By the definition of $\tilde{\tau}_n$ and Novikov’s condition, 
	we know the process
	\begin{align*}
		G_1(t) \triangleq \exp\Big(3\varepsilon\int_{0}^{t}Z_sdW_s - \frac{9}{2}\varepsilon^2\int_{0}^{t}|Z_s|^2ds\Big), \q t\in [\tau, \tilde{\tau}_n]
	\end{align*}
	is a positive martingale.
	Then by Optional Sampling theorem, we have
	\begin{align}\label{eqq::6.3}
		&\dbE_{\tau} \Big[\exp\Big(3\varepsilon\int_{\tau}^{\tilde{\tau}_n}Z_sdW_s - \frac{9}{2}\varepsilon^2\int_{t}^{\tilde{\tau}_n}|Z_s|^2ds\Big)\Big] \nonumber\\
		=&\frac{\dbE_{\tau} \Big[\exp\Big(3\varepsilon\int_{0}^{\tilde{\tau}_n}Z_sdW_s - \frac{9}{2}\varepsilon^2\int_{0}^{\tilde{\tau}_n}|Z_s|^2ds\Big)\Big]}{\exp\Big(3\varepsilon\int_{0}^{\tau}Z_sdW_s - \frac{9}{2}\varepsilon^2\int_{0}^{t}|Z_s|^2ds\Big)} = 1.
	\end{align}
	Taking conditional expectation $\dbE_\tau$ on both sides of (\ref{eq:6.13}), by H\"{o}lder's inequality, and (\ref{eqq::6.3}), we obtain
	\begin{align*}
		\dbE_\tau\Big[\exp \Big( \frac{\gamma}{2}\varepsilon \int_{\tau}^{\tilde{\tau}_n}|Z_s|^2ds\Big)\Big] \leq \Big(\dbE_\tau\Big[\exp(3\varepsilon N)\Big]\Big)^{\frac{1}{3}} \Big(\dbE_\tau\Big[\exp\Big(\frac{9}{2}\varepsilon^2\int_{\tau}^{\tilde{\tau}_n}|Z_s|^2ds\Big)\Big]\Big)^{\frac{1}{3}}.
	\end{align*}
	Consequently, for any $\varepsilon \leq \frac{\gamma}{9}$ and $\tau \in \mathscr{T}[0, T]$, we have
	\begin{align*}
		\Big(\dbE_\tau\Big[\exp \Big( \frac{\gamma}{2}\varepsilon \int_{\tau}^{\tilde{\tau}_n}|Z_s|^2ds\Big)\Big]\Big)^{\frac{2}{3}} \leq \Big(\dbE_\tau\Big[\exp(3\varepsilon N)\Big]\Big)^{\frac{1}{3}},
	\end{align*}
	which yields (\ref{eq:6.10}) from Fatou's lemma easily.
	For the case of (\ref{eq:6.9}). Proceding the above and since 
	\begin{equation*}
		G_2(t) \triangleq \exp\Big(-3\varepsilon\int_{0}^{t}Z_sdW_s - \frac{9}{2}\varepsilon^2\int_{0}^{t}|Z_s|^2ds\Big), \q t\in [\tau, \tilde{\tau}_n]
	\end{equation*}
	is also positive martingale, we obtain that (\ref{eq:6.10}) is also satisfied.
\end{proof}

%	\bibliography{/Users/jqwen/Ref/ref}  %%ÎÄÏ×¿â

\begin{thebibliography}{90}
	\addtolength{\itemsep}{-1.0ex}
	
%	\bibitem{Bismut1973optimal}
%	\rm J.M. Bismut, 
%	\it Conjugate convex functions in optimal stochastic control, 
%	\sl J. Math. Anal. Appl.
%	\rm 44 (2) (1973) 384--404.
	\bibitem{Bahali-Eddahbi-Ouknine-17}
	\rm K. Bahlali, M. Eddahbi, and Y. Ouknine,
	\it Quadratic BSDE with $L^2$-terminal data: Krylov's estimate, It\^{o}-Krylov's formula and existence results, 
	\sl Ann. Probab.,
	\rm 45 (2017), 2377--2397.
	
	\bibitem{bass1994probabilistic}
	\rm R.F. Bass,
	\sl Probabilistic techniques in analysis,
	\rm Springer, 
	 1995.
	
	\bibitem{Briand-Karoui-13}
	\rm P. Barrieu and N. El Karoui,
	\it Monotone stability of quadratic semimartingales with applications to unbounded general quadratic BSDEs, 
	\sl Ann. Probab.,
	\rm 41 (2013), 1831--1863.
	
	
	\bibitem{Briand2006quadratic}
	\rm P. Briand and Y. Hu, 
	\it BSDE with quadratic growth and unbounded terminal value, 
	\sl Probab. Theory Relat. Fields,
	\rm 136 (4) (2006), 604--618.
	
	
	\bibitem{Briand and Hu}
	\rm P. Briand and Y. Hu,
	\it Quadratic BSDEs with convex generators and unbounded terminal conditions,
	\sl Probab. Theory Relat. Fields, 
	\rm 141 (3) (2008), 543--567.
	
	
	
	
%	 \bibitem{buckdahn1998hedging}
%	 \rm R. Buckdahn and Y. Hu, 
%	 \it Hedging contingent claims for a large investor in an incomplete market, 
%	 \sl Adv. Appl. Probab.,
%	 \rm 30 (1998), 239--255.
	
	
%	\bibitem{Buckdahn2024meanfield}
%	\rm R. Buckdahn, J. Li , Y. Li, et al.,
%	\it A Global Stochastic Maximum Principle for Mean-Field Forward-Backward Stochastic Control Systems with Quadratic Generators,
%	\sl arXiv preprint arXiv:
%	\rm 2404.06826, 2024.
	
	\bibitem{Chen2010delaySDDE}
	\rm L. Chen and Z. Wu,
	\it Maximum principle for the stochastic optimal control problem with delay and application, 
	\sl Automatica,
	\rm 46 (2010), 1074--1080.
	
%	\bibitem{chen2011quadratic}
%	\rm L. Chen, Z. Wu,
%	\it The quadratic problem for stochastic linear control systems with delay, 
%	\sl IEEE: Proceedings of the 30th Chinese Control Conference
%	\rm 2011, 1344--1349.
	
	\bibitem{chen2012delayed}
	\rm L. Chen, Z. Wu, and Z. Yu,
	\it Delayed stochastic linear-quadratic control problem and related applications, 
	\sl J. Appl. Math.,
	\rm 2012 (1) (2012), 835319.

\bibitem{Cheridito-Nam-14}
\rm P. Cheridito and K. Nam,
\it BSDEs with terminal conditions that have bounded Malliavin derivative, 
\sl J. Funct. Anal.,
\rm 266 (3) (2014), 1257--1285.

\bibitem{Cheridito-Nam-15}
\rm P. Cheridito and K. Nam,
\it Multidimensional quadratic and subquadratic BSDEs with special structure, 
\sl Stochastics,
\rm 87 (5) (2015), 871--884.

\bibitem{Cheridito-Nam-18}
\rm P. Cheridito and K. Nam,
\it BSE's, BSDE's and fixed-point problems, 
\sl Ann. Probab.,
\rm 45 (2017), 3795--3828.



\bibitem{deFeoSwiech2025}
F. de Feo and A. \'{S}wiech,
\it  Optimal control of stochastic delay differential equations: optimal feedback controls,
\sl  J. Differential Equations,
\rm 420 (2025), 450--508.


	
	\bibitem{Douissi2019fractional}
	\rm S. Douissi, J. Wen, and Y. Shi,
	\it Mean-ﬁeld anticipated BSDEs driven by fractional Brownian motion and related stochastic
	control problem, 
	\sl Appl. Math. Comput.,
	\rm 355 (2019), 282--298.
	
	\bibitem{Fan S}
	\rm S. Fan, Y. Hu, and S. Tang,
	\it On the uniqueness of solutions to quadratic BSDEs with non-convex generators and unbounded terminal conditions,
	\sl C.R. Math. Acad. Sci. Paris,
	\rm 358 (2) (2020), 227--235.
	
	\bibitem{fan2023multi}
	\rm S. Fan, Y. Hu, and S. Tang,
	\it Multi-dimensional backward stochastic differential equations of diagonally quadratic generators: the general result,
	\sl J. Differential Equations,
	\rm 368 (2023), 105--140.
	
	\bibitem{Frei}
	\rm C. Frei and G. Dos Reis,
	\it A ﬁnancial market with interacting investors: does an equilibrium exist?, 
	\sl Math. Financ. Econ.,
	\rm 4 (3) (2011), 161--182.
	   
	
	
	\bibitem{Fujii2018anticipated}
	\rm M. Fujii and A. Takahashi,
	\it Anticipated backward SDEs with jumps and quadratic–exponential growth drivers, 
	\sl Stoch. Dyn.,
	\rm 19 (3) (2019), 1950020.
	
	\bibitem{GM2021MCRF}
	\rm G. Giuseppina and M. Federica,
	\it Federica Stochastic maximum principle for problems with delay with dependence on the past through general measures, 
	\sl Math. Control Relat. Fields,
	\rm 11 (4) (2021), 829--855.
	
	\bibitem{Hamaguchi2023}
	\rm Y. Hamaguchi,
	\it On the maximum principle for optimal control problems of stochastic Volterra integral equations with delay, 
	\sl Appl. Math. Optim.,
	\rm 87 (3) (2023), 42.
	
	
	\bibitem{han2024stochastic}
	\rm Y. Han and Y. Li,
	\it Stochastic maximum principle for control systems with time-varying delay, 
	\sl Syst. Control Lett.,
	\rm 191 (2024), 105864.
	
	
	\bibitem{HuChen2016}
	\rm F. Hu and Z. Chen,
	\it $L^p$ solutions of anticipated backward stochastic differential equations under monotonicity and general increasing conditions, 
	\sl Stochastics,
	\rm 88 (2) (2016), 267--284.
	
	\bibitem{hu2021anticipated}
	\rm Y. Hu, X. Li, and J. Wen, 
	\it Anticipated backward stochastic differential equations with quadratic growth,
	\sl J. Differential Equations,
	\rm 270 (2021), 1298--1331.
	
%	\bibitem{hu2018quadratic}
%	\rm Y. Hu, Y. Lin, A. Hima,
%	\it Quadratic backward stochastic differential equations driven by G-Brownian motion: Discrete solutions and approximation, 
%	\sl Stoch. Process. Appl.
%	\rm 128 (11) (2018) 3724--3750.
%	
%	
	\bibitem{Hu-Tang-16}
	\rm Y. Hu and S. Tang,
	\it Multi-dimensional backward stochastic diﬀerential equations of diagonally
	quadratic generators, 
	\sl Stoch. Process. Appl.,
	\rm 126 (4) (2016), 1066--1086.
	
	
	
%	\bibitem{hu2022quadratic}
%	\rm Y. Hu, S. Tang, F. Wang,
%	\it Quadratic G-BSDEs with convex generators and unbounded terminal conditions, 
%	\sl Stoch. Process. Appl.
%	\rm 153  (2022) 363--390.
	
	
	
	
	\bibitem{HuLiangTang2020}
	Y. Hu, G. Liang, and S. Tang,
	\it Systems of {Ergodic} {BSDEs} {Arising} in {Regime} {Switching} {Forward} {Performance} {Processes},
	\sl SIAM J. Cont.  Optim.,
	\rm 58 (2020), 2503--2534.
	
	
	\bibitem{HuLiangTang2024}
	Y. Hu, G. Liang, and S. Tang,
	\sl Utility maximization in constrained and unbounded financial markets: {Applications} to indifference valuation, regime switching, consumption and {Epstein}-{Zin} recursive utility,
	\rm arXiv:1707.00199v5, 2024.
	
%\bibitem{Huang-Shi-12}
%J.~Huang and J.~Shi,
%\it Maximum principle for optimal control of fully coupled forward-backward stochastic differential delayed equations,
%\sl ESAIM: COCV.,
%\rm 18(4) (2012) 1073--1096.

	
	\bibitem{LeGall}
	\rm J. Le Gall,
	\it Brownian motion, martingales, and stochastic calculus, 
	\sl Springer, Switzerland,
	\rm 2016.
	
	\bibitem{Juan2024mean-field}
	\rm W. Jiang, J. Li, and Q. Wei,
	\it General mean-field BSDEs with diagonally quadratic generator in multi-dimension,
	\sl Discrete Contin. Dyn. Syst.,
	\rm 44 (10) (2024), 2957--2984.
	
	\bibitem{Kazamaki N}
	\rm N. Kazamaki,
	\sl Continuous Exponential Martingales and BMO,
	\rm Springer, Berlin,
	\rm  1994.
	
	\bibitem{KR2023SPA}
	\rm T. Klimsiak and M. Rzymowski,
	\it Nonlinear BSDEs on a general filtration with drivers depending on the martingale part of the solution, 
	\sl Stochastic Process. Appl.,
	\rm 161 (2023), 424--450.
	
	 \bibitem{Kobylanski2000PDE}
	\rm M. Kobylanski, 
	\it Backward stochastic differential equations and partial differential equations with quadratic growth, 
	\sl Ann. Probab.,
	\rm 28 (2) (2000), 558--602.
	
	
	\bibitem{lu2013anticipated}
	\rm W. Lu and Y. Ren,
	\it Anticipated backward stochastic differential equations on Markov chains, 
	\sl Stat. Probab. Lett.,
	\rm 83(7) (2013), 1711--1719.
	
	
	
%	\bibitem{luo2018bounded}
%	\rm H. Luo and S. Fan,
%	\it Bounded solutions for general time interval BSDEs with quadratic growth coefficients and stochastic conditions, 
%	\sl Stochastics and Dynamics
%	\rm 18(5) (2018) 1850034.
	
%	\bibitem{luo2021comparison}
%	\rm P. Luo,
%	\it Comparison theorem for diagonally quadratic BSDEs, 
%	\sl Discrete Contin. Dyn. Syst.
%	\rm 41(6) (2021).
	
	\bibitem{madoui2022quadratic}
	\rm I. Madoui, M. Eddahbi, and N. Khelfallah
	\it Quadratic BSDEs with jumps and related PIDEs, 
	\sl Stochastics,
	\rm 94(3) (2022), 386--414.
	
	
	
	
	\bibitem{MengShiWangZhang2025}
	W. Meng, J. Shi, T. Wang, and J. Zhang,
	\it A general maximum principle for optimal control of stochastic differential delay systems,
	\sl SIAM J. Control Optim.,
	\rm 63 (2025),  175--205.
	
	
	
	\bibitem{Menoukeu-Pamen2015}
	\rm O. Menoukeu-Pamen,
	\it Non-linear time-advanced backward stochastic partial differential equations with jumps, 
	\sl Stoch. Anal. Appl.,
	\rm 33 (4) (2015), 673--700.
	
	\bibitem{NieWangWu2024}
	\rm T. Nie, S. Wang, and Z. Wu,
	\it Linear-quadratic delayed mean-field social optimization, 
	\sl Appl. Math. Optim.,
	\rm 89 (1) (2024), 4.
	
	\bibitem{pak2024wellposedness}
	\rm J. Pak, M.  Kim, and K. Kim,
	\it Wellposedness of anticipated BSDEs with quadratic growth and unbounded terminal value, 
	\sl Braz. J. Probab. Stat.,
	\rm 38(1) (2024), 108--127.
	
	\bibitem{Pardoux1990BSDE}
	\rm E. Pardoux and S. Peng, 
	\it Adapted solution of a backward stochastic differential equation, 
	\sl Syst. Control Lett.,
	\rm 4 (1990), 55--61.
	
	
	
	\bibitem{peng1999open}
	\rm S. Peng, 
	\it Open problems of backward stochastic differential equations, 
	\sl in: S. Chen, et al. (Eds.), Control of 
	Distributed Parameter and Stochastic Systems, Hangzhou, 1998, Kluwer Academic Publishers, Boston,
	\rm 1999, pp. 265--273.
	
	
	\bibitem{Peng2009anticipated}
	\rm S. Peng and Z. Yang,
	\it Anticipated backward stochastic differential equations, 
	\sl Ann. Probab.,
	\rm 37 (2009), 877--902.
	
\bibitem{Richou-12}
\rm A. Richou,
\it Markovian quadratic and superquadratic BSDEs with an unbounded terminal condition, 
\sl Stochastic Process. Appl.,
\rm 122 (2012), 3173--3208.

    \bibitem{Tevzadze2008quadratic}
    \rm R. Tevzadze, 
    \it  Solvability of backward stochastic differential equations with quadratic growth, 
    \sl Stoch. Process. Appl.,
    \rm 118 (3) (2008), 503--515.
%    
%    \bibitem{wang2023risk}
%    \rm P. Wang, 
%    \it Risk-sensitive maximum principle for controlled system with delay, 
%    \sl Mathematics
%    \rm 11 (4) (2023) 1058.
    
  
    
%    \bibitem{wen2020LQ}
%    \rm J. Wen, X. Li, and J. Xiong,
%    \it Weak closed-loop solvability of stochastic linear quadratic optimal control problems of Markovian regime switching system, 
%    \sl Appl. Math. Optim.,
%    \rm 84 (2021),  535--565.
    
    \bibitem{wen2017delayed}
    \rm J. Wen and Y. Shi,
    \it Anticipative backward stochastic differential equations driven by fractional Brownian motion, 
    \sl Stat. Probab. Lett.,
    \rm 122 (2017), 118--127.
    
     \bibitem{Wu2012nonlipschitz}
    \rm H. Wu, W. Wang, and J. Ren,
    \it Anticipated backward stochastic differential equations with non-Lipschitz coefﬁcients, 
    \sl Stat. Probab. Lett.,
    \rm 82 (2012), 835319.

\bibitem{Xing-Zitkovic}
\rm H. Xing and G. Zitkovic,
\it A class of globally solvable Markovian quadratic BSDE systems and applications, 
\sl Ann. Probab.,
\rm 46 (2018), 491--550.

%    \bibitem{yu2012stochastic}
%    \rm Z. Yu,
%    \it The stochastic maximum principle for optimal control problems of delay systems involving continuous and impulse controls, 
%    \sl Automatica
%    \rm 48 (10) (2012) 2420--2432.
\bibitem{YangElliott2013}
\rm Z. Yang and R. Elliott,
\it Some properties of generalized anticipated backward stochastic differential equations, 
\sl Electron. Commun. Probab.,
\rm 18 (63) (2013), 10.
    
    \bibitem{Yong1999control}
    \rm J. Yong and X. Zhou,
    \sl Stochastic Controls: Hamiltonian Systems and HJB Equations, 
    \rm Springer, New York,
    \rm 1999.
    
    \bibitem{ChenYang2024}
    \rm C. Zhang and L. Yang,
    \it Well-posedness for anticipated backward stochastic Schrödinger equations, 
    \sl Stochastics,
    \rm 96 (4) (2024), 1307--1327.
    
    \bibitem{zhang2017backward}
    \rm J. Zhang,
    \sl Backward stochastic differential equations, 
    \rm Springer,
    \rm 2017.
    
%    \bibitem{zhang2021maximum}
%    \rm Q. Zhang,
%    \it Maximum principle for stochastic optimal control problem with distributed delays, 
%    \sl Acta Math. Sci.
%    \rm 41 (2) (2021) 437--449.
%    
 




\end{thebibliography}
%	\bibliographystyle{/Users/jqwen/Ref/plainnat-doi}	%% ²Î¿¼ÎÄÏ×µÄ¸ñÊ½
%		% bbl ÎÄ¼þÎª²Î¿¼ÎÄÏ×µÄtex¸ñÊ½
%		

%------------------------------------------------------------------------------------------------

\end{document}